\documentclass[review,hidelinks,onefignum,onetabnum]{siamart250211}

\usepackage{amsmath, amsfonts}
\numberwithin{equation}{section}
\usepackage{bbm}
\usepackage{graphicx}
\usepackage{times}
\usepackage{tikz}
\usepackage{color}
\usepackage{mathrsfs}

\ifpdf
\hypersetup{
   colorlinks=true,%
            breaklinks=true,%
            linkcolor=blue,
            urlcolor=blue,%
            citecolor=blue,%
            pdftitle={Sensitivity-Driven Adaptive Surrogate Modeling for Simulation and Optimization of Dynamical Systems}, %
            pdfauthor={J. R. Cangelosi and M. Heinkenschloss},%
            bookmarksopen=false,%
}
\fi

\newtheorem{remark}[theorem]{Remark}
\newtheorem{assumption}[theorem]{Assumption}

\newcommand{\set}[2]{\left\{ #1 \;:\; #2 \right\}}

\newcommand{\aall}{\mathrm{a.a.} \;}
\newcommand {\nat} {\mathbb{N}}
\newcommand {\real} {\mathbb{R}}
\newcommand{\WW}{\big( W^{1, \infty}(I) \big)}
\newcommand{\LL}{\big(L^\infty(I) \big)}

\newcommand{\BA}{\ensuremath{\mathbf{A}} } %
\newcommand{\BB}{\ensuremath{\mathbf{B}} } %
\newcommand{\BC}{\ensuremath{\mathbf{C}} } %
\newcommand{\BF}{\ensuremath{\mathbf{F}} } %
\newcommand{\BH}{\ensuremath{\mathbf{H}} } %
\newcommand{\BL}{\ensuremath{\mathbf{L}} } %
\newcommand{\BQ}{\ensuremath{\mathbf{Q}} } %
\newcommand{\BR}{\ensuremath{\mathbf{R}} } %

\newcommand{\bc}{\ensuremath{\mathbf{c}}} %
\newcommand{\bd}{\ensuremath{\mathbf{d}}} %
\newcommand{\bff}{\ensuremath{\mathbf{f}}} %
\newcommand{\bg}{\ensuremath{\mathbf{g}}} %
\newcommand{\bh}{\ensuremath{\mathbf{h}}} %
\newcommand{\bk}{\ensuremath{\mathbf{k}}} %
\newcommand{\bp}{\ensuremath{\mathbf{p}}} %
\newcommand{\bq}{\ensuremath{\mathbf{q}}} %
\newcommand{\br}{\ensuremath{\mathbf{r}}} %
\newcommand{\bu}{\ensuremath{\mathbf{u}}} %
\newcommand{\bv}{\ensuremath{\mathbf{v}}} %
\newcommand{\bx}{\ensuremath{\mathbf{x}}} %
\newcommand{\by}{\ensuremath{\mathbf{y}}} %
\newcommand{\bz}{\ensuremath{\mathbf{z}}} %

\newcommand{\balpha}{\ensuremath{\mbox{\boldmath $\alpha$}}}
\newcommand {\bdelta} {\boldsymbol{\delta}}
\newcommand {\bepsilon} {\mbox{\boldmath $\epsilon$}}
\newcommand {\bgamma} {\mbox{\boldmath $\gamma$}}
\newcommand {\blambda} {\mbox{\boldmath $\lambda$}}
\newcommand {\BPhi} {\mbox{\boldmath $\Phi$}}
\newcommand {\bphi} {\mbox{\boldmath $\phi$}}
\newcommand {\bpsi} {\boldsymbol{\psi}}

\newcommand {\cB} {\mathcal{B}}
\newcommand {\cG} {\mathcal{G}}
\newcommand {\cH} {\mathcal{H}}
\newcommand {\cN} {\mathcal{N}}
\newcommand {\cV} {\mathcal{V}}
\newcommand {\cY} {\mathcal{Y}}
\newcommand {\cZ} {\mathcal{Z}}

\DeclareMathOperator*{\esssup}{ess\,sup}

\headers{Sensitivity-Driven Adaptive Surrogate Modeling}{J. R. Cangelosi and M. Heinkenschloss}

\title{Sensitivity-Driven Adaptive Surrogate Modeling for Simulation and Optimization of Dynamical Systems
       \thanks{This paper has been accepted for publication in the SIAM Journal on Scientific Computing.
       \funding{This research was supported in part by AFOSR Grant FA9550-22-1-0004 and NSF Grant DMS-2231482.}
       }}
         
\author{Jonathan R. Cangelosi
              \thanks{Department of Aeronautics and Astronautics,
             Stanford University, 
             496 Lomita Mall, Stanford, CA 94305
                     E-mail: jcange@stanford.edu}
             \and
             Matthias Heinkenschloss
             \thanks{Department of Computational Applied Mathematics and Operations Research,
                 MS-134, Rice University, 6100 Main Street,
                 Houston, TX 77005-1892, and the Ken Kennedy Institute, Rice University.
                 E-mail: heinken@rice.edu}
        }

\ifpdf
\hypersetup{
   colorlinks=true,%
            breaklinks=true,%
            linkcolor=blue,
            urlcolor=blue,%
            citecolor=blue,%
            pdftitle={Sensitivity-Driven Adaptive Surrogate Modeling for Simulation and Optimization of Dynamical Systems}, %
            pdfauthor={J. R. Cangelosi and M. Heinkenschloss},%
            bookmarksopen=false,%
}
\fi

\begin{document}

\maketitle

\begin{abstract}
This paper develops a surrogate model refinement approach for the simulation of dynamical systems and the solution of optimization problems governed by dynamical systems in which surrogates replace expensive-to-compute state- and control-dependent component functions in the dynamics or objective function. For example, trajectory simulation and optimization tasks for an aircraft depend on aerodynamic coefficient functions whose evaluation requires expensive computational fluid dynamics simulations for every value of the state and control encountered in simulation and optimization algorithms, which would result in prohibitively long run times, often exacerbated further by the lack of derivative information. To overcome this bottleneck, this work employs differentiable surrogates that are computed from values of the true component functions at a few points. The proposed approach updates the current surrogates on an as-needed basis as follows: given a surrogate and corresponding solution of the simulation or optimization problem, the approach combines solution sensitivity information with pointwise error estimates between the true component functions and their surrogates to define an acquisition function that is used to determine new points at which to evaluate the true component function to refine the surrogate. The performance of the proposed approach is demonstrated on a numerical example of a notional hypersonic vehicle with aerodynamic coefficient models that are approximated using kernel interpolation.
\end{abstract}


\medskip
                 
%%%%%%%%%%%%%%%%%%%%%%%%%%%%%%%%%%%%%%%%%%%%%%%%%%%%%%%%%        
%\input{intro}
\section{Introduction} \label{sec:intro}
We present a new approach for surrogate model refinement for the simulation of dynamical systems and the solution of optimization problems governed by dynamical systems 
in which computationally inexpensive surrogate models replace computationally expensive state- and control-dependent component functions in the dynamics or in the objective function. 
This approach is useful for simulation and optimization problems where the dynamics or the objective function depend on expensive-to-evaluate component functions.
For example, trajectory simulation and optimization tasks for an aircraft in longitudinal flight depend on aerodynamic coefficient functions for lift, drag, and moment about the pitch axis,
which in turn depend on altitude, velocity, angle of attack, and possibly other state variables. 
Naively attempting to solve the simulation or optimization problem directly using 
these coefficient functions requires expensive computational fluid dynamics simulations for every value of the state and control encountered in simulation and optimization algorithms, resulting in 
prohibitively long run times of these algorithms, often exacerbated further by the lack of derivative information for the coefficients. 
In other examples, the dynamics may depend on state-dependent constitutive laws that can be determined experimentally at select points, making the experimental design crucial to strong numerical performance.
To solve these problems efficiently and accurately, we approximate the expensive-to-evaluate component functions by surrogates, use sensitivity information and surrogate error bounds to assess the impact of the error between the true component functions and their surrogates on the computed solution or on
a solution-dependent quantity of interest (QoI), then use this information to systematically improve the surrogate and reduce the error in the computed solution or QoI.

Our approach builds on the sensitivity results for simulation and optimization of dynamical systems in \cite{JRCangelosi_MHeinkenschloss_2025c}, \cite{JRCangelosi_MHeinkenschloss_2025a} using standard error estimates for surrogate models constructed via kernel interpolation. As demonstrated in  \cite{JRCangelosi_MHeinkenschloss_2025c}, \cite{JRCangelosi_MHeinkenschloss_2025a}, the
sensitivity of the (in the optimization context, local) solution of the simulation problem or of the optimization problem with respect to the component function
can be used to construct useful bounds for the error between the solution with the true component function and the solution computed with the surrogates.
These bounds require pointwise bounds for the error between the true component function and its surrogates along the computed solution.
Surrogates computed via kernel interpolation admit such bounds; see, e.g., \cite[Ch.~8]{AIske_2018b}, \cite{GSantin_BHaasdonk_2021a},
\cite{HWendland_2004a}, and  \cref{sec:surrogates}.
We note that kernel interpolation is closely related to Gaussian process (GP) regression
\cite{ABerlinet_CThomas-Agnan_2004a}, \cite{CERasmussen_CKIWilliams_2006a}.

In this paper, we will integrate sensitivity-based error estimates and pointwise surrogate error bounds along the computed solution to assess the quality of the current surrogate. 
If the surrogate needs to be refined, we want to determine a new point (or points) at which to evaluate the true component function to construct a refined surrogate.
A key contribution in this work is the construction of a new, goal-oriented, computationally efficient acquisition function that, 
for a  given trial point, estimates the error between the value of the QoI corresponding to the true component function and 
the QoI computed with the surrogate were it to be refined at the trial point, which is constructed by incorporating the 
information from the value of the true component function at the trial point through the pointwise error bound for 
the refined surrogate. 
Crucially, the error bound for the refined surrogate must be independent of the actual value of the true component 
function at the trial point, ensuring that the acquisition function does not require the evaluation of the true, 
expensive component function and thus can be computed inexpensively at many trial points. 
Kernel interpolants have such error bounds, and we will review the relevant results.
However,  we formulate our model refinement approach in a way that is agnostic about the type of surrogate used, as other surrogate modeling methods may also admit error bounds with the required properties for use in our approach.
We apply our model refinement procedure to trajectory simulation and optimization problems 
for a notional hypersonic vehicle with expensive-to-evaluate aerodynamic coefficients, where we integrate our
acquisition function, error bounds for kernel interpolants, and selection of new, trajectory-dependent samples for
surrogate refinement. Our procedure outperforms other refinement approaches based on random sampling or
maximum error of the surrogate model alone.

The need to use data-driven surrogate models in the simulation and optimization of dynamical systems, particularly when certain components of the 
dynamics are expensive to compute or are only known experimentally, is an old problem; see, e.g., \cite{AEBryson_MNDesai_WCHoffman_1969a},
\cite[Sec.~6.2]{JTBetts_2010a}. 
In applications, the construction and improvement of surrogates through the experimental design and systematic acquisition of additional data is important to ensure strong performance at low computational costs.
In the context of a vehicle co-design and trajectory optimization problem similar to the one considered in \cref{sec:numerics_OCP},
\cite{JRCangelosi_MHeinkenschloss_JTNeedels_JJAlonso_2024a} updated surrogates by selecting samples along the
current optimal trajectory using a maximum variance-based adaptive sampling approach similar to the maximum error bound (MEB) approach considered in \cref{sec:numerics}. Our results
in \cref{sec:numerics} show that our sensitivity-based data acquisition strategy outperforms this approach.
The paper  \cite{JTNeedels_JJAlonso_2024a} incorporates adjoint-based sensitivity information into an adaptive sampling approach; however, their sensitivities
do not fully incorporate state- and control-dependent coefficient functions as done in \cref{sec:sensitivity}, and their sampling does not use the predicted 
change in QoI error due to adding a sample point, as proposed in \cref{sec:model_refinement}.
The paper \cite{JHart_BvanBloemenWaanders_LHood_JParish_2023a} uses a high-dimensional parametrization to capture model discrepancies and
standard parametric solution sensitivities within a design-of-experiments framework to update the model. Our work extends their 
approach to a parameter-free, meshfree setting, which reduces computational costs associated with computing sensitivities by removing the need to parametrize surrogates and their corresponding error bounds. Our approach also obtains error bounds systematically using kernel methods rather than employing \emph{ad hoc} error reduction estimates.
Bayesian optimization approaches often employ GP/kernel-based surrogates, as seen in e.g.,  \cite{DRJones_MSchonlau_WJWelch_1998a}, \cite{PIFrazier_2018a}, \cite{RGarnett_2023a},
but they typically approximate the entire objective and constraint functions, resulting in inability to preserve the underlying system dynamics exactly. By contrast, in this paper we aim to adaptively model component functions appearing in the dynamics while preserving the structure of the dynamical system itself.
Global sensitivities and derivative information are combined in \cite{MVohra_AAlexanderian_CSafta_SMahadevan_2019a} to 
reduce effective input dimension and reduce sampling cost; however, the setting is different from ours. In particular, we are not interested in
computing component function surrogates that are good approximations everywhere, but only where their errors impact the computed solution of the simulation or
optimization problem. This reduces the need for expensive high-fidelity computations in settings where one only aims for the solution of a single simulation or optimization problem, or small perturbations thereof.

An outline of the paper is as follows: 
\Cref{sec:problem_formulation} gives formulations for simulation and optimization problems of the form we want to study. \Cref{sec:sensitivity} 
discusses sensitivity analysis of solutions with respect to a model function. 
\Cref{sec:surrogates} reviews error bounds for kernel interpolation-based surrogates. 
\Cref{sec:model_refinement} combines sensitivity results from \cref{sec:sensitivity} with the findings in \cref{sec:surrogates} to develop an adaptive refinement procedure  for surrogates and is the main contribution of our paper. 
\Cref{sec:numerics} demonstrates the model refinement procedure on trajectory simulation and optimization problems for a notional hypersonic vehicle 
with expensive-to-evaluate aerodynamic coefficients. 
\Cref{sec:conclusion} summarizes main takeaways and directions for future work.

{\bf Notation.} 
We will use $\| \cdot \|$ to denote a vector norm on $\real^m$ (where $m$ depends on the context) or an 
induced matrix norm.
By $\mathcal{B}_R(0) \subset \real^m$ we denote the closed ball in $\real^m$ around zero with radius $R>0$.
When infinite-dimensional normed linear spaces are considered, the norm will always be specified explicitly 
using subscripts.

Given an interval $I := (t_0, t_f)$,  $\big( L^\infty(I) \big)^m$ denotes the Lebesgue space of essentially bounded functions on $I$ with
values in $\real^m$, and $\big( W^{1, \infty}(I) \big)^m$ denotes the Sobolev space of functions on $I$ with 
values in $\real^m$ that are weakly differentiable on $I$ and have essentially bounded derivative.

We typically use bold font for vector- or matrix-valued functions and regular font for scalars, vectors, and matrices.
For example, the function $\bx: I \rightarrow \real^{n_x}$ has values $\bx(t) \in \real^{n_x}$, and $x \in \real^{n_x}$ denotes a vector. This distinction will be useful when
studying compositions of functions. Also, when using subscripts for derivatives, regular subscripts will be used to denote partial derivatives with respect to an argument, while boldface subscripts will be used to denote Fr\'echet derivatives with respect to a function.

%%%%%%%%%%%%%%%%%%%%%%%%%%%%%%%%%%%%%%%%%%%%%%%%%%%%%%%%%%      
%%%%%%%%%%%%%%%%%%%%%%%%%%%%%%%%%%%%%%%%%%%%%%%%%%%%%%%%%%
\section{Problem Formulation} \label{sec:problem_formulation}

We consider adaptive surrogate model refinement in the two problem settings of simulation and optimization. 
In \cref{sec:problem_formulation_simulation,sec:problem_formulation_optimization},
we first state the prototypical problems in these two contexts in a suitable function space setting. The precise smoothness assumptions 
on the problem that are invoked to ensure existence and (in the optimization context, local) uniqueness of solutions in this setting 
will be stated in \cref{sec:sensitivity}.

%%%%%%%%%%%%%%%%%%%%%%%%%%%%%%%%%%%%%%%%%%%%%%%%%%%%%%%%%%
\subsection{Simulation Problem} \label{sec:problem_formulation_simulation}
Let $I := (t_0, t_f)$ be an interval. Given functions
\[
	\bg : I \times \real^{n_x}  \rightarrow \real^{n_g}, \qquad
	\bff : I \times  \real^{n_x} \times \real^{n_g} \rightarrow \real^{n_x} 
\]
and $x_0 \in \real^{n_x}$, we seek the solution $\bx_* \in \WW^{n_x}$
of the initial value problem (IVP)
\begin{equation} \label{eq:problem_formulation:IVP_true}
	\bx_*'(t) = \bff\Big( t, \bx_*(t), \bg_*\big( t, \bx_*(t) \big) \Big) \quad \mbox{for almost all (a.a.) } t \in I, \qquad
	\bx(t_0) = x_0.
\end{equation}
The solution $\bx_*$ of \eqref{eq:problem_formulation:IVP_true} is also referred to as the state.
The function $\bff$ represents the dynamics of the system, which depend on a component function $\bg_*$ that is expensive to evaluate. 
For example, in flight simulation, $\bg_*(t, x)$ may represent the aerodynamic coefficients of an aircraft at the time $t \in I$ and flight state $x \in \real^{n_x}$ computed by expensive CFD simulations. Thus, instead of solving \eqref{eq:problem_formulation:IVP_true} directly using the expensive $\bg_*$, we solve
\begin{equation} \label{eq:problem_formulation:IVP}
	\bx'(t) = \bff\Big( t, \bx(t), \bg\big( t, \bx(t) \big) \Big), \quad \aall t \in I, \qquad
	\bx(t_0) = x_0,
\end{equation}
where $\bg \approx \bg_*$ is a sample-based surrogate model for the component function. If this approximation is good (in some appropriate sense), then the resulting state $\bx$ approximates $\bx_*$. The focus of our paper is the development of an adaptive, sensitivity-driven strategy to obtain surrogates $\bg \approx \bg_*$ that yield strong approximations $\bx \approx \bx_*$ with a small number of expensive $\bg_*$ queries.

For given functions 
$\phi : \real^{n_x} \rightarrow \real$ and  $\ell : I \times  \real^{n_x} \times \real^{n_g} \rightarrow \real$,
we are also interested in computing a quantity of interest
\begin{subequations}    \label{eq:problem_formulation:QoI-ODE}
\begin{equation} \label{eq:problem_formulation:QoI-ODE-g-only}
	\widetilde{q}(\bg) := q\big(\bx(\cdot \, ; \bg), \bg \big),
\end{equation}
where $\bx(\cdot \, ; \bg)$ is the solution of \eqref{eq:problem_formulation:IVP} with the given function $\bg$, and
\begin{equation}     \label{eq:problem_formulation:QoI-ODE-b}
	q(\bx, \bg) := \phi\big(\bx(t_f)\big) + \int_{t_0}^{t_f} \ell \Big( t, \bx(t), \bg\big( t, \bx(t) \big) \Big) \, dt.
\end{equation}
\end{subequations}

The sensitivity results in our work are easily extensible to QoIs involving evaluations of the state at any discrete times $t_i \in \overline{I}$ through suitable modifications to the associated adjoint equations, but we consider the case of final time only for ease of presentation.

%%%%%%%%%%%%%%%%%%%%%%%%%%%%%%%%%%%%%%%%%%%%%%%%%%%%%%%%%%
\subsection{Optimization Problem}   \label{sec:problem_formulation_optimization}
Let  $I := (t_0, t_f)$. Given functions
\begin{align*}
	\varphi &: \real^{n_x} \times \real^{n_p} \rightarrow \real, &
	l &: I \times  \real^{n_x} \times \real^{n_u} \times \real^{n_p} \times \real^{n_g} \rightarrow \real, \\
	\bff &: I \times  \real^{n_x} \times \real^{n_u} \times \real^{n_p} \times \real^{n_g} \rightarrow \real^{n_x}, &
		\bg &: I \times \real^{n_x} \times \real^{n_u} \times \real^{n_p} \rightarrow \real^{n_g},
\end{align*}
and $x_0 \in \real^{n_x}$, we consider the problem of finding the optimal 
state $\bx_* \in\big( W^{1,\infty}(I) \big)^{n_x}$, control $ \bu_* \in  \big( L^\infty(I) \big)^{n_u}$, and parameter $ \bp_* \in \real^{n_p}$
for the optimal control problem
\begin{equation} \label{eq:Sensitivity:OCP:OCP_true}
\begin{aligned}
	\min_{\bx, \bu, \bp} \quad & \varphi\big(\bx(t_f), \bp \big) + \int_{t_0}^{t_f} l\Big(t, \bx(t), \bu(t), \bp, \bg_* \big( t, \bx(t), \bu(t), \bp \big) \Big) \, dt \\
	\mbox{s.t.} \quad 
	& \bx'(t) = \bff \Big(t, \bx(t), \bu(t), \bp, \, \bg_* \big(t, \bx(t), \bu(t), \bp \big) \Big), \quad \aall t \in I, \\
	& \bx(t_0) = x_0.
\end{aligned}
\end{equation}
We slightly abuse notation here by using dynamics $\bff$ and component function $\bg_*$ as in the previous section, but with additional dependencies on controls $\bu$ and parameters $\bp$.

Similarly to the approach outlined in the previous section, we opt to solve
\begin{equation} \label{eq:Sensitivity:OCP:OCP}
\begin{aligned}
	\min_{\bx, \bu, \bp} \quad & \varphi\big(\bx(t_f), \bp \big) + \int_{t_0}^{t_f} l\Big(t, \bx(t), \bu(t), \bp, \bg \big( t, \bx(t), \bu(t), \bp \big) \Big) \, dt \\
	\mbox{s.t.} \quad 
	& \bx'(t) = \bff \Big(t, \bx(t), \bu(t), \bp, \, \bg \big(t, \bx(t), \bu(t), \bp \big) \Big), \quad \aall t \in I, \\
	& \bx(t_0) = x_0
\end{aligned}
\end{equation}
with an adaptive sample-based surrogate $\bg \approx \bg_*$ to obtain an approximate solution to \eqref{eq:Sensitivity:OCP:OCP_true}.
The costate $\blambda \in \WW^{n_x}$ associated with the constraints of \eqref{eq:Sensitivity:OCP:OCP} appears in the optimality conditions of \eqref{eq:Sensitivity:OCP:OCP}, which will be important for the construction of our sensitivity-based acquisition function.

Given functions
$\phi : \real^{n_x} \times \real^{n_p} \rightarrow \real$ and
$\ell : I \times  \real^{n_x} \times \real^{n_u} \times \real^{n_p} \times \real^{n_g} \rightarrow \real$,
we are also interested in a QoI
\begin{subequations}   \label{eq:problem_formulation:QoI-OCP}
\begin{equation} \label{eq:problem_formulation:QoI-OCP-g-only}
	\widetilde{q}(\bg) := q\big(\bx(\cdot \, ; \bg), \bu(\cdot \, ; \bg), \bp(\bg), \bg \big),
\end{equation}
where $\big(\bx(\cdot \, ; \bg), \bu(\cdot \, ; \bg), \bp(\bg)\big)$ solves \eqref{eq:Sensitivity:OCP:OCP} with the model $\bg$, and
\begin{equation} \label{eq:problem_formulation:QoI-OCP-b}
	q(\bx, \bu, \bp, \bg) := \phi\big(\bx(t_f), \bp\big) + \int_{t_0}^{t_f} \ell \Big(t, \bx(t), \bu(t), \bp, \, \bg \big(t, \bx(t), \bu(t), \bp \big) \Big) \, dt.
\end{equation}
\end{subequations}
%%%%%%%%%%%%%%%%%%%%%%%%%%%%%%%%%%%%%%%%%%%%%%%%%%%%%%%%%      
%%%%%%%%%%%%%%%%%%%%%%%%%%%%%%%%%%%%%%%%%%%%%%%%%%%%%%%%%%%%%
\section{Sensitivity Analysis} \label{sec:sensitivity}    \label{sec:sensitivity-ODE-summary}
Our surrogate model adaptation uses the sensitivities of the solution of \eqref{eq:problem_formulation:IVP} 
or of \eqref{eq:Sensitivity:OCP:OCP} with respect to changes in the model function $\bg$.  
The most notable difference between traditional sensitivity analysis---which considers perturbations of the solution with respect to parameters---and our sensitivity analysis
is that in our setting, the solution $\bx(\cdot \, ; \bg)$ itself enters as an argument into the component function $\bg$.
In this section we will review the  results from \cite{JRCangelosi_MHeinkenschloss_2025c} and \cite{JRCangelosi_MHeinkenschloss_2025a} 
on the computation of sensitivities of the solution of \eqref{eq:problem_formulation:IVP} or of \eqref{eq:Sensitivity:OCP:OCP} 
with respect to the model function $\bg$. 

To define the function space for $\bg$ used in the simulation and optimization setting,  we use the notation
\begin{subequations}    \label{eq:y-value-def}
\begin{align}
	y &= x \in \real^{n_y} := \real^{n_x}, &&\mbox{ (simulation)}, \\ 	
	y &= (x, u, p) \in \real^{n_y} :=  \real^{n_x} \times \real^{n_u} \times \real^{n_p}, &&\mbox{ (optimization)}.
\end{align}
\end{subequations}
The computation of the sensitivities is based on the implicit function theorem and assumes functions $\bg$ in the space
\begin{subequations} \label{eq:Sensitivity:caratheodory:g-space}
\begin{align}
	\cG^s := \big\{ \bg &: I \times \real^{n_y} \rightarrow \real^{n_g} \; : \; 
	                    \bg( t, y ) \mbox{ is $s$-times continuously partially }   \nonumber \\
	                   & \mbox{differentiable  with respect to } y \in \real^{n_y}  \mbox{ for } \aall t \in I,   \nonumber \\
                           &   \mbox{is measurable in $t$ for each $y \in \real^{n_y}$, and }  
                             \| \bg \|_{\cG^s}  < \infty \big\},
\end{align}
for $s = 2$ (simulation) and $s = 3$ (optimization), where
\begin{equation} \label{eq:Sensitivity:caratheodory:g-space-norm}
	\| \bg \|_{\cG^s} := \sum_{n=0}^s \; \esssup_{t \in I} \; \sup_{y \in \real^{n_y}} \left\| \frac{\partial^n}{\partial y^n} \bg(t, y) \right\|.
\end{equation}
\end{subequations}

In \cref{sec:sensitivity-ODE-summary} and \cref{sec:sensitivity-OCP-summary} we will summarize
the sensitivity results for the IVP solution and the OCP solution with respect to $\bg$.
The technical details and precise statements of these results will be given in
\cref{sec:sensitivity-ODE} and \cref{sec:sensitivity-OCP}. 
The technical details presented in the latter two technical subsections are important for the foundation,
but are not needed for the application of the sensitivity results in  \cref{sec:surrogates}-\cref{sec:numerics}.

%%%%%%%%%%%%%%%%%%%%%%%%%%%%%%%%%%%%%%%%%%%%%%%%%%%%%%%%%%%%%
\subsection{Summary of IVP Solution Sensitivity}   \label{sec:sensitivity-ODE-summary}
In this section we summarize the results on the sensitivity of the solution $\bx(\bg)$ of \eqref{eq:problem_formulation:IVP} and of the QoI \eqref{eq:problem_formulation:QoI-ODE}.
The precise statements of assumptions needed for these results to hold, as well as the relevant theorems for the sensitivity analysis,
will be given in \cref{sec:sensitivity-ODE}.

Let $\overline{\bg} \in\cG^2$ and let $\overline{\bx} \in \WW^{n_x}$ be the corresponding solution of \eqref{eq:problem_formulation:IVP}.
Under suitable conditions on the problem, which will be stated later in
\cref{thm:Sensitivity:ODE:sensitivity-result-ODE}, 
there exist neighborhoods $\cN(\overline{\bg}) \subset\cG^2$ 
and $\cN(\overline{\bx}) \subset \big( W^{1,\infty}(I) \big)^{n_x}$
and a unique map $\bx : \cN(\overline{\bg}) \rightarrow \cN(\overline{\bx})$ satisfying $\bx(\overline{\bg}) = \overline{\bx}$
such that  $\bx(\cdot \, ; \bg) \in \WW^{n_x}$ is the solution of the initial value problem \eqref{eq:problem_formulation:IVP} 
with $\bg \in \cN(\overline{\bx})$.
Moreover,  $\cN(\overline{\bg}) \ni \bg \mapsto \bx(\cdot \, ; \bg) \in \WW^{n_x}$  is continuously Fr\'echet 
differentiable with respect to $\bg \in \cN(\overline{\bg})$, and the Fr\'echet  derivative 
$\delta \bx := \bx_\bg(\overline{\bg}) \delta \bg$ applied to $\delta \bg$
is given by the solution of the linear initial value problem
\begin{equation}    \label{eq:Sensitivity:ODE:x-sensitivity}
	\delta \bx'(t) = \overline{\BA}[t] \delta \bx(t) + \overline{\BB}[t] \delta \bg\big(t, \overline{\bx}(t) \big), \quad \aall t \in I, \qquad
	\delta \bx(t_0) = 0,
\end{equation}
where $\overline{\bx} := \bx(\cdot \, ; \overline{\bg})$ and $\overline{\BA} \in \LL^{n_x \times n_x}$, $\overline{\BB} \in \LL^{n_x \times n_g}$ 
are given by
\begin{equation} \label{eq:Sensitivity:ODE:shorthand}
	\begin{aligned}
	  \overline{\BA}[\cdot] &:= \bff_x\Big( \cdot, \overline{\bx}(\cdot), \overline{\bg}\big(\cdot, \overline{\bx}(\cdot) \big) \Big)
	                          + \bff_g\Big( \cdot, \overline{\bx}(\cdot), \overline{\bg}\big(\cdot, \overline{\bx}(\cdot) \big) \Big) \overline{\bg}_x \big(\cdot, \overline{\bx}(\cdot) \big), \\
	\overline{\BB}[\cdot]  &:= \bff_g\Big( \cdot, \overline{\bx}(\cdot), \overline{\bg}\big(\cdot, \overline{\bx}(\cdot) \big) \Big).
	\end{aligned}
\end{equation}

The sensitivity of the QoI \eqref{eq:problem_formulation:QoI-ODE} with respect to $\bg$ follows from
the Fr\'echet differentiability of $\bg \mapsto \bx(\cdot \, ; \bg)$ and the chain rule,
and can be computed using an adjoint equation. This is
summarized in the following result. 
In addition to \eqref{eq:Sensitivity:ODE:shorthand}, we use the shorthand
\begin{equation} \label{eq:Sensitivity:ODE:shorthand-2}
	\overline{\ell}[\cdot] := \ell \big( \cdot, \overline{\bx}(\cdot), \overline{\bg} \big( \cdot, \overline{\bx}(\cdot) \big) \big), \qquad 
	\overline{\phi}[t_f] := \phi\big( \overline{\bx}(t_f) \big).
\end{equation}

Under suitable conditions on the problem, which will be made precise in 
\cref{thm:Sensitivity:ODE:QoI-sensitivity}, the QoI $\cN(\overline{\bg}) \ni \bg \mapsto \widetilde{q}(\bg) \in \real$
defined in \eqref{eq:problem_formulation:QoI-ODE} is continuously Fr\'echet 
differentiable with Fr\'echet derivative given by
\begin{equation} \label{eq:Sensitivity:ODE:sensitivity_result_qoi_adjoint}
    	\widetilde{q}_\bg(\overline{\bg}) \delta \bg = \int_{t_0}^{t_f} \big( \overline{\BB}[t]^T \overline{\blambda}(t) + \nabla_g \overline{\ell}[t] \big)^T \delta \bg \big( t, \overline{\bx}(t) \big) \, dt,
\end{equation}
where $\overline{\blambda} \in \WW^{n_x}$ solves the adjoint equation
 \begin{equation} \label{eq:Sensitivity:ODE:x-adjoint}
	\begin{aligned}
	       - \overline{\blambda}'(t) &= \overline{\BA}[t]^T \overline{\blambda}(t) + \nabla_x \overline{\ell}[t]  
	                                                + \overline{\bg}_x\big(t, \overline{\bx}(t) \big)^T \nabla_g \overline{\ell}[t], \quad \aall t \in I, \\
		\overline{\blambda}(t_f) &= \nabla_x \overline{\phi}[t_f].
       \end{aligned}
\end{equation}

The advantage of the adjoint-based approach is that at the cost of one linear ODE solve \eqref{eq:Sensitivity:ODE:x-adjoint}, 
it allows one to compute the QoI sensitivity for any $\delta \bg$ by simply applying the linear operator \eqref{eq:Sensitivity:ODE:sensitivity_result_qoi_adjoint}, which is cheaper than solving a linear IVP \eqref{eq:Sensitivity:ODE:x-sensitivity} for every $\delta \bg$.

%%%%%%%%%%%%%%%%%%%%%%%%%%%%%%%%%%%%%%%%%%%%%%%%%%%%%%
\subsection{Summary of OCP Solution Sensitivity}   \label{sec:sensitivity-OCP-summary}
In this section we summarize the results on the sensitivity of the solution of the OCP
\eqref{eq:Sensitivity:OCP:OCP} and of the QoI \eqref{eq:problem_formulation:QoI-OCP}, obtained via a post-optimality sensitivity analysis.
The precise statement of assumptions needed for these results to hold, as well as the relevant theorems on sensitivity analysis,
will be given in \cref{sec:sensitivity-OCP}.

The optimization variables  are combined in $\by = (\bx, \bu, \bp)$
and the optimization variables with the costate $\blambda$ associated with \eqref{eq:Sensitivity:OCP:OCP}
are combined in $\bz = (\bx, \bu, \bp, \blambda)$.
The functions $\by$ and $\bz$ (resp.) belong to the Banach spaces
\begin{equation} \label{eq:Sensitivity:OCP:function-spaces}
\begin{aligned}
	\cY^\infty &:= \WW^{n_x} \times \LL^{n_u} \times \real^{n_p}, \\
	\cZ^\infty &:= \WW^{n_x} \times \LL^{n_u} \times \real^{n_p} \times \WW^{n_x}.
\end{aligned}
\end{equation}

To state the sensitivity results, we will use  the shorthand
\begin{equation} \label{eq:Sensitivity:OCP:shorthand-1}
    \begin{aligned}
         \overline{\bff}[\cdot] &:= \bff\big( \cdot, \overline{\bx}(\cdot), \overline{\bu}(\cdot), \overline{\bp}, \, \overline{\bg} \big( \cdot, \overline{\bx}(\cdot), \overline{\bu}(\cdot), \overline{\bp} \big) \big), & \qquad
          \overline{\bg}[\cdot] &:= \overline{\bg} \big( \cdot, \overline{\bx}(\cdot), \overline{\bu}(\cdot), \overline{\bp} \big), \\
    	\overline{l}[\cdot] &:=  l\big( \cdot, \overline{\bx}(\cdot), \overline{\bu}(\cdot), \overline{\bp}, \, \overline{\bg} \big( \cdot, \overline{\bx}(\cdot), \overline{\bu}(\cdot), \overline{\bp} \big) \big), & \qquad
    	\overline{\varphi}[t_f] &:= \varphi( \overline{\bx}(t_f), \bp ),
    \end{aligned}
\end{equation}
and corresponding shorthand for partial derivatives.
Moreover, we will use the Hamiltonian
\begin{subequations}     \label{eq:Sensitivity:OCP:shorthand}
\begin{equation}
		\overline{\BH}[\cdot] := \overline{l}[\cdot] + \overline{\blambda}(\cdot)^T \overline{\bff}[\cdot]
\end{equation}
and its total derivatives
\begin{align}
	\overline{\BH}_{z_1 z_2}[\cdot] 
	&= \nabla_{z_1 z_2}^2 \overline{\BH}[\cdot] + \nabla_{z_1 g}^2 \overline{\BH}[\cdot] \overline{\bg}_{z_2}[\cdot] + (\nabla_{z_1 g}^2 \overline{\BH}[\cdot] \overline{\bg}_{z_2}[\cdot])^T \\ &\quad + \overline{\bg}_{z_1}[\cdot]^T \nabla_{gg}^2 \overline{\BH}[\cdot] \overline{\bg}_{z_2}[\cdot], & z_1, z_2 \in \{ x, u, p \}, \nonumber \\
	\overline{\BH}_{z_1 g}[\cdot]
	&= \nabla_{z_1 g}^2 \overline{\BH}[\cdot] + \overline{\bg}_{z_1}[\cdot]^T \nabla_{gg}^2 \overline{\BH}[\cdot], & z_1 \in \{ x, u, p \},
\end{align}
as well as
\begin{equation}
\begin{aligned}
        \overline{\BA}[\cdot] &= \overline{\bff}_x[\cdot] + \overline{\bff}_g[\cdot] \overline{\bg}_x[\cdot], &
	\overline{\BB}[\cdot] &= \overline{\bff}_u[\cdot] + \overline{\bff}_g[\cdot] \overline{\bg}_u[\cdot], \\
	\overline{\BC}[\cdot] &= \overline{\bff}_p[\cdot] + \overline{\bff}_g[\cdot] \overline{\bg}_p[\cdot], &
	\overline{\bd}[\cdot] &= \overline{\bff}_g[\cdot]^T \overline{\blambda}(\cdot).
\end{aligned}
\end{equation}
\end{subequations}

Let  $(\overline{\bx}, \overline{\bu}, \overline{\bp}) \in \cY^\infty$ be
a local minimum of \eqref{eq:Sensitivity:OCP:OCP} with $\bg = \overline{\bg} \in \cG^3$ 
and let $\overline{\blambda} \in \WW^{n_x}$ be the corresponding costate.
Under suitable conditions on \eqref{eq:Sensitivity:OCP:OCP}, which will be stated rigorously in
\cref{thm:OCP-sensitivity}, 
there exist neighborhoods
    $\cN(\overline{\bg}) \subset \cG^3$,
    $\cN(\overline{\bx}) \subset \big( W^{1,\infty}(I) \big)^{n_x}$,
    $\cN(\overline{\bu}) \subset \big( L^\infty(I) \big)^{n_u}$,
    $\cN(\overline{\bp}) \subset \real^{n_p}$,
    $\cN(\overline{\blambda}) \subset \big( W^{1,\infty}(I) \big)^{n_x}$
and unique mappings
\begin{align*}
		\bx : \cN(\overline{\bg}) \rightarrow \cN(\overline{\bx}), \quad
		\bu : \cN(\overline{\bg}) \rightarrow \cN(\overline{\bu}), \quad
		\bp : \cN(\overline{\bg}) \rightarrow \cN(\overline{\bp}), \quad
		\blambda : \cN(\overline{\bg}) \rightarrow \cN(\overline{\blambda})
\end{align*}
satisfying $\bx(\overline{\bg}) = \overline{\bx}$,
		$\bu(\overline{\bg}) = \overline{\bu}$,
		$\bp(\overline{\bg}) = \overline{\bp}$,
		$\blambda(\overline{\bg}) = \overline{\blambda}$, 
such that $\big(\bx(\bg), \bu(\bg), \bp(\bg)\big)$ is a local solution of \eqref{eq:Sensitivity:OCP:OCP} with
 $\bg \in \cN(\overline{\bg})$ and $\blambda(\bg)$ is the corresponding costate.
Moreover, the mapping 
$ \cN(\overline{\bg}) \ni \bg \mapsto \bz(\bg) := \big(\bx(\bg), \bu(\bg), \bp(\bg), \blambda(\bg) \big)
\in \cZ^\infty $ is continuously Fr\'echet differentiable at $\overline{\bg}$.
The $\delta \bx, \delta \bu, \delta \bp$ components of the 
Fr\'echet derivative
$\delta \bz = \bz_\bg(\overline{\bg}) \delta \bg = (\delta \bx, \delta \bu, \delta \bp, \delta \blambda) \in \cZ^\infty$
are the solution of the  linear quadratic optimal control  problem (LQOCP) 
\begin{equation}     \label{eq:Sensitivity:OCP:LQOCP}
    \begin{aligned}
    	\min_{\delta \bx, \delta \bu, \delta \bp} \quad & \int_{t_0}^{t_f} \begin{bmatrix} \bc_x(t) \\ \bc_u(t) \end{bmatrix}^T \begin{bmatrix} \delta \bx(t) \\ \delta \bu(t) \end{bmatrix} \, dt  \\
    	& + \frac{1}{2} \int_{t_0}^{t_f} \begin{bmatrix} \delta \bx(t) \\ \delta \bu(t) \\ \delta \bp \end{bmatrix}^T \begin{bmatrix} \overline{\BH}_{xx}[t] & \overline{\BH}_{xu}[t] & \overline{\BH}_{xp}[t] \\ \overline{\BH}_{ux}[t] & \overline{\BH}_{uu}[t] & \overline{\BH}_{up}[t] \\ \overline{\BH}_{px}[t] & \overline{\BH}_{pu}[t] & \overline{\BH}_{pp}[t] \end{bmatrix} \begin{bmatrix} \delta \bx(t) \\ \delta \bu(t) \\ \delta \bp \end{bmatrix} \, dt \\
    		& + \begin{bmatrix} \sigma_f \\ \gamma \end{bmatrix}^T \begin{bmatrix} \delta \bx(t_f) \\ \delta \bp \end{bmatrix} 
    		  + \frac{1}{2} \begin{bmatrix} \delta \bx(t_f) \\ \delta \bp \end{bmatrix}^T
    		            \begin{bmatrix}  \nabla_{xx}^2 \overline{\varphi}[t_f] &  \nabla_{xp}^2 \overline{\varphi}[t_f] \\ 
    		                                      \nabla_{px}^2 \overline{\varphi}[t_f] &  \nabla_{pp}^2 \overline{\varphi}[t_f] 
    		           \end{bmatrix} 
    		           \begin{bmatrix} \delta \bx(t_f) \\ \delta \bp \end{bmatrix} \\[1ex]
            \mbox{s.t.} \quad
    	& \delta \bx'(t) = \overline{\BA}[t] \delta \bx(t) + \overline{\BB}[t] \delta \bu(t) + \overline{\BC}[t] \delta \bp + \br(t), \qquad \aall t \in I, \\
    	& \delta \bx(t_0) = r_0
    \end{aligned}
\end{equation}
 with
 \begin{subequations}   \label{eq:Sensitivity:OCP:ocp-solution-sensitivity-system-linear} 
	 \begin{align}
		\br(t) &= \overline{\bff}_g[t] \delta \bg[t],  &
		 \begin{bmatrix} \bc_x(t) \\ \bc_u(t) \end{bmatrix} 
		 &= \begin{bmatrix} \overline{\BH}_{xg}[t] \delta \bg[t] + \delta \bg_x[t]^T \overline{\bd}[t] \\  
		                             \overline{\BH}_{ug}[t] \delta \bg[t] + \delta \bg_u[t]^T \overline{\bd}[t]
		        \end{bmatrix}, \\
		r_0 &= \sigma_f = 0, &
		 \gamma &= \int_{t_0}^{t_f}   \overline{\BH}_{pg}[t] \delta \bg[t] + \delta \bg_p[t]^T \overline{\bd}[t] \, dt,
	 \end{align}
\end{subequations}
and the $\delta \blambda$ component is the corresponding adjoint.     
The necessary and sufficient optimality conditions for the LQOCP \eqref{eq:Sensitivity:OCP:LQOCP}
are stated in \cite[Eqn.~(2.19)]{JRCangelosi_MHeinkenschloss_2025a}.

Next, consider the sensitivity of the QoI \eqref{eq:problem_formulation:QoI-OCP}.
As in the previous section, 
the continuous Fr\'echet differentiablity of the QoI \eqref{eq:problem_formulation:QoI-OCP-g-only} 
with respect to $\bg \in \cG^3$ follows from the OCP solution differentiability and the chain rule.
Once again, the Fr\'echet derivative of the QoI  can be computed using adjoints.
We continue to use the shorthands \eqref{eq:Sensitivity:OCP:shorthand-1} and \eqref{eq:Sensitivity:OCP:shorthand}.

Under suitable assumptions on \eqref{eq:Sensitivity:OCP:OCP} and \eqref{eq:problem_formulation:QoI-OCP}, 
which will be made precise in \cref{thm:Sensitivity:OCP:sensitivity_result_qoi-adjoint},
the map $ \cN(\overline{\bg}) \ni \bg \mapsto  \widetilde{q}(\bg)$ given by 
 \eqref{eq:problem_formulation:QoI-OCP-g-only} is continuously Fr\'echet differentiable 
 and the Fr\'echet derivative is	
 \begin{align} \label{eq:Sensitivity:OCP:sensitivity_result_qoi_adjoint}
    	& \widetilde{q}_\bg(\overline{\bg}) \delta \bg  \nonumber \\ 
	&=  \int_{t_0}^{t_f}  \overline{\bh}[t]^T  \delta \bg[t]  
                                             +  \overline{\bd}[t]^T \delta \bg_x[t] \widetilde{\delta \bx}(t) 
                                             +  \overline{\bd}[t]^T \delta \bg_u[t] \widetilde{\delta \bu}(t) 
                                             +   \overline{\bd}[t]^T \delta \bg_p[t] \widetilde{\delta \bp}  \, dt,
 \end{align}
 where 
  \[
               \overline{\bh}[t] 
              := \overline{\BH}_{gx}[t]  \widetilde{\delta \bx}(t) 
                                                  + \overline{\BH}_{gu}[t]  \widetilde{\delta \bu}(t)
                                                  + \overline{\BH}_{gp}[t]  \widetilde{\delta \bp}
                                                  + \overline{\bff}_{g}[t]^T \widetilde{\delta \blambda}(t) 
                                                  +   \nabla_g \overline{\ell}[t], 
  \]
          \sloppy
 where
 $(\widetilde{\delta \bx}, \; \widetilde{\delta \bu}, \; \widetilde{\delta \bp}) \in \cY^\infty$
solves the LQOCP \eqref{eq:Sensitivity:OCP:LQOCP} with $(\delta \bx, \delta \bu, \delta \bp)$ 
replaced by $(\widetilde{\delta \bx}, \widetilde{\delta \bu}, \widetilde{\delta \bp})$ and using
\begin{align*}
		\br(t) &\equiv 0,  &
		 \begin{bmatrix} \bc_x(t) \\ \bc_u(t) \end{bmatrix} 
		 &= \begin{bmatrix}   \nabla_x \overline{\ell}[t] + \overline{\bg}_x[t]^T \nabla_g \overline{\ell}[t] \\  
		                      \nabla_u \overline{\ell}[t] + \overline{\bg}_u[t]^T \nabla_g \overline{\ell}[t]
		        \end{bmatrix}, \\
		r_0 &= 0, \; \sigma_f = \nabla_x \overline{\phi}[t_f], &
		 \gamma &= \nabla_p \overline{\phi}[t_f] + \int_{t_0}^{t_f} \nabla_p \overline{\ell}[t] + \overline{\bg}_p[t]^T \nabla_g \overline{\ell}[t] \, dt,
 \end{align*}
and where $\widetilde{\delta \blambda} \in \WW^{n_x}$ is the corresponding costate.
   
 The necessary and sufficient optimality conditions for the LQOCP that needs to be solved to
 compute   $(\widetilde{\delta \bx}, \; \widetilde{\delta \bu}, \; \widetilde{\delta \bp}, \; \widetilde{\delta \blambda}) \in \cZ^\infty$
 are stated in \cite[Eqn.~(2.24)]{JRCangelosi_MHeinkenschloss_2025a}.
The advantage of the adjoint-based approach is that at the cost of solving one LQOCP, it allows one to compute the QoI sensitivity for any $\delta \bg$ by simply applying the linear operator \eqref{eq:Sensitivity:OCP:sensitivity_result_qoi_adjoint}, which is much cheaper than solving the LQOCP \eqref{eq:Sensitivity:OCP:LQOCP} for each $\delta \bg$.

\begin{remark}
    Our post-optimality sensitivity analysis relies on the continuous optimality/KKT conditions.
    In practice, discretization of the optimal control problem is solved using a nonlinear optimization
    solver that will satisfy the KKT conditions for a discretized problem up to solver tolerance. 
    As a result, the computed sensitivities inherit errors due to the discretization error and the solver tolerances.
    However, the computed sensitivities can be interpreted as the sensitivities of a perturbed problem, with
    perturbations related to discretization and solver errors. As these errors converge to zero, the 
    computed sensitivities converge to the true sensitivities. 
    When we apply these sensitivity results in \cref{sec:numerics}, we use a higher-order collocation
    discretization for the optimal control problem, with proven convergence as the discretization is refined,
    and we use a second-order optimization code with relatively fine solver tolerances. 
 \end{remark}

%%%%%%%%%%%%%%%%%%%%%%%%%%%%%%%%%%%%%%%%%%%%%%%%%%%%%%%%%%%%%
\subsection{Sensitivity of IVP Solution}  \label{sec:sensitivity-ODE}
This section provides the technical details for and the precise statements of the sensitivity results
summarized in \cref{sec:sensitivity-ODE-summary}. 
The main results on the sensitivity of the solution $\bx(\bg)$ of \eqref{eq:problem_formulation:IVP} and of the QoI \eqref{eq:problem_formulation:QoI-ODE}
with respect to $\bg$ are stated in \cref{thm:Sensitivity:ODE:sensitivity-result-ODE,thm:Sensitivity:ODE:QoI-sensitivity} below.
These results are obtained by applying the implicit function theorem to an operator equation derived from \eqref{eq:problem_formulation:IVP}.
We begin by specifying the setting needed to apply the implicit function theorem.

The following assumption on $\bff$ ensures existence and uniqueness of solutions to the IVP \eqref{eq:problem_formulation:IVP} in $\WW^{n_x}$,
which is required for the sensitivity of the solution to be well-defined.
\begin{assumption} \label{as:Sensitivity:caratheodory:IVP_unique_solution}
	Let the function $\bff : I \times \real^{n_x} \times \real^{n_g} \rightarrow \real^{n_x}$ satisfy the following:
   \begin{itemize}
   \item[(i)] The function $\bff$ is measurable in $t \in I$ for each $x \in \real^{n_x}$ and $g  \in \real^{n_g}$, 
   	it is continuous in $x \in \real^{n_x}$ and $g \in \real^{n_g}$ for a.a.\  $t \in I$,
      and there exists a function $m_f \in L^\infty(I)$ such that 
      $\| \bff( t, x, g ) \| \le m_f(t) (1 + \| g \|)$ for $\aall t \in I$ and all $x \in \real^{n_x}, g \in \real^{n_g}$.
      
   \item[(ii)]   \sloppy 
        There exists an integrable function $l_f : I \rightarrow \real$ such that 
                    $\| \bff( t, x_1, g_1 ) -  \bff( t, x_2, g_2 )  \| \le l_f(t) ( \| x_1 - x_2  \| + \| g_1 - g_2  \| )$
                    for $\aall t \in I$ and all $x_1, x_2 \in \real^{n_x}, g_1, g_2 \in \real^{n_g}$.
\end{itemize}
\end{assumption}

\begin{theorem} \label{thm:Sensitivity:caratheodory:IVP_unique_solution}
    If Assumption~\ref{as:Sensitivity:caratheodory:IVP_unique_solution}  holds, then for any $\bg \in\cG^2$
     the IVP \eqref{eq:problem_formulation:IVP} has a unique solution $\bx(\cdot \, ; \bg) \in \WW^{n_x}$.
\end{theorem}
\begin{proof}
This result is a corollary of \cite[Thm.~2.2]{JRCangelosi_MHeinkenschloss_2025c}.
\end{proof}

The continuous Fr\'echet differentiability of the solution mapping
$\cG^2 \ni \bg \mapsto \bx(\cdot \, ; \bg) \in \WW^{n_x}$ follows from the application of the
implicit function theorem to the operator equation 
\begin{equation} \label{eq:Sensitivity:psi-ODE}
	\bpsi(\bx, \bg) 
	            = \begin{bmatrix} \BF(\bx, \bg )  -  \bx' \\  \bx(t_0) - x_0 \end{bmatrix} = 0,
\end{equation}
where $\bpsi :  \WW\times\cG^2  \rightarrow  \LL \times \real^{n_x}$ and
\[
	\WW^{n_x} \times\cG^2 \ni (\bx, \bg) \mapsto 
	\BF(\bx, \bg) := \bff \big( \cdot, \bx(\cdot), \bg \big( \cdot, \bx(\cdot) \big) \big) \in \LL^{n_x}.
\]

The operator $\BF$ is a Nemytskii operator, also called a superposition operator.
To study the continuous Fr\'echet differentiability of the QoI \eqref{eq:problem_formulation:QoI-ODE-g-only} we also need the Nemytskii operator
\[
	\WW^{n_x} \times\cG^2 \ni (\bx, \bg) \mapsto \BL(\bx, \bg) := \ell \big( \cdot, \bx(\cdot), \bg \big( \cdot, \bx(\cdot) \big) \big) \in L^\infty(I).
\]

The next result establishes the continuous Fr\'echet differentiability of Nemytskii operators in the form of $\BF$ and $\BL$. 
Continuous Fr\'echet differentiability of general nonlinear Nemytskii operators of this form can be proven in $L^\infty$ 
spaces, e.g.,  \cite[Sec.~4.3]{FTroeltzsch_2010a}. The assumptions made in the following theorem are consistent with those outlined in \cite[Lemma~4.13]{FTroeltzsch_2010a}.
\begin{theorem} \label{thm:Sensitivity:nemytskii:nemytskii-derivative}
	Let $\bphi : I \times \real^{n_x} \times \real^{n_g} \rightarrow \real^{n_\phi}$ be a given function such that $\bphi(t, x, g)$ is measurable in $t \in I$ for all $x \in \real^{n_x}$, $g \in \real^{n_g}$ and is continuously partially differentiable with respect to $x \in \real^{n_x}$ and $g  \in \real^{n_g}$ for a.a.\  $t \in I$.
	Moreover, assume that the partial derivatives of $\bphi$ satisfy the following conditions:
   \begin{itemize}                   
   \item[(i)] \emph{Boundedness:}  There exists $K$ such that $ \| \bphi_x( t, 0, 0 )  \| \le K$ and 
                    $ \| \bphi_g( t, 0, 0 )  \| \le K$  for a.a.\  $t \in I$.
                   
   \item[(ii)]  \emph{Local Lipschitz continuity:} 
       \sloppy  For all $R > 0$ there exists $L(R)$ such that
    \begin{align*}   
    	&\| \bphi_x(t, x_1, g_1) - \bphi_x(t, x_2, g_2) \| + \| \bphi_g(t, x_1, g_1) - \bphi_g(t, x_2, g_2) \| \\
         &\leq L(R) ( \| x_1 - x_2 \| + \| g_1 - g_2 \| ), \quad  \aall  t \in I \mbox{ and  } \\
         & \hspace*{30ex}  \mbox{ all } x_1, x_2 \in \cB_R(0), \; g_1, g_2 \in \cB_R(0).
    \end{align*}
    \end{itemize}
	Then the Nemytskii operator $\BPhi : \LL^{n_x} \times\cG^2 \rightarrow \LL^{n_\phi}$ given by
	\begin{equation} \label{eq:Sensitivity:nemytskii:nemytskii-operator}
		\BPhi(\bx, \bg) := \bphi \Big( \cdot, \bx(\cdot), \bg \big( \cdot, \bx(\cdot) \big) \Big)
	\end{equation}
	is continuously Fr\'echet differentiable, and its derivative is given by
	\begin{equation} \label{eq:Sensitivity:nemytskii:Frechet_deriv}
	\begin{aligned}
      	& [\BPhi'(\bx, \bg) (\delta \bx, \delta \bg)](t)   \\
	&= \Big[ \bphi_x\big( t, \bx(t), \bg\big(t, \bx(t) \big) \big) 
	              + \bphi_g\big( t, \bx(t), \bg\big(t, \bx(t) \big) \big) \bg_x\big(t, \bx(t) \big) \Big] \delta \bx(t)    \\ 
	 &\quad + \bphi_g\big( t, \bx(t), \bg\big( t,  \bx(t) \big) \big) \delta \bg\big(t,  \bx(t) \big).
    \end{aligned}
    \end{equation}
\end{theorem}
\begin{proof}
See \cite[Thm.~2.5]{JRCangelosi_MHeinkenschloss_2025c}.
\end{proof}

The following result gives the sensitivity of the solution of \eqref{eq:problem_formulation:IVP} to perturbations in $\bg$.

\begin{theorem} \label{thm:Sensitivity:ODE:sensitivity-result-ODE}
	Suppose $\bff : I \times \real^{n_x} \times \real^{n_g} \rightarrow \real^{n_x}$ satisfies Assumption~\ref{as:Sensitivity:caratheodory:IVP_unique_solution} 
	and the assumptions of Theorem~\ref{thm:Sensitivity:nemytskii:nemytskii-derivative} with $\bphi = \bff$. 
	If $\overline{\bg} \in\cG^2$ is given and $\overline{\bx} \in \WW^{n_x}$ is the corresponding solution of \eqref{eq:problem_formulation:IVP},
	then there exist neighborhoods $\cN(\overline{\bg}) \subset\cG^2$ and $\cN(\overline{\bx}) \subset \big( W^{1,\infty}(I) \big)^{n_x}$
	and a unique map $\bx : \cN(\overline{\bg}) \rightarrow \cN(\overline{\bx})$ satisfying $\bx(\overline{\bg}) = \overline{\bx}$
	and $\bpsi\big( \bx(\bg), \bg \big) = 0$ for all $\bg \in \cN(\overline{\bg})$, where $\bpsi$ is defined in \eqref{eq:Sensitivity:psi-ODE}.
	
	Moreover,  the unique solution $\bx(\cdot \, ; \bg) \in \WW^{n_x}$ of the initial value problem \eqref{eq:problem_formulation:IVP} is continuously Fr\'echet 
	differentiable with respect to $\bg \in \cN(\overline{\bg})$, and the Fr\'echet  derivative $\delta \bx := \bx_\bg(\overline{\bg}) \delta \bg$ applied to $\delta \bg$
	 is given by the solution of the linear initial value problem~\eqref{eq:Sensitivity:ODE:x-sensitivity}.
\end{theorem}
\begin{proof}
	See \cite[Thm.~2.10]{JRCangelosi_MHeinkenschloss_2025c}.
\end{proof}

The sensitivity of the QoI \eqref{eq:problem_formulation:QoI-ODE} with respect to $\bg$ follows from
\cref{thm:Sensitivity:ODE:sensitivity-result-ODE} and the chain rule, and can be computed using an adjoint equation. This is
summarized in the following result. 
In addition to \eqref{eq:Sensitivity:ODE:shorthand}, we use the shorthand
\begin{equation} \label{eq:Sensitivity:ODE:shorthand-2}
	\overline{\ell}[\cdot] := \ell \big( \cdot, \overline{\bx}(\cdot), \overline{\bg} \big( \cdot, \overline{\bx}(\cdot) \big) \big), \qquad 
	\overline{\phi}[t_f] := \phi\big( \overline{\bx}(t_f) \big).
\end{equation}

\begin{theorem} \label{thm:Sensitivity:ODE:QoI-sensitivity}
	Let the assumptions of \cref{thm:Sensitivity:ODE:sensitivity-result-ODE} hold. 
	If $\phi$ is continuously differentiable and $\bphi = \ell$ satisfies the assumptions of  
	\cref{thm:Sensitivity:nemytskii:nemytskii-derivative},
    then the QoI $\cN(\overline{\bg}) \ni \bg \mapsto \widetilde{q}(\bg) \in \real$
      defined in \eqref{eq:problem_formulation:QoI-ODE} is continuously Fr\'echet 
     differentiable with Fr\'echet derivative given by \eqref{eq:Sensitivity:ODE:sensitivity_result_qoi_adjoint}
\end{theorem}
\begin{proof}
	See \cite[Thm.~2.13]{JRCangelosi_MHeinkenschloss_2025c}.
\end{proof}

%%%%%%%%%%%%%%%%%%%%%%%%%%%%%%%%%%%%%%%%%%%%%%%%%%%%%%
\subsection{Sensitivity of OCP Solution}   \label{sec:sensitivity-OCP}
This section provides the technical details for and the precise statements of the sensitivity results
summarized in \cref{sec:sensitivity-OCP-summary}. 
As before, the optimization variables are combined in $\by = (\bx, \bu, \bp)$
and  $\bz = (\bx, \bu, \bp, \blambda)$, where $\blambda$ is the costate associated with \eqref{eq:Sensitivity:OCP:OCP}. 
The functions $\by$ and $\bz$ (resp.) belong to the Banach spaces \eqref{eq:Sensitivity:OCP:function-spaces}.
We begin by specifying the setting needed to state the optimality conditions for \eqref{eq:Sensitivity:OCP:OCP} and to apply the implicit function theorem.

The following smoothness assumption is needed to obtain the necessary optimality conditions for \eqref{eq:Sensitivity:OCP:OCP} 
and is adapted from \cite[Assumption~2.2.8]{MGerdts_2012a}. 
\begin{assumption} \label{as:Sensitivity:OCP:smoothness}
      Let the functions $\varphi$, $l$, $\bff$, $\bg$ in \eqref{eq:Sensitivity:OCP:OCP} satisfy the following:
	\begin{itemize}
		\item[(i)] $\varphi$ is continuously differentiable.
		
		\item[(ii)] The mappings
		     $t \mapsto l \big(t, y, \bg(t, y) \big)$ and
		     $ t \mapsto \bff\big(t, y, \bg(t, y) \big)$
		are measurable on $I$ for all $y \in \real^{n_y}$.
		
		\item[(iii)] The mappings
		     $y \mapsto l\big(t, y, \bg(t, y)\big)$ and
		     $ y \mapsto \bff\big(t, y, \bg(t, y)\big)$
		are uniformly continuously differentiable on $\real^{n_y}$ for all $t \in I$.
		
		\item [(iv)] \sloppy The first-order (total) derivatives
		     $(t, y) \mapsto \frac{d}{dy} l\big(t, y, \bg(t, y)\big)$ and
		     $(t, y) \mapsto \frac{d}{dy}\bff\big(t, y, \bg(t, y) \big)$
		are bounded in $I \times \real^{n_y}$.
	\end{itemize}
\end{assumption}

The first-order necessary optimality conditions for \eqref{eq:Sensitivity:OCP:OCP} are given in the following theorem. 
See \cite{JRCangelosi_MHeinkenschloss_2025a} for more details on the optimality conditions.

\begin{theorem} \label{thm:Sensitivity:OCP:NOC}
  Let \cref{as:Sensitivity:OCP:smoothness} hold with $\bg = \overline{\bg}$.
  If  $(\overline{\bx}, \overline{\bu}, \overline{\bp}) \in \cY^\infty$
  is a local minimum of \eqref{eq:Sensitivity:OCP:OCP} with $\bg = \overline{\bg}$, then
  there exists $\overline{\blambda} \in \big(W^{1,\infty}(I)\big)^{n_x}$ such that 
    \begin{align*}
    		\overline{\bx}'(t) &= \overline{\bff}[t], \hspace*{50ex}  \aall t \in (t_0, t_f), \\ 
    	        \overline{\bx}(t_0) &= x_0, \\ 
    		\overline{\blambda}'(t) &= - \big(\overline{\bff}_x[t] + \overline{\bff}_g[t] \overline{\bg}_x[t] \big)^T \overline{\blambda}(t) - (\nabla_x \overline{l}[t] + \overline{\bg}_x[t]^T \nabla_g \overline{l}[t]), \hspace*{3ex} \aall t \in (t_0, t_f), \\  
    		\overline{\blambda}(t_f) &= \nabla_x \overline{\varphi}[t_f],\\ 
    		0 &=\big(\overline{\bff}_u[t] + \overline{\bff}_g[t] \overline{\bg}_u[t] \big)^T \overline{\blambda}(t) + (\nabla_u \overline{l}[t] + \overline{\bg}_u[t]^T \nabla_g \overline{l}[t]), \hspace*{5ex} \aall t \in (t_0, t_f), \\  
    		0 &= \nabla_p \overline{\varphi}[t_f] 
		          + \int_{t_0}^{t_f} \big(\overline{\bff}_p[t] + \overline{\bff}_g[t] \overline{\bg}_p[t] \big)^T \overline{\blambda}(t) 
		          + \nabla_p \overline{l}[t] + \overline{\bg}_p[t]^T \nabla_g \overline{l}[t] \, dt,
    \end{align*}
    using the shorthand \eqref{eq:Sensitivity:OCP:shorthand-1}
    and corresponding shorthand for partial derivatives.
\end{theorem}
\begin{proof}
	This result follows from \cite[Thm.~3.1.11]{MGerdts_2012a}.
\end{proof}

The optimality conditions of \cref{thm:Sensitivity:OCP:NOC} are considered as an operator equation to which the
implicit function theorem is applied to obtain sensitivities.
We continue to use the shorthand notation \eqref{eq:Sensitivity:OCP:shorthand-1}.
If we define the operator
\begin{subequations}     \label{eq:Sensitivity:OCP:KKT-operator}
\begin{equation}
		\bpsi : \cZ^\infty \times \cG^3 \rightarrow \cV^\infty
\end{equation}  
with $\cV^\infty := \LL^{n_x} \times \real^{n_x} \times \LL^{n_x} \times \real^{n_x} \times \LL^{n_u} \times \real^{n_p}$ and
\begin{align}
    	& \bpsi(\overline{\bx}, \overline{\bu}, \overline{\bp}, \overline{\blambda}; \overline{\bg})    \nonumber \\
	&:= \begin{bmatrix} 
    		\overline{\bff}[\cdot] - \overline{\bx}'(\cdot) \\[0.5ex] 
    		x_0 - \overline{\bx}(t_0)  \\[0.5ex] 
    		\overline{\blambda}'(\cdot) + \big(\overline{\bff}_x[\cdot] 
		+ \overline{\bff}_g[\cdot] \overline{\bg}_x[\cdot] \big)^T \overline{\blambda}(\cdot) 
		+ \big(\nabla_x \overline{l}[\cdot] + \overline{\bg}_x[\cdot]^T \nabla_g \overline{l}[\cdot]\big)  \\[0.5ex] 
    		\nabla_x \overline{\varphi}[t_f] - \overline{\blambda}(t_f)  \\[0.5ex] 
    		\big(\overline{\bff}_u[\cdot] + \overline{\bff}_g[\cdot] \overline{\bg}_u[\cdot] \big)^T \overline{\blambda}(\cdot) 
		+ \big(\nabla_u \overline{l}[\cdot] + \overline{\bg}_u[\cdot]^T \nabla_g \overline{l}[\cdot]\big)  \\[0.5ex] 
    		 \nabla_p \overline{\varphi}[t_f] + \int_{t_0}^{t_f} \Big(\big(\overline{\bff}_p[t] 
		  + \overline{\bff}_g[t] \overline{\bg}_p[t] \big)^T \overline{\blambda}(t) + \big(\nabla_p \overline{l}[t] 
		  + \overline{\bg}_p[t]^T \nabla_g \overline{l}[t]\big) \Big) \, dt
		  \end{bmatrix},
    \end{align}
 \end{subequations}
the operator equation associated with the first-order necessary optimality conditions of \cref{thm:Sensitivity:OCP:NOC} 
is given by
\begin{equation}     \label{eq:Sensitivity:OCP:KKT-operator-eq}
    	\bpsi(\overline{\bz}; \overline{\bg})  = \bpsi(\overline{\bx}, \overline{\bu}, \overline{\bp}, \overline{\blambda}; \overline{\bg})  = 0.
\end{equation}

The next assumption ensures the continuous Fr\'echet differentiability of the operator $\bpsi$, which is required for the implicit function theorem.
\begin{assumption} \label{as:Sensitivity:OCP:smoothness-strong}
	 In addition to the properties listed in \cref{as:Sensitivity:OCP:smoothness}, let the following properties hold
	  for the functions $\varphi$, $l$, and $\bff$  in \eqref{eq:Sensitivity:OCP:OCP}:
	\begin{itemize}
		\item[(i)] $\varphi$ is twice continuously differentiable.
		
		\item[(ii)] There exists $K$ such that 
					$\big\| \frac{\partial^s}{\partial (x, u, p, g)^s} \bff(t, 0, 0, 0, 0) \big\| \leq K$ for $\aall t \in I$
					and all $s \in \{ 0, 1, 2 \}$.

		\item[(iii)] For all $R > 0$ there exists $L(R)$ such that for all $s \in \{ 0, 1, 2 \}$ it holds
		\begin{align*}
			&\left\| \frac{\partial^s}{\partial (x, u, p, g)^s} \bff(t, x_1, u_1, p_1, g_1) 
			         - \frac{\partial^s}{\partial (x, u, p, g)^s} \bff(t, x_2, u_2, p_2, g_2) \right\| \\ 
		        &\leq L(R) \| (x_1, u_1, p_1, g_1) - (x_2, u_2, p_2, g_2) \|, \quad
		           \aall t \in I \mbox{ and }  \\
		        & \hspace*{35ex} \mbox{ all } (x_i, u_i, p_i, g_i) \in \cB_R(0), \quad i = 1, 2.
		\end{align*}
		
		\item[(iv)] Properties (ii) and (iii) also hold for $l$ (\emph{mutatis mutandis}).
	\end{itemize}
\end{assumption}

The existence of a continuous inverse $\bpsi_{\bz}(\overline{\bz}; \overline{\bg})^{-1}$ is linked to the existence and uniqueness of 
solutions to a particular linear quadratic optimal control problem (see \eqref{eq:Sensitivity:OCP:LQOCP} below) whose solution gives the sensitivity of the solution of \eqref{eq:Sensitivity:OCP:OCP} to perturbations in $\bg$. 
We will use the notation \eqref{eq:Sensitivity:OCP:shorthand}.

The second-order sufficient optimality condition is stated next.
\begin{assumption} \label{as:Sensitivity:OCP:SSOC}
	The matrix
	\[
		     \nabla^2 \overline{\varphi}[t_f] 
		      :=    \begin{bmatrix}  \nabla_{xx}^2 \overline{\varphi}[t_f] &  \nabla_{xp}^2 \overline{\varphi}[t_f] \\ 
		                                      \nabla_{px}^2 \overline{\varphi}[t_f] &  \nabla_{pp}^2 \overline{\varphi}[t_f] 
		           \end{bmatrix} 
	\]
	is symmetric positive semidefinite, there exists $\epsilon > 0$ such that the matrix
	\[
	       \overline{\BH}_{yy}[t] 
		:= \begin{bmatrix} 
		\overline{\BH}_{xx}[t] & \overline{\BH}_{xu}[t] & \overline{\BH}_{xp}[t] \\ 
		\overline{\BH}_{ux}[t] & \overline{\BH}_{uu}[t] & \overline{\BH}_{up}[t] \\ 
		\overline{\BH}_{px}[t] & \overline{\BH}_{pu}[t] & \overline{\BH}_{pp}[t] 
		\end{bmatrix}
	\]
	obeys
	\begin{align*}
		&\int_{t_0}^{t_f}  \delta \by(t)^T \overline{\BH}_{yy}[t] \, \delta \by(t) \, dt 
		\geq \epsilon \, \| \delta \by \|_{( L^2(I) )^{n_x} \times ( L^2(I) )^{n_u} \times \real^{n_p}}^2
	\end{align*}
	for all $\delta \by = (\delta \bx, \delta \bu, \delta \bp) \in \big( W^{1, 2}(I) \big)^{n_x} \times \big( L^2(I) \big)^{n_u} \times \real^{n_p}$ satisfying
	\[
		\delta \bx'(t) = \overline{\BA}[t] \delta \bx(t) + \overline{\BB}[t] \delta \bu(t) + \overline{\BC}[t] \delta \bp, \quad \aall t \in I, \qquad
	\delta \bx(t_0) = 0,
	\]	
	the matrix $\overline{\BH}_{uu}[t]^{-1}$ exists for a.a.\  $t \in I$, and $\overline{\BH}_{uu}[\cdot]^{-1}$ is essentially bounded.
\end{assumption}

Finally, we state the main sensitivity result, which may be found in \cite[Thm.~2.17]{JRCangelosi_MHeinkenschloss_2025a}.
\begin{theorem} \label{thm:OCP-sensitivity}
     Let Assumption~\ref{as:Sensitivity:OCP:smoothness-strong} hold, and let $\bphi = l$ satisfy the assumptions of Theorem~\ref{thm:Sensitivity:nemytskii:nemytskii-derivative}.
      If $(\overline{\bx}, \overline{\bu}, \overline{\bp}) \in \cY^\infty$ is a local minimum of \eqref{eq:Sensitivity:OCP:OCP} with costate $\overline{\blambda} \in \WW^{n_x}$ and 
      model $\overline{\bg} \in \cG^3$ that satisfies Assumption~\ref{as:Sensitivity:OCP:SSOC},
      then there exist neighborhoods
    $\cN(\overline{\bg}) \subset \cG^3$,
    $\cN(\overline{\bx}) \subset \big( W^{1,\infty}(I) \big)^{n_x}$,
    $\cN(\overline{\bu}) \subset \big( L^\infty(I) \big)^{n_u}$,
    $\cN(\overline{\bp}) \subset \real^{n_p}$,
    $\cN(\overline{\blambda}) \subset \big( W^{1,\infty}(I) \big)^{n_x}$
	and unique mappings
	\begin{align*}
		\bx : \cN(\overline{\bg}) \rightarrow \cN(\overline{\bx}), \quad
		\bu : \cN(\overline{\bg}) \rightarrow \cN(\overline{\bu}), \quad
		\bp : \cN(\overline{\bg}) \rightarrow \cN(\overline{\bp}), \quad
		\blambda : \cN(\overline{\bg}) \rightarrow \cN(\overline{\blambda})
	\end{align*}
	satisfying $\bx(\overline{\bg}) = \overline{\bx}$,
		$\bu(\overline{\bg}) = \overline{\bu}$,
		$\bp(\overline{\bg}) = \overline{\bp}$,
		$\blambda(\overline{\bg}) = \overline{\blambda}$, and
	\begin{align*}
		\bpsi\big( \bx(\bg), \bu(\bg), \bp(\bg), \blambda(\bg); \bg \big) = 0 \qquad \qquad \mbox{ for all } \; \bg \in \cN(\overline{\bg}).
	\end{align*}
	
	Moreover, the mapping $\bz(\bg) := \big(\bx(\bg), \bu(\bg), \bp(\bg), \blambda(\bg) \big)$ is continuously Fr\'echet differentiable at $\overline{\bg}$.
	The $\delta \bx, \delta \bu, \delta \bp$ components of the 
         Fr\'echet derivative
         $\delta \bz = \bz_\bg(\overline{\bg}) \delta \bg = (\delta \bx, \delta \bu, \delta \bp, \delta \blambda) \in \cZ^\infty$
         are the solution of the LQOCP \eqref{eq:Sensitivity:OCP:LQOCP}
         and the $\delta \blambda$ component is the corresponding adjoint.
 \end{theorem}        
 
%%%%%%%%%%%%%%%%%%%%%%%%%%%%%%%%%%%%%%%%%%%%%
Next, consider the sensitivity of the QoI \eqref{eq:problem_formulation:QoI-OCP} in the spaces
\begin{equation} \label{eq:Sensitivity:OCP:QoI-tilde}
	\widetilde{q} : \cG^3 \rightarrow \real, \qquad
	q :  \cY^\infty  \times \cG^3 \rightarrow \real.
\end{equation}
As in the previous section, the sensitivity of the QoI \eqref{eq:problem_formulation:QoI-OCP-g-only} with respect to $\bg \in \cG^3$ follows from
\cref{thm:OCP-sensitivity} and the chain rule, and can once again  be computed using adjoints.
This is summarized in the following result. 
We continue to use the shorthands \eqref{eq:Sensitivity:OCP:shorthand-1} and \eqref{eq:Sensitivity:OCP:shorthand}.

\begin{theorem} \label{thm:Sensitivity:OCP:sensitivity_result_qoi-adjoint}
	If the assumptions of \cref{thm:OCP-sensitivity} hold, $\bphi = \ell$ satisfies the 
	assumptions of \cref{thm:Sensitivity:nemytskii:nemytskii-derivative}, 
	and $\phi$ is continuously differentiable,   then 
	the map $ \cN(\overline{\bg}) \ni \bg \mapsto  \widetilde{q}(\bg)$ given by 
        \eqref{eq:problem_formulation:QoI-OCP-g-only} is continuously Fr\'echet differentiable 
        with  Fr\'echet derivative given by \eqref{eq:Sensitivity:OCP:sensitivity_result_qoi_adjoint}.
 \end{theorem}
\begin{proof}
  See  \cite[Thm.~2.19]{JRCangelosi_MHeinkenschloss_2025a}.
\end{proof}    

%%%%%%%%%%%%%%%%%%%%%%%%%%%%%%%%%%%%%%%%%%%%%%%%%%%%%%%%%
%%%%%%%%%%%%%%%%%%%%%%%%%%%%%%%%%%%%%%%%%%%%%%%%%%%%%%%%%      
\section{Kernel-Based Surrogate Modeling} \label{sec:surrogates}
\sloppy
Another ingredient of our surrogate model adaptation is the construction of models $\bg$
for which certain pointwise error bounds are available. Interpolation in reproducing kernel Hilbert spaces 
(RKHSs) is one option for providing such models.
In this section, we review these models and highlight the properties leveraged in our surrogate model adaptation approach. Pointwise error bounds for kernel interpolants
allow us to make use of the fact that the inexactness in the component function at a given time depends on the current state, and that inexactness is propagated in time through sensitivity equations.

Throughout this section, we consider a subset $\Omega \subset \real^d$ and points $y \in \Omega$. 
Here we consider scalar-valued functions $\bg: \Omega \rightarrow \real$.
If $\bg$ is vector-valued, which is the case in our application, one can construct a separate surrogate model 
for each component of $\bg$, though one could also consider vector-valued kernels.

%%%%%%%%%%%%%%%%%%%%%%%%%%%%%%%%%%%%%%%%%%%%%%%%%%%%%%%%%%%
\subsection{Reproducing Kernel Hilbert Spaces}
We consider RKHSs of real-valued functions on $\Omega \subset \real^d$. 
For more general treatments see, e.g., 
\cite{ABerlinet_CThomas-Agnan_2004a}, \cite{VIPaulsen_MRaghupathi_2016a}, \cite[Sec.~4]{ISteinwart_AChristmann_2008a},
or \cite{HWendland_2004a}.

A symmetric kernel $\bk : \Omega \times \Omega \rightarrow \real$ is said to be positive definite if for all distinct $y_1, \dots, y_N \in \Omega$ the Gram matrix
\begin{equation} \label{eq:Surrogate:RKHS:kernel-matrix}
\bk(Y, Y) :=
	\begin{bmatrix} 
		\bk(y_1, y_1) & \cdots & \bk(y_1, y_N) \\
		\vdots & \ddots & \vdots \\
		\bk(y_N, y_1) & \cdots & \bk(y_N, y_N)
	\end{bmatrix}
\end{equation}
is symmetric positive semidefinite, and $\bk$ is said to be strictly positive definite if this matrix is symmetric positive definite.

Given a positive definite kernel $\bk$,  the unique RKHS with reproducing kernel $\bk$ is denoted $\cH_\bk(\Omega)$ and is
constructed by taking the completion of the pre-Hilbert space
$\cH_\bk^{\rm pre}(\Omega) := \set{ \sum_{i=1}^n a_i \bk(\cdot, y_i) }{  n \in \mathbb{N}, \, a_i \in \real, \, y_i \in \Omega  }$
with respect to the inner product 
\begin{equation} \label{eq:Surrogate:RKHS:RKHS-inner-product}
	\left\langle \sum_{i=1}^n a_i \bk(\cdot, y_i), \; \sum_{j=1}^m b_j \bk(\cdot, z_j) \right\rangle = \sum_{i=1}^n \sum_{j=1}^m a_i b_j \bk(y_i, z_j).
\end{equation}
 
 Models $\bg$ appearing in our simulation problem \eqref{eq:problem_formulation:IVP} or 
optimal control problem \eqref{eq:Sensitivity:OCP:OCP} need to be sufficiently smooth to enable
sensitivity analysis of the solution of these problems with respect to the model function, as discussed 
in \cref{sec:sensitivity}. 
Because these models will later be chosen from a RKHS $\cH_\bk(\Omega)$, we need results
about the smoothness of functions in $\cH_\bk(\Omega)$. These smoothness properties
are tied to the smoothness properties of the kernel $\bk$, as summarized in the next result.
We let $\alpha = (\alpha_1, \ldots, \alpha_d) \in \nat_0^d$ be a multi-index with $| \alpha | = \sum_{i=1}^d \alpha_i$,
and we use $D^\alpha = \partial_1^{\alpha_1} \ldots \partial_d^{\alpha_d}$ to denote the partial derivatives.
Moreover, for a kernel $\bk : \Omega \times \Omega \rightarrow \real$, $D_1^\alpha \bk(\cdot, \cdot)$ applies
partial derivatives with respect to the first argument and  $D_2^\alpha \bk(\cdot, \cdot)$ with respect to the second argument.

\begin{lemma} \label{lemma:Surrogate:RKHS:smoothness}
	Let $\bk : \Omega \times \Omega \rightarrow \real$ be a symmetric positive definite kernel. 
	If $D_1^\alpha D_2^\alpha \bk(\cdot, \cdot)$ exists and is continuous on $\Omega \times \Omega$ for all multi-indices $\alpha$ with $| \alpha | \leq s$,
	then the following hold:
	\begin{itemize}
		\item[(i)] All functions $\bg \in \cH_\bk(\Omega)$ are $s$-times continuously differentiable;
		
		\item[(ii)] $D_2^\alpha \bk(\cdot, y) \in \cH_\bk(\Omega)$ for all $y \in \Omega$ and all multi-indices $\alpha$ with $| \alpha | \leq s$;
		
		\item[(iii)] $D^\alpha \bg(y) = \langle \bg, D_2^\alpha \bk(\cdot, y) \rangle_{\cH_\bk(\Omega)}$ for all $\bg \in \cH_\bk(\Omega)$, all $y \in \Omega$, and all multi-indices $\alpha$ with $| \alpha | \leq s$.
	\end{itemize}
\end{lemma}
\begin{proof}
	See, e.g.,  \cite[Cor.~4.36]{ISteinwart_AChristmann_2008a} and \cite[Lemma~10.44, Thm.~10.45]{HWendland_2004a}.
\end{proof}

%%%%%%%%%%%%%%%%%%%%%%%%%%%%%%%%%%%%%%%%%%%%%%%%%%%%%%%%%%%
\subsection{Interpolation in Reproducing Kernel Hilbert Spaces}
Interpolation in RKHSs is studied in detail, e.g., in \cite[Ch.~8]{AIske_2018b}, \cite{VIPaulsen_MRaghupathi_2016a}, and \cite{HWendland_2004a}. We focus our discussion on the computation of the interpolant and pointwise error bounds.

Let $\cH_\bk(\Omega)$ be a RKHS with symmetric strictly positive definite kernel $\bk$.
Given points $Y = [y_1, \dots, y_N] \in \Omega^N$ and values $G = [g_1, \dots, g_N]^T \in \real^N$, consider the so-called optimal recovery problem
\[
	\min_{\bg \in \cH_\bk(\Omega)} \| \bg \|_{\cH_\bk(\Omega)} \qquad \mbox{s.t.} \quad \bg(y_i) = g_i, \qquad i = 1, \dots, N.
\]
The solution of this problem is the function in $\cH_\bk(\Omega)$ that interpolates $(Y, G)$ with the smallest RKHS norm, i.e., the minimum-norm interpolant, and is given by
\begin{align}   \label{eq:Surrogate:RKHS:kernel_interpolant_standard_form}
	\bg(y) = \bk(y, Y) \, \bk(Y, Y)^{-1} \, G, 
\end{align}
where $\bk(Y,Y)$ is the Gram matrix \eqref{eq:Surrogate:RKHS:kernel-matrix}
and $\bk(y, Y)$ is the row vector $[\bk(y, y_1) \ \cdots \ \bk(y, y_N)]$. Using \eqref{eq:Surrogate:RKHS:RKHS-inner-product}, the RKHS norm of the interpolant is 
\begin{equation}
	\| \bg \|_{\cH_\bk(\Omega)} = \sqrt{G^T \, \bk(Y, Y)^{-1} \, G}. \label{eq:Surrogate:RKHS:interpolant-RKHS-norm}
\end{equation}

In applications, the values $G$ often come from the evaluation of a function $\bg_*$ at the points $Y$, i.e., $G = [\bg_*(y_1), \dots, \bg_*(y_N)] \in \real^N$.
If $\bg_* \in \cH_\bk(\Omega)$, then the interpolant \eqref{eq:Surrogate:RKHS:kernel_interpolant_standard_form} is the orthogonal projection 
of $\bg_*$ onto the span of $\{ \bk(\cdot, y_i) \; | \; i = 1, \dots, N \}$ \cite[Thm.~13.1]{HWendland_2004a}.

The following result may be used to obtain a pointwise error bound for the partial derivatives of the kernel interpolant.
\begin{theorem} \label{thm:Surrogate:RKHS:derivative-error-bound}
    Let $\bk$ satisfy the assumptions of \cref{lemma:Surrogate:RKHS:smoothness}.
    If $\bg \in \cH_\bk(\Omega)$ is the kernel interpolant \eqref{eq:Surrogate:RKHS:kernel_interpolant_standard_form} 
    of $\bg_* \in \cH_\bk(\Omega)$ at the points $Y = [y_1, \dots, y_N] \in \Omega^N$ and values $G = [\bg_*(y_1), \dots, \bg_*(y_N)] \in \real^N$,
    then for any multi-index $\alpha$ with $|\alpha| \leq s$, the bound 
    \begin{equation}    \label{eq:Surrogate:RKHS:derivative-error-bound}
    	| D^\alpha \bg(y) - D^\alpha \bg_*(y)| \leq  \| \bg_* \|_{\cH_\bk(\Omega)} \, P_\bk^\alpha(y; Y), \quad \forall y \in \Omega
    \end{equation}
    holds, where the so-called power function $P_\bk^\alpha(y; Y)$ is given by
    \begin{align*}
    	P_\bk^\alpha(y; Y) 
    	&= \sqrt{D_1^\alpha D_2^\alpha \bk(y, y) - D_1^\alpha \bk(y, Y) \, \bk(Y, Y)^{-1} \, D_2^\alpha \bk(Y, y)}.
    \end{align*}
\end{theorem}
\begin{proof}
	The result follows from \cite[Thm.~11.4]{HWendland_2004a}.
\end{proof}
In the case $\alpha = (0, \dots, 0)$, \eqref{eq:Surrogate:RKHS:derivative-error-bound} reads
\begin{equation} \label{eq:Surrogate:RKHS:pointwise-term-0}
      | \bg(y) - \bg_*(y)| \leq  \| \bg_* \|_{\cH_\bk(\Omega)} \sqrt{\bk(y, y) - \bk(y, Y) \, \bk(Y, Y)^{-1} \, \bk(Y, y) } \quad \forall y \in \Omega.
\end{equation}
Note that $\bk(y, Y) \, \bk(Y, Y)^{-1} \, G$ in \eqref{eq:Surrogate:RKHS:kernel_interpolant_standard_form} and
$\bk(y, y) - \bk(y, Y) \, \bk(Y, Y)^{-1} \, \bk(Y, y)$ in \eqref{eq:Surrogate:RKHS:pointwise-term-0} are (resp.) equivalent to the mean
and the variance of a posterior Gaussian process conditioned on the points $Y$, which
further solidifies the connections between Gaussian process regression and kernel interpolation mentioned in~\cref{sec:intro}. 

Since $\| \bg_* \|_{\cH_\bk(\Omega)}$ is not known in practice, we introduce a computable estimate $\gamma > 0$ of $\| \bg_* \|_{\cH_\bk(\Omega)}$ (if no estimate is available, $\gamma = 1$), yielding
the following upper bound adapted from \eqref{eq:Surrogate:RKHS:derivative-error-bound}:
\begin{equation}    \label{eq:Surrogate:RKHS:derivative-error-bound1}
    	| D^\alpha \bg(y) - D^\alpha \bg_*(y)| \leq  \gamma^{-1}   \| \bg_* \|_{\cH_\bk(\Omega)} \,  \gamma P_\bk^\alpha(y; Y), \quad \forall y \in \Omega.
\end{equation}
Because $P_\bk^\alpha(y; Y)$ depends on  the kernel $\bk$, the points  $Y = [y_1, \dots, y_N] \in \Omega^N$, and $y \in \Omega$,
the part $\gamma P_\bk^\alpha(y; Y)$ of the upper bound \eqref{eq:Surrogate:RKHS:derivative-error-bound} can 
be evaluated efficiently.

The pointwise error bounds \eqref{eq:Surrogate:RKHS:derivative-error-bound1} involve two kernel-dependent terms 
$\| \bg_* \|_{\cH_\bk(\Omega)}$ and $P_\bk^\alpha(y; Y)$, and the choice of kernel impacts these bounds.
See, e.g., \cite{JLAkian_LBonnet_HOwhadi_ESavin_2022a} for the construction of kernels.
The fast approximate solution of large systems with kernel matrix $\bk(Y, Y)$ is described, e.g., in
\cite{FSchaefer_TJSullivan_HOwhadi_2021a} and \cite{YChen_HOwhadi_FSchaefer_2025a}.

%%%%%%%%%%%%%%%%%%%%%%%%%%%%%%%%%%%%%%%%%%%%%%%%%%%%%%%%%      
%%%%%%%%%%%%%%%%%%%%%%%%%%%%%%%%%%%%%%%%%%%
\section{Model Refinement} \label{sec:model_refinement}
This section uses the sensitivity results of \cref{sec:sensitivity} to develop the surrogate model adaptation approach.
We begin by specifying surrogate models and their error bounds needed for our model adaptation approach.
Kernel interpolation from \cref{sec:surrogates} is one option to construct such surrogates; we will
describe how the theory of kernel interpolation fits into the general setting, and later, in \cref{sec:numerics},
we will use kernel interpolation in our numerical examples.
However, the presentation in this section is agnostic to the surrogate modeling technique used, 
as other tools besides kernel interpolation could be used, so long as they meet the 
requirements specified in \cref{sec:model_refinement_models_errors} of the model refinement approach.

\begin{remark} \label{rk:Refinement:intro:compatibility}
    In \cref{sec:sensitivity}, the functions $\bg$ were assumed to be functions of $(t, x) \in I \times \real^{n_x}$ (simulation) and 
    $(t, x, u, p) \in I \times \real^{n_x} \times \real^{n_u} \times \real^{n_p}$ (optimization), 
    while in \cref{sec:surrogates} the functions $\bg$ belonging to the RKHS $\cH_\bk(\Omega)$ were defined over a domain $\Omega \subset \real^d$.
    There are multiple ways to ensure compatibility between these two contexts.
    To focus our discussion and simplify the notation in this section, we assume that $\Omega = I \times \real^{n_x}$ (simulation) and 
    $\Omega = I \times \real^{n_x} \times \real^{n_u} \times \real^{n_p}$ (optimization), and that $\bg: \Omega \rightarrow \real^{n_g}$.
    Moreover, the shorthand $y = (t, x)$ (simulation) or $y = (t, x, u, p)$ (optimization) is used instead of $y = x$ (simulation) and $y = (x, u, p)$ (optimization) as in \cref{sec:sensitivity}.
\end{remark}

%%%%%%%%%%%%%%%%%%%%%%%%%%%%%%%%%%%%%%%%%%%
\subsection{Models and Error Bounds}   \label{sec:model_refinement_models_errors}

Let  $\bg_* : \Omega \rightarrow \real^{n_g}$ be the true component function that is expensive to evaluate. 
Given points $Y = [y_1, \dots, y_N] \in \Omega^N$ and function values $G = [\bg_*(y_1), \dots, \bg_*(y_N)] \in \real^{N \times n_g}$, 
let  $\bg(\cdot \, ; Y, G)$ be a surrogate of $\bg_*$.
Where appropriate, the notation $\bg(y; Y, G)$ will be used for surrogates of $\bg_*$ to emphasize the dependence of the function $\bg$ both on its input $y \in \Omega$ and the interpolation points $(Y, G)$ used to construct it, with similar notation for surrogate error bounds.

Suppose that $(Y_c, G_c)$ represents the current surrogate $\bg_c$, i.e.,
\[
	\bg_c(\cdot) = \bg(\cdot \, ; Y_c, G_c),
\]
with componentwise error bound
\begin{align}   \label{eq:Refinement:intro:error-c}
	\big| \bg_i(y; Y_c, G_c) - (\bg_*)_i(y) \big| &\leq c_i \; \bepsilon_i(y; Y_c, G_c), & i = 1, \dots, n_g
\end{align}
and constants $c_i$.
To establish a notion of model refinement, define the augmented data set
\begin{align*}
	& Y_+ := [y_1, \dots, y_N, y_+] \in \Omega^{N+1}, && G_+ := [\bg_*(y_1), \dots, \bg_*(y_N), \bg_*(y_+)] \in \real^{(N+1) \times n_g},
\end{align*}
so that $\bg_+(\cdot) = \bg(\cdot \, ; Y_+, G_+)$ is the surrogate obtained by adding the point \big($y_+$, $\bg_*(y_+)\big)$. 
To simplify the presentation, we add one point, but it is easily possible to add more than one point.
The error bound   \eqref{eq:Refinement:intro:error-c} with $Y_c$, $G_c$ replaced by $Y_+$, $G_+$ reads
\begin{align}    \label{eq:Refinement:intro:error-+}
	\big| \bg_i(y; Y_+, G_+) - (\bg_*)_i(y) \big| &\leq c_i \, \bepsilon_i(y; Y_+, G_+), & i = 1, \dots, n_g.
\end{align}
The error bound \eqref{eq:Refinement:intro:error-+} depends on $G_+$, which in turn depends on $\bg_*(y_+)$.
Because evaluating $\bg_*$ is expensive, we need to avoid computing $\bg_*(y_+)$ at trial points. Therefore,
we assume the availability of an error bound that is independent of $\bg_*(y_+)$, i.e., we assume an error bound of the type
\begin{align}  \label{eq:Refinement:intro:error-indicator}
	\big| \bg_i(y; Y_+, G_+) - (\bg_*)_i(y)  \big| &\leq  c_i \,  \bepsilon_i(y; Y_+, G_c), & i = 1, \dots, n_g.
\end{align}
In the optimization setting, we  also require a bound like \eqref{eq:Refinement:intro:error-indicator}   
for the partial derivatives, i.e, we assume that 
\begin{align} \label{eq:Refinement:intro:error-indicator-derivatives}
	\big| D^\alpha \bg_i(y; Y_+, G_+) - D^\alpha (\bg_*)_i(y) \big| \leq  c^\alpha_i \,  \bepsilon^\alpha_i(y; Y_+, G_c),
	 \quad |\alpha| \leq 1, \quad  i = 1, \dots, n_g.
\end{align}
We refer to \eqref{eq:Refinement:intro:error-indicator} and \eqref{eq:Refinement:intro:error-indicator-derivatives} as post-refinement error bounds.
Note that $D^\alpha$ denotes partial derivatives, whereas $c^\alpha_i$ and $\bepsilon^\alpha_i$ denote constants and pointwise terms that represent the error in the $D^\alpha$ derivatives of the interpolant; they are {\em not} the 
$D^\alpha$  derivatives of $c_i$ and $\bepsilon_i$  in \eqref{eq:Refinement:intro:error-indicator}.

Our surrogate model adaptation approach relies on the existence of post-refinement error bounds  \eqref{eq:Refinement:intro:error-indicator} (simulation)
and \eqref{eq:Refinement:intro:error-indicator-derivatives} (optimization) with functions $\bepsilon_i$, $i = 1, \ldots, n_g$, 
and $\bepsilon^\alpha_i$, $i = 1, \ldots, n_g$, $|\alpha| \leq 1$, that can be computed relatively inexpensively for points 
$Y_+ \in  \Omega^{N+1}$.
Before we continue with the discussion of our surrogate model adaptation approach, 
we comment on how kernel interpolation from \cref{sec:surrogates} 
can be used to construct  surrogate models $ \bg_i(\cdot \, ; Y_c, G_c)$ and $ \bg_i(\cdot \, ; Y_+, G_+)$.

To apply the sensitivity results of \cref{sec:sensitivity} to RKHS-based surrogates
$\bg \in \big(\cH_\bk(\Omega)\big)^{n_g}$,
we require the space $\big(\cH_\bk(\Omega)\big)^{n_g}$ to be contained in $\cG^s$ for $s = 2$ (simulation) or $s = 3$ (optimization). 
The following result ensures the continuous embedding 
of $\big(\cH_\bk(\Omega)\big)^{n_g}$ in $\cG^s$. It suffices to consider the case $n_g = 1$.

\begin{assumption} \label{as:surrogates:kernel-requirements}
	Let $\bk : \Omega \times \Omega \rightarrow \real$ be a symmetric positive definite kernel with the following properties:
		\begin{itemize}
		\item[(i)] $D_1^\alpha D_2^\alpha \bk(\cdot, \cdot)$ exists and is continuous on $\Omega \times \Omega$ for all multi-indices $\alpha$ with $| \alpha | \leq s$;
		
		\item[(ii)] The mapping $y \mapsto D_1^\alpha D_2^\alpha \bk(y, y)$ is bounded on $\Omega$ 
		               for all multi-indices $\alpha$ with $| \alpha | \leq s$.
	\end{itemize}
\end{assumption}
\begin{remark}
     In the case $n_g > 1$, the individual components of $\bg$ need not be defined by the same kernel, e.g., 
     they may have kernels with different hyperparameters or may have different kernels entirely, so long as 
     they individually satisfy the requirements of \cref{as:surrogates:kernel-requirements}.
\end{remark}
\begin{theorem} \label{thm:surrogates:Hk-and-G2}
	If the the kernel $\bk$ satisfies \cref{as:surrogates:kernel-requirements}, then $\cH_\bk(\Omega)$ is continuously embedded in $\cG^s$, i.e., 
	$\cH_\bk(\Omega) \subset \cG^s(I)$ and there exists $C > 0$ such that
	$\| \bg \|_{\cG^s} \leq C \| \bg \|_{\cH_\bk(\Omega)}$  for all $\bg \in \cH_\bk(\Omega)$.
\end{theorem}
\begin{proof}
	Let $\bg \in \cH_\bk(\Omega)$. By \cref{as:surrogates:kernel-requirements} (i) and  \cref{lemma:Surrogate:RKHS:smoothness} (i), $\bg$ is $s$-times continuously differentiable. 
	Moreover, due to \cref{as:surrogates:kernel-requirements} (ii), we may apply \cref{lemma:Surrogate:RKHS:smoothness} (iii) and the Cauchy-Schwarz inequality to obtain
	\begin{align*}
		| D^\alpha \bg(y) | 
		&= | \langle \bg, D_2^\alpha \bk(\cdot, y) \rangle_{\cH_\bk(\Omega)} | 
		 \leq \| \bg \|_{\cH_\bk(\Omega)} \| D_2^\alpha \bk(\cdot, y) \|_{\cH_\bk(\Omega)},  \qquad \qquad | \alpha | \leq s,
	\end{align*}
	which by \eqref{eq:Sensitivity:caratheodory:g-space-norm} and repeated use of \cref{lemma:Surrogate:RKHS:smoothness} (iii)
	implies the result
	with $C = \sum_{| \alpha | \leq s} \sqrt{\sup_{z \in \Omega} | D_1^\alpha D_2^\alpha \bk(z, z) |}$.
\end{proof}

If $(\bg_*)_i \in  \cH_\bk(\Omega)$  and if $\bg_i(\cdot \, ; Y_+, G_+) \in  \cH_\bk(\Omega)$ is the kernel interpolant \eqref{eq:Surrogate:RKHS:kernel_interpolant_standard_form} 
with $Y, G$ replaced by $Y_+, G_+$, then the post-refinement error bound  \eqref{eq:Refinement:intro:error-indicator} is satisfied with
\begin{equation}  \label{eq:Refinement:intro:error-indicator-RKHS}
\begin{aligned}
         c_i &=  \gamma_i^{-1} \| (\bg_*)_i \|_{\cH_\bk(\Omega)},  \\
	 \bepsilon_i(y; Y_+, G_c) &=  \gamma_i \, \sqrt{\bk(y, y) - \bk(y, Y_+) \, \bk(Y_+, Y_+)^{-1} \, \bk(Y_+, y) },
\end{aligned}
\end{equation}
where  $ \gamma_i$ is a computationally inexpensive estimate of $\| (\bg_*)_i \|_{\cH_\bk(\Omega)}$ 
or simply $ \gamma_i = 1$;
see \eqref{eq:Surrogate:RKHS:pointwise-term-0} and \eqref{eq:Surrogate:RKHS:derivative-error-bound1}.
Moreover, if \cref{as:surrogates:kernel-requirements} is satisfied with $s=1$, 
then the post-refinement error bound  \eqref{eq:Refinement:intro:error-indicator-derivatives} is satisfied with
\begin{equation}  \label{eq:Refinement:intro:error-indicator-drivative-RKHS}
\begin{aligned}
       c^\alpha_i &=  \gamma_i^{-1} \| (\bg_*)_i \|_{\cH_\bk(\Omega)},  \\
	 \bepsilon^\alpha_i(y; Y_+, G_c) 
	& =  \gamma_i \, \sqrt{D_1^\alpha D_2^\alpha \bk(y, y) - D_1^\alpha \bk(y, Y_+) \, \bk(Y_+, Y_+)^{-1} \, D_2^\alpha \bk(Y_+, y)},
\end{aligned}
\end{equation}
where, again,  $ \gamma_i$ is a computationally inexpensive estimate of $\| (\bg_*)_i \|_{\cH_\bk(\Omega)}$ or simply $ \gamma_i = 1$;
see \eqref{eq:Surrogate:RKHS:derivative-error-bound} and  \eqref{eq:Surrogate:RKHS:derivative-error-bound1}.
The functions $y \mapsto  \bepsilon_i(y; Y_+, G_c) $ in \eqref{eq:Refinement:intro:error-indicator-RKHS} and 
$y \mapsto  \bepsilon^\alpha_i(y; Y_+, G_c)$ in \eqref{eq:Refinement:intro:error-indicator-drivative-RKHS} 
depend only on the kernel $\bk$, not on the expensive $\bg_*$. 
Moreover, because $Y_+$ differs from $Y_c$ by one column (or a few columns if more than one
point is added), computations previously done with $\bk(Y_c, Y_c)$ and similar kernel-dependent vectors can be updated to compute
\eqref{eq:Refinement:intro:error-indicator-RKHS} and \eqref{eq:Refinement:intro:error-indicator-drivative-RKHS} efficiently in large-scale settings.

%%%%%%%%%%%%%%%%%%%%%%%%%%%%%%%%%%%%%%%%%%%%%%%%%%%%%%%%%%%%%
\subsection{Simulation}   \label{sec:Refinement:ODE}

This section discusses model refinement for simulation problems of the form \eqref{eq:problem_formulation:IVP}.
We assume that the current model  $\bg_c(\cdot) = \bg(\cdot \, ; Y_c, G_c) \in \cG^2$ is given
and we want to select $y_+ \in \Omega$ to construct a new model 
$\bg_+(\cdot) = \bg(\cdot \, ; Y_+, G_+) \in \cG^2$.
If kernel interpolation-based models are used, we assume that the kernel satisfies 
\cref{as:surrogates:kernel-requirements} with $s=2$,
so that $\bg_*$, $\bg_c(\cdot) = \bg(\cdot \, ; Y_c, G_c)$, $\bg_+(\cdot) = \bg(\cdot \, ; Y_+, G_+) \in \big(\cH_\bk(\Omega)\big)^{n_g} \subset \cG^2$.

Given $\bg_c \in \cG^2$, the corresponding solution of the IVP \eqref{eq:problem_formulation:IVP} is $\bx_c := \bx(\cdot \, ; \bg_c) \in \WW^{n_x}$
and we set 
\[
	\by_c(\cdot) := \big( \cdot, \bx_c(\cdot)\big),
\]
cf. \cref{rk:Refinement:intro:compatibility}. 
To declutter notation, similar to \eqref{eq:Sensitivity:ODE:shorthand} and \eqref{eq:Sensitivity:ODE:shorthand-2},
we use the shorthand
\begin{align*}
		\bff_c[\cdot] &:= \bff\Big(\cdot, \bx_c(\cdot), \bg_c\big( \cdot, \bx_c(\cdot)\big) \Big), &
	         \bg_c[\cdot] &:= \bg_c \big( \cdot, \bx_c(\cdot) \big),  \\
		\BA_c[\cdot] &:= (\bff_c)_x[\cdot] + (\bff_c)_g[\cdot] (\bg_c)_x[\cdot],  &
		\BB_c[\cdot] &:= (\bff_c)_g[\cdot], \\
	\ell_c[\cdot] &:=  \ell\Big( \cdot, \bx_c(\cdot), \bg_c \big( \cdot, \bx_c(\cdot) \big) \Big), & \qquad
    	\phi_c[t_f] &:= \phi\big( \bx_c(t_f)\big).
\end{align*}

%%%%%%%%%%%%%%%%%%%%%%%%%%%%%%%%%%%%%%%%%%%%%%%%%%
Consider the error in the QoI \eqref{eq:problem_formulation:QoI-ODE-g-only}.
We want to compute a new model $\bg_+[y_+])$ so that the error 
$|\widetilde{q}(\bg_+[y_+]) - \widetilde{q}(\bg_*)|$ is reduced.
We approximate this error using sensitivities by
\begin{equation} \label{eq:Refinement:ODE:approx-QoI-error}
	|\widetilde{q}(\bg_+[y_+]) - \widetilde{q}(\bg_*)| 
	\approx | \widetilde{q}_\bg(\bg_+[y_+])(\bg_+[y_+] - \bg_*) |
	\approx | \widetilde{q}_\bg(\bg_c)(\bg_+[y_+] - \bg_*) |.
\end{equation}
Because $\bg_+[y_+] - \bg_*$ satisfies the post-refinement error indicator \eqref{eq:Refinement:intro:error-indicator}, 
and the QoI sensitivity is computed using  \cref{thm:Sensitivity:ODE:QoI-sensitivity},
an upper bound for the right-hand side in \eqref{eq:Refinement:ODE:approx-QoI-error} is
given by the optimal value of 
\begin{equation} \label{eq:Refinement:ODE:OCP-QoI0}
\begin{aligned}
	\sup_{\delta \bg \in (\cG^2(I) )^{n_g}} \quad 
	&\left|  \int_{t_0}^{t_f} \big( \BB_c[t]^T \blambda_c(t) + \nabla_g \ell_c[t] \big)^T 
	                                                    \delta \bg \big( t, \bx_c(t) \big) \, dt,\right| \\
	\mbox{s.t.} \quad 
	&  \big|  \delta \bg \big( t, \bx_c(t) \big) \big|
	      \leq c_i \bepsilon_i \big( \by_c(t); Y_+, G_c \big), \quad \aall t \in I, \quad  i = 1, \dots, n_g,
\end{aligned}
\end{equation}
where $\blambda_c$ solves \eqref{eq:Sensitivity:ODE:x-adjoint} with $\overline{\blambda}, \overline{\BA}, \overline{\ell}, \overline{\phi}, \overline{\bg}$ replaced by $\blambda_c, \BA_c, \ell_c, \phi_c, \bg_c$.

To obtain an easily computable upper bound, we relax \eqref{eq:Refinement:ODE:OCP-QoI0} further
and replace the state $\bx_c$-dependent $\delta \bg \in \cG^2$ by
$\bdelta \in \big( L^\infty(I) \big)^{n_g}$. Furthermore, we drop the constants $c_i$.
Thus, we compute  a scaled upper bound for the right-hand side in \eqref{eq:Refinement:ODE:approx-QoI-error} 
as the optimal function value of 
\begin{equation} \label{eq:Refinement:ODE:OCP-QoI}
\begin{aligned}
	\max_{\bdelta} \quad &\left| \int_{t_0}^{t_f} \big( \BB_c[t]^T \blambda_c(t) + \nabla_g \ell_c[t] \big)^T \bdelta(t) \, dt \right| \\
	\mbox{s.t.} \quad 
	& - \bepsilon\big( \by_c(t); Y_+, G_c \big) \leq \bdelta(t) \leq \bepsilon\big( \by_c(t); Y_+, G_c \big), & \qquad \aall t \in I,
\end{aligned}
\end{equation}
where the inequality costraints are understood componentwise, and
where $\blambda_c$ again 
solves \eqref{eq:Sensitivity:ODE:x-adjoint} with $\overline{\blambda}, \overline{\BA}, \overline{\ell}, \overline{\phi}, \overline{\bg}$ replaced by $\blambda_c, \BA_c, \ell_c, \phi_c, \bg_c$.
Due to the symmetry of the box constraints and the linearity of the objective function, the absolute value may be dropped in the objective function. The resulting linear program has a simple analytical solution, where $\odot$ denotes componentwise multiplication:
\[
	\bdelta(t) = \mathrm{sgn} \big( \BB_c[t]^T \blambda_c(t) + \nabla_g \ell_c[t] \big) \odot \bepsilon\big( \by_c(t); Y_+, G_c \big), \qquad \aall t \in I,
\]
with optimal objective value
\begin{equation} \label{eq:Refinement:ODE:OCP-QoI-aquisition-function}
	 V^{\rm QoI}( \by_c; Y_+, G_c \big)  
	 = \int_{t_0}^{t_f} \big| \BB_c[t]^T \blambda_c(t) + \nabla_g \ell_c[t] \big|^T \bepsilon\big( \by_c(t); Y_+, G_c \big) \, dt,
\end{equation}
where the absolute value is applied componentwise.
Thus, the following result is proven.
\begin{theorem} \label{thm:Sensitivity:ODE:QoI-error-bound}
        If the assumptions of \cref{thm:Sensitivity:ODE:QoI-sensitivity} and  \eqref{eq:Refinement:intro:error-indicator} hold, 
        then the approximate QoI error measure \eqref{eq:Refinement:ODE:approx-QoI-error} satisfies the bound
	\[
		| \widetilde{q}_\bg(\bg_c)(\bg_+[y_+] - \bg_*) | \leq c \,  V^{\rm QoI}( \by_c; Y_+, G_c \big),
	\]
        and $c = \max\{ c_1, \ldots, c_{n_g}\}$, where $c_1, \ldots, c_{n_g}$ are the constants in  \eqref{eq:Refinement:intro:error-indicator}.
\end{theorem}
\cref{thm:Sensitivity:ODE:QoI-error-bound} is the adaptive equivalent of \cite[Thm.~3.6]{JRCangelosi_MHeinkenschloss_2025c}.
Since the $c_i$ are dropped from \eqref{eq:Refinement:intro:error-indicator} in the formulation of the box constraints in \eqref{eq:Refinement:ODE:OCP-QoI}, it is beneficial to have constants $c_1 \approx  \ldots \approx c_{n_g} \approx 1$,
 which in the kernel interpolation case is achieved when $\gamma_i  \approx \| (\bg_*)_i \|_{\cH_\bk(\Omega)}$,
 $i = 1, \ldots, n_g$; see \eqref{eq:Refinement:intro:error-indicator-RKHS}.
	
The new refinement point $y_+$ is selected to minimize the acquisition function 
\eqref{eq:Refinement:ODE:OCP-QoI-aquisition-function} over a chosen set of candidates $Y_{\rm cand} \subset \Omega$, 
i.e.,  $y_+$ and the corresponding $Y_+ := [y_1, \dots, y_N, y_+]$ are computed as a solution of
\[
	\min_{y_+ \in Y_{\rm cand}} \;  V^{\textrm{QoI}}( \by_c; Y_+, G_c \big).
\]

%%%%%%%%%%%%%%%%%%%%%%%%%%%%%%%%%%%%%%%%%%%%%%%%%%%%%%%%%
\subsection{Optimization}   \label{sec:Refinement:OCP}
This section discusses model refinement for optimization problems of the form  \eqref{eq:Sensitivity:OCP:OCP}. Here the set $\Omega = I \times \real^{n_x} \times \real^{n_u} \times \real^{n_p}$ denotes time-state-control-parameter space.
We assume that the current model  $\bg_c(\cdot) = \bg(\cdot \, ; Y_c, G_c) \in \cG^3$
and we want to select $y_+ \in \Omega$ to obtain a new model $\bg_+(\cdot) = \bg(\cdot \, ; Y_+, G_+) \in \cG^3$.
If kernel interpolation-based models are used, we assume that the kernel satisfies \cref{as:surrogates:kernel-requirements} with $s=3$,
so that $\bg_*$, $\bg_c(\cdot) = \bg(\cdot \, ; Y_c, G_c)$, $\bg_+(\cdot) = \bg(\cdot \, ; Y_+, G_+) \in \big(\cH_\bk(\Omega)\big)^{n_g} \subset \cG^3$.

Given $\bg_c \in \cG^3$, the corresponding solution  of \eqref{eq:Sensitivity:OCP:OCP}  is $(\bx_c, \bu_c, \bp_c) \in \WW^{n_x} \times \LL^{n_u} \times \real^{n_p}$
with corresponding costate $\blambda_c \in \WW^{n_x}$.
To declutter notation, we use
\[
	\by_c(\cdot) := \big( \cdot, \bx_c(\cdot), \bu_c(\cdot), \bp_c \big)
\]
and the shorthand
\begin{align*}
		\BH_c[\cdot], \ldots, \bd_c[\cdot], \ell_c[\cdot], \phi_c[t_f],
\end{align*}
analogous to \eqref{eq:Sensitivity:OCP:shorthand} and \eqref{eq:Sensitivity:OCP:shorthand-1}, 
but with functions evaluated at $\bg_c, \bx_c, \bu_c, \bp_c, \blambda_c$
instead of at $\overline{\bg}, \overline{\bx}, \overline{\bu}, \overline{\bp}, \overline{\blambda}$.

%%%%%%%%%%%%%%%%%%%%%%%%%%%%%%%%%%%%%%%%%%%%%%%%%%%%%%%%%
Consider the QoI \eqref{eq:problem_formulation:QoI-OCP-g-only}.
A similar strategy to the one in \cref{sec:Refinement:ODE} may be applied to reduce the error in the QoI \eqref{eq:problem_formulation:QoI-OCP-g-only} 
as a function of the solution to the OCP \eqref{eq:Sensitivity:OCP:OCP}. 
Using the OCP sensitivity result from \eqref{eq:Sensitivity:OCP:sensitivity_result_qoi_adjoint} with the post-refinement error bound (including first derivatives) 
given by \eqref{eq:Refinement:intro:error-indicator-derivatives}, one obtains an upper bound on the approximate QoI error measure
\begin{equation} \label{eq:Refinement:OCP:approx-QoI-error}
	|\widetilde{q}(\bg_+[y_+]) - \widetilde{q}(\bg_*)| 
	\approx | \widetilde{q}_\bg(\bg_+[y_+])(\bg_+[y_+] - \bg_*) |
	\approx | \widetilde{q}_\bg(\bg_c)(\bg_+[y_+] - \bg_*) |
\end{equation}
as the solution of the optimization problem
     \begin{equation} \label{eq:Refinement:OCP:OCP-QoI}
\begin{aligned}
	\sup_{\delta \bg} \quad &\Big| \int_{t_0}^{t_f} \bh_c[t]^T  \delta \bg[t] 
	     + \bd_c[t]^T \delta \bg_x[t] \widetilde{\delta \bx}(t)  
                                             +  \bd_c[t]^T \delta \bg_u[t] \widetilde{\delta \bu}(t)  
                                             +  \bd_c[t]^T \delta \bg_p[t] \widetilde{\delta \bp}   \, dt	   \\
	\mbox{s.t.} \quad 
	&  \big| \delta \bg_i \big( \by_c(t) \big) \big| \leq c_i \bepsilon_i \big( \by_c(t); Y_+, G_c \big), 
	   \qquad \aall t \in I,  \;  i = 1, \dots, n_g, \\
	&  \big|  (\delta \bg_i)_z\big( \by_c(t) \big) \big|  \leq c_i^z \bepsilon^z\big( \by_c(t); Y_+, G_c \big), 
	      \;  \aall t \in I, \;  i = 1, \dots, n_g, \; z \in \{ x, u, p \}, 
\end{aligned}
\end{equation}
where
\[
     \bh_c[t] :=  (\BH_c)_{gx}[t]  \widetilde{\delta \bx}(t) 
                                                  + (\BH_c)_{gu}[t]  \widetilde{\delta \bu}(t)
                                                  + (\BH_c)_{gp}[t]  \widetilde{\delta \bp}
                                                  + (\bff_c)_{g}[t]^T \widetilde{\delta \blambda}(t) 
                                                  +   \nabla_g \ell_c[t].
\]
In addition to replacing $\delta \bg\big( \by_c(\cdot) \big)$ by $\bdelta(\cdot)$ as in \cref{sec:Refinement:ODE}, 
the derivatives 
$\delta \bg_x\big( \by_c(\cdot) \big) \in \LL^{n_g \times n_x}$,
$\delta \bg_u\big( \by_c(\cdot) \big) \in \LL^{n_g \times n_u}$,
$\delta \bg_p\big( \by_c(\cdot) \big) \in \LL^{n_g \times n_p}$
are similarly replaced by functions
\[
	\bdelta^x \in \LL^{n_g \times n_x}, \quad \bdelta^u \in \LL^{n_g \times n_u}, \quad \bdelta^p \in \LL^{n_g \times n_p}
\]
respectively to obtain 
\begin{align}  \label{eq:Refinement:OCP:OCP-QoI-delta}
	\max_{\bdelta, \bdelta^x, \bdelta^u, \bdelta^p} \quad &\Big| \int_{t_0}^{t_f}  \bh_c[t]^T  \bdelta(t) 
	                                    + \bd_c[t]^T \bdelta^x(t) \widetilde{\delta \bx}(t)
                                             +   \bd_c[t]^T \bdelta^u(t) \widetilde{\delta \bu}(t) \
                                             +   \bd_c[t]^T \bdelta^p(t) \widetilde{\delta \bp}  \, dt	  \Big|   \nonumber \\ 
      \mbox{s.t.} \quad 
	& - \bepsilon\big( \by_c(t); Y_+, G_c \big) \leq \bdelta(t) \leq \bepsilon\big( \by_c(t); Y_+, G_c \big),  \qquad  \aall t \in I,  \\
	& - \bepsilon^z\big( \by_c(t); Y_+, G_c \big) \leq \bdelta^z(t) \leq \bepsilon^z\big( \by_c(t); Y_+, G_c \big), \; \aall t \in I, 
	\; z \in \{ x, u, p \}, \nonumber 
\end{align}
where the box constraints are understood elementwise.

Note that $\bdelta^x$, $\bdelta^u$, $\bdelta^p$ in \eqref{eq:Refinement:OCP:OCP-QoI-delta}
are not derivatives of $\bdelta$; they are independent functions, which is why they are indicated by superscripts instead of subscripts. The same is true for $\bepsilon^x$, $\bepsilon^u$, $\bepsilon^p$. Analogously to \cref{sec:Refinement:ODE}, \eqref{eq:Refinement:OCP:OCP-QoI-delta} has an analytical solution:
\begin{align*}
	\bdelta(t) &= \mathrm{sgn}\big(  \bh_c[t]  \big) \odot \bepsilon\big( \by_c(t); Y_+, G_c \big), \\
	\bdelta_{ij}^x(t) &= \mathrm{sgn}\big((\bd_c)_i[t] \widetilde{\delta \bx}_j(t)\big) \, \bepsilon_{ij}^x\big( \by_c(t); Y_+, G_c \big), & i = 1, \dots, n_g, \quad j = 1, \dots, n_x, \\
	\bdelta_{ij}^u(t) &= \mathrm{sgn}\big((\bd_c)_i[t] \widetilde{\delta \bu}_j(t)\big) \, \bepsilon_{ij}^u\big( \by_c(t); Y_+, G_c \big), & i = 1, \dots, n_g, \quad j = 1, \dots, n_u, \\
	\bdelta_{ij}^p(t) &= \mathrm{sgn}\big((\bd_c)_i[t] \widetilde{\delta \bp}_j\big) \, \bepsilon_{ij}^p\big( \by_c(t); Y_+, G_c \big),    
	& i = 1, \dots, n_g, \quad j = 1, \dots, n_p,
\end{align*}
with corresponding objective value
\begin{align}  \label{eq:Refinement:OCP:OCP-acquisition-function}
\hspace*{5ex}
	   & V^{\rm QoI}( \by_c; Y_+, G_c \big)  \nonumber  \\
	   &= \int_{t_0}^{t_f}
	              \big|   \bh_c[t]  \big|^T  \bepsilon\big( \by_c(t); Y_+, G_c \big) 
	              +  \big| \bd_c[t] \big|^T \bepsilon^x\big( \by_c(t); Y_+, G_c \big) \, \big| \widetilde{\delta \bx}(t) \big|   \nonumber  \\
            & \qquad  \quad +    \big|  \bd_c[t] \big|^T \bepsilon^u\big( \by_c(t); Y_+, G_c \big) \, \big| \widetilde{\delta \bu}(t) \big| 
                                      +  \big|  \bd_c[t] \big|^T \bepsilon^p\big( \by_c(t); Y_+, G_c \big) \, \big| \widetilde{\delta \bp} \big|  \, dt, 
\end{align}
where the absolute value is applied componentwise. Thus, the following result is proven.
\begin{theorem} \label{thm:Refinement:OCP:OCP-acquisition-function}
	If the assumptions of \cref{thm:Sensitivity:OCP:sensitivity_result_qoi-adjoint} and 
	 \eqref{eq:Refinement:intro:error-indicator-derivatives} hold, 
	 then the approximate QoI error measure \eqref{eq:Refinement:OCP:approx-QoI-error} satisfies the bound
	\[
	       | \widetilde{q}_\bg(\bg_c)(\bg_+[y_+] - \bg_*) |  \nonumber \\ 
		\leq c \, V^{\rm QoI}( \by_c; Y_+, G_c \big),
	\]
          with $c = \max_{|\alpha| \le 1, i \in \{1, \ldots, n_g\}} c^\alpha_i$, where $c^\alpha_i$, $i \in \{ 1, \ldots, n_g \}$, $|\alpha| \le 1$,
            are the constants in \eqref{eq:Refinement:intro:error-indicator-derivatives}.
\end{theorem}

 Since the $c_i^\alpha$ are dropped from \eqref{eq:Refinement:intro:error-indicator-derivatives} in the formulation of the box constraints for \eqref{eq:Refinement:OCP:OCP-QoI-delta},
it is beneficial to have constants 
 $c^\alpha_i \approx 1$, $i \in \{ 1, \ldots, n_g \}$, $|\alpha| \le 1$,
 which in the kernel interpolation case is achieved when $\gamma_i  \approx \| (\bg_*)_i \|_{\cH_\bk(\Omega)}$,
 $i = 1, \ldots, n_g$; see \eqref{eq:Refinement:intro:error-indicator-drivative-RKHS}.
 
The refinement point $y_+$ is chosen among a set of candidates $Y_{\rm cand}$ to minimize the acquisition function \eqref{eq:Refinement:OCP:OCP-acquisition-function}.

%%%%%%%%%%%%%%%%%%%%%%%%%%%%%%%%%%%%%%%%%%%%%%%%%%%%%%%%%
%%%%%%%%%%%%%%%%%%%%%%%%%%%%%%%%%%%%%%%%%%%%%%%%%%%%%%%%%      
\section{Simulation and Trajectory Optimization of a Hypersonic Vehicle} \label{sec:numerics}
In this section, we present numerical results for the simulation and trajectory optimization of a hypersonic vehicle.
The vehicle dynamics will be specified in \cref{sec:numerics_ODE}. The lift, drag, and moment coefficient functions 
will play the role of the component function $\bg$, and they will be functions of two variables.
See \eqref{eq:Refinement:ODE:true-model} below.
Because we want to compare the results obtained with surrogates
to the ones obtained with the true models, we assume that the true lift, drag, and moment coefficient functions are
given by polynomial models (specified in \eqref{eq:Refinement:ODE:true-model}). 

We use radial kernel interpolants for $C_L$, $C_D$, $C_M$ constructed with the Wendland radial functions.
Recall that in our application, the component function $\bg$ will be a function of two variables.
We use $\bk(y, z) = \Psi_{3,1}(\| y - z \|_2 / \ell)$ for simulation and $\bk(y, z) = \Psi_{3,2}(\| y - z \|_2 / \ell)$ 
for optimization where the $C^2$ function $\Psi_{3,1}$ and the $C^4$ function $\Psi_{3,2}$ are given by
\begin{equation}   \label{eq:wendland-kernel}
	 \Psi_{3,1}(r) = \big((1-r)_+\big)^4 (4r+1), \quad
	 \Psi_{3,2}(r) = \big((1-r)_+\big)^6 (35r^2 + 18r+ 3),
\end{equation}
for $r \ge 0$ and are extended on the whole real line by even extension.
Here  $z_+ = \max\{ z, 0 \}$.
See \cite{HWendland_1995a}, \cite[Table~9.1]{HWendland_2004a} for details on these
functions. We use the lengthscale $\ell = 3$. 
These kernels give better results than Gaussian kernels used in \cite{JRCangelosi_2025a}.
Wendland kernels, unlike the Gaussian kernel, are compactly supported, making them well-suited for local corrections. 
Since we do not adjust the lengthscale, we chose it large enough to cover the domain of interest, 
so the samples never lie outside the support of the kernel in this case, but reducing the lengthscale 
could introduce sparsity into the Gram matrix $\bk(Y, Y)$.

%%%%%%%%%%%%%%%%%%%%%%%%%%%%%%%%%%%%%%%%%%%%%%%%%%%%%%%%%      
\subsection{Simulation} \label{sec:numerics_ODE}
We use
the adaptive model refinement procedure proposed in \cref{sec:Refinement:ODE} to solve an IVP for a notional hypersonic vehicle in longitudinal flight. 
The control surface of the vehicle is an elevator (or flap) that makes an angle $\delta$ [rad] with the chord line of the vehicle. This angle may be adjusted to influence the pitch of the vehicle. We use the second derivative of $\delta$ as a given input, making $\delta$ and its first derivative $\dot{\delta}$ states.
The states are downrange $x_1$ [m], altitude $x_2$  [m], speed $v$  [m/s], flight path angle $\gamma$  [rad], 
angle of attack $\alpha$ [rad], pitch rate $q$ [rad/s], flap angle $\delta$ [rad], and flap rate $\dot\delta$ [rad/s], i.e.,
in this example,
\begin{equation} \label{eq:Refinement:ODE:hypersonic-states}
       \bx(t) =  \big( \bx_1(t), \bx_2(t), \bv(t), \bgamma(t), \balpha(t), \bq(t), \bdelta(t), \dot{\bdelta}(t) \big).
\end{equation}
Note that in this context, $\bdelta$ is a physical quantity and is not related to the
$\bdelta$ in  \eqref{eq:Refinement:ODE:OCP-QoI} or \eqref{eq:Refinement:OCP:OCP-QoI-delta}.
Moreover, $\bdelta(t)$ and $\dot{\bdelta}(t)$ denote state functions, and $\ddot{\bdelta}(t)$  denotes the input.
We do not use $\dot{ \; }$ for differentiation with respect to time, which is denoted by $'$.

The dynamics of the hypersonic vehicle are given by 
\begin{subequations}  \label{eq:Refinement:ODE:hypersonic_IVP}
\begin{align} 
    \bx_1'(t) &= \bv(t) \cos \bgamma(t), \\
    \bx_2'(t) &= \bv(t) \sin \bgamma(t), \\
    \bv'(t) &= -\frac{1}{m} \Big( D\big(\bx_2(t), \bv(t), \balpha(t), \bdelta(t) \big) + m g\big(\bx_2(t)\big) \sin \bgamma(t) \Big), \\
    \bgamma'(t) &= \frac{1}{m \bv(t)} \Big( L \big(\bx_2(t), \bv(t), \balpha(t), \bdelta(t) \big) - mg \big(\bx_2(t) \big) \cos \bgamma(t)  \\ 
                        & \hspace*{40ex} + \frac{m \bv(t)^2 \cos \bgamma(t)}{R_E + \bx_2(t)} \Big),   \nonumber  \\
    \balpha'(t) &= \bq(t) - \bgamma'(t), \\
    \bq'(t) &= M \big(\bx_2(t), \bv(t), \balpha(t), \bdelta(t) \big) \; / \; I_z, \\
    \bdelta'(t) &=  \dot\bdelta(t), \quad
    \dot\bdelta'(t) =  \ddot\bdelta(t),     \label{eq:Refinement:ODE:hypersonic_IVP-delta}
\end{align}
\end{subequations}
for $t \in (0,T)$,
with initial conditions
\begin{equation} \label{eq:Refinement:ODE:initial-conditions}
\begin{aligned}
 	\bx_1(0) &= 0, \quad\! \bx_2(0) = 80,000, \quad\! \bv(0) = 5,000, \quad\! \bgamma(0) = 0, \quad\!
	\balpha(0) = 20 \pi / 180, \\
	 \bq(0) &= 0,  \quad\! \bdelta(0) = -10 \pi / 180, \quad\! \dot\bdelta(0) = 0.
\end{aligned}
\end{equation}
Given an input $\ddot\bdelta$, the dynamic equations \eqref{eq:Refinement:ODE:hypersonic_IVP-delta} and
initial conditions \eqref{eq:Refinement:ODE:initial-conditions} can be used to compute $\bdelta$ analytically, and
make $\bdelta$ the input. However, because  the states \eqref{eq:Refinement:ODE:hypersonic-states} and the
state equations \eqref{eq:Refinement:ODE:hypersonic_IVP} will also be used in the optimal control setting in \cref{sec:numerics_OCP}, 
where $\ddot\bdelta$ is the control and box constraints are imposed on $\bdelta$, we also use \eqref{eq:Refinement:ODE:hypersonic-states} and 
\eqref{eq:Refinement:ODE:hypersonic_IVP} in the simulation setting.

The hypersonic vehicle considered in this example has mass $m = 1000 \textrm{ kg}$, moment of inertia 
$I_z = 247 \textrm{ kg $\times$ m$^2$}$ about the pitch axis, reference area $A_w = 4.4 \textrm{ m$^2$}$, 
and reference length $L_w = 3.6 \textrm{ m}$.
The dynamics \eqref{eq:Refinement:ODE:hypersonic_IVP} of the hypersonic vehicle also depend on gravitational acceleration, computed as
\begin{align*}
	g(x_2) = \mu \ / \ (R_E + x_2)^2 \qquad \mbox{[m/s$^2$]},
\end{align*}
where $\mu = 3.986 \times 10^{14} \textrm{ m$^3$/s$^2$}$ is the standard gravitational parameter and 
$R_E \approx 6.371 \times 10^6 \textrm{ m}$ is the radius of Earth.
Note that the gravitational acceleration $g$ is given and not a component function that needs to be approximated.
The dynamics \eqref{eq:Refinement:ODE:hypersonic_IVP} also depend
on lift, drag, and vehicle moment about the pitch axis, given by
\begin{align*}
	L(x_2, v, \alpha, \delta) &= \overline{q}(x_2, v) C_L(\alpha, \delta) A_w  \;\; \textrm{[N]}, \quad
        D(x_2, v, \alpha, \delta) = \overline{q}(x_2, v) C_D(\alpha, \delta) A_w  \;\; \textrm{[N]}, \\
        M(x_2, v, \alpha, \delta) &= \overline{q}(x_2, v) C_M(\alpha, \delta) A_w L_w \;\; \textrm{[N $\times$ m]},
\end{align*}
with dynamic pressure $\overline{q}$ and atmospheric density $\rho$ given by
\begin{equation} \label{eq:Refinement:ODE:DP}
	\overline{q}(x_2, v) = \frac{1}{2} \rho(x_2) v^2 \quad \textrm{[Pa]}, \qquad
	\rho(x_2) = 1.225 \, \mathrm{exp}(-0.00014x_2) \quad \textrm{[kg/m$^3$]}.
\end{equation}

The lift, drag, and moment coefficients $C_L, C_D, C_M$ for the vehicle are assumed to depend on angle of attack and flap deflection angle. They play the role of the model function in this example, i.e.,
\[
      \bg\big(t, \bx(t) \big) 
      =  \Big(  C_L\big(\balpha(t), \bdelta(t)\big), \;  C_D\big(\balpha(t), \bdelta(t) \big), \;  C_M\big(\balpha(t), \bdelta(t) \big) \Big)^T.
\]
In this example, the ``true'' lift, drag, and moment coefficients,
which compose $\bg_*$, are polynomial models given by
\begin{equation} \label{eq:Refinement:ODE:true-model}
\begin{aligned}
	C_L^*(\alpha, \delta) &= -0.04 + 0.8 \alpha + 0.13 \delta, \\
	C_D^*(\alpha, \delta) &= 0.012 - 0.01 \alpha + 0.6 \alpha^2 - 0.02 \delta + 0.12 \delta^2, \\
	C_M^*(\alpha, \delta) &= 0.1745 - \alpha - \delta.
\end{aligned}
\end{equation}
As mentioned earlier, this is a synthetic example that allows us to compare with the true model; in a real-world setting, these functions would be 
expensive to evaluate using computational fluid dynamics simulations.

In our simulation, we use the  given input
\begin{equation} \label{eq:Refinement:ODE:input-control}
      \ddot\bdelta(t) =  \big((t - 300)_+\big)^2 \times 10^{-9},  
\end{equation}
where $z_+ = \max\{ z, 0 \}$.
The simulation is run until time $T = 550$ [s].

The dynamics \eqref{eq:Refinement:ODE:hypersonic_IVP} are solved using a flipped Legendre-Gauss-Radau collocation.
This discretization will also be used in our trajectory optimization. See, e.g., \cite{SKameswaran_LTBiegler_2008a} for more details.
Formulas for derived quantities and other equations involving states are expressed in base units (kg/m/s) unless otherwise specified, 
but numerical results are reported in (kg/km/s).
Using the (kg/km/s) system led to better numerical performance in the implementation 
of the model refinement approach. Angles appearing in formulas are in radians, but are reported in degrees for figures.

%%%%%%%%%%%%%%%%%%%%%%%%%%%%%%%%%%%%%%%%%%%%%%%%%%%%%%%%%%%%%%
We refine the model to improve the QoI, which, in this example, is downrange:
\[
	\widetilde{q}(\bg) = q\big( \bx(\cdot \, ; \bg), \bg \big) = \bx_1(T).
\]
We construct an initial model $\bg_c$ from $N = 4$ samples of $\bg_*$ generated in ranges
\begin{equation} \label{eq:Refinement:ODE:ranges}
	\alpha \in [-5\pi / 180, 25 \pi / 180], \qquad \delta \in [-15 \pi / 180, 10 \pi / 180]
\end{equation}
using Latin hypercube sampling. 
We use $\gamma_i = \| (\bg_c)_i \|_{\cH_\bk(\Omega)}$ 
as our computable  estimate for $\| (\bg_*)_i \|_{\cH_\bk(\Omega)}$ in \eqref{eq:Refinement:intro:error-indicator-RKHS}, and we take
$\bepsilon_i(y; Y_+, G_c)$ in \eqref{eq:Refinement:intro:error-indicator-RKHS} as the model
error bound. The constants $\gamma_i$ are kept fixed as the model is refined.

To demonstrate the efficacy of our sensitivity-driven approach, we compare it to two other approaches,
max error bound and Latin hypercube.

\emph{Sensitivity-driven (SD)}: Generate $20$ candidate samples $\{y_i^+\}_{i=1}^{20}$ using a hybrid approach, where $10$ of the samples are generated in ranges \eqref{eq:Refinement:ODE:ranges} using Latin hypercube and the other $10$ are states selected at equispaced times along the trajectory computed with the current surrogate. Then, select the point that minimizes \eqref{eq:Refinement:ODE:OCP-QoI-aquisition-function}.
	
\emph{Max error bound (MEB)}: Generate candidates $\{y_i^+\}_{i=1}^{20}$ using the same scheme as the SD approach, but select the point that maximizes
	\begin{equation}    \label{eq:Refinement:ODE:MEB-acquisition-function} 
	 \big\| \big(
		 P_L(y_i^+; Y_c), \;
		 P_D(y_i^+; Y_c), \;
		 P_M(y_i^+; Y_c)
		\big)^T  \big\|_2,   \quad i = 1, \dots, 20,
	\end{equation}
	where $P_L$, $P_D$, $P_M$ represent the pointwise part of the error bound \eqref{thm:Surrogate:RKHS:derivative-error-bound} with $\alpha = (0, \dots, 0)$ corresponding to $C_L$, $C_D$, $C_M$ respectively. Because the three component functions are approximated by the same kernel and use the same samples, $P_L$, $P_D$, and $P_M$ are all identical in this example, and so this is equivalent to simply maximizing $P_L(y_i^+; Y_c)$. This is similar to maximum uncertainty sampling in the context of Bayesian/Gaussian process modeling.

\emph{Latin hypercube (LH)}: In refinement step $r$,  generate $N + r$ samples in ranges \eqref{eq:Refinement:ODE:ranges} using the Latin hypercube method.
	       This is not an adaptive method, but is useful to determine whether leveraging sensitivity information leads to faster convergence compared to selecting samples randomly.

The left plot in \cref{fig:Refinement:ODE:error-comparison} shows the minimum value of \eqref{eq:Refinement:ODE:OCP-QoI-aquisition-function} over the candidates $Y_{\rm cand}$ 
for the selection of the $r$-th refinement point using the sensitivity-driven approach.
This minimum value of \eqref{eq:Refinement:ODE:OCP-QoI-aquisition-function}  is denoted by $| \delta q_{\rm UB}|$.
The right plot in  \cref{fig:Refinement:ODE:error-comparison} shows the relative error between the downrange QoI computed
with the current model, $q_c$, and the downrange QoI computed with the true model, $q_*$, after $r$ refinements for the three approaches. 
We immediately observe that the SD approach significantly outperforms the other approaches, and the sensitivity-based QoI error bound $| \delta q_{\rm UB}|$ follows
roughly the same trend as the actual QoI error $| q_c - q_*|$ when the SD approach is used, giving numerical evidence that our sensitivity-based acquisition function is useful for model refinement.
\begin{figure}[!htb]
    \centering
    \includegraphics[width=0.40\textwidth]{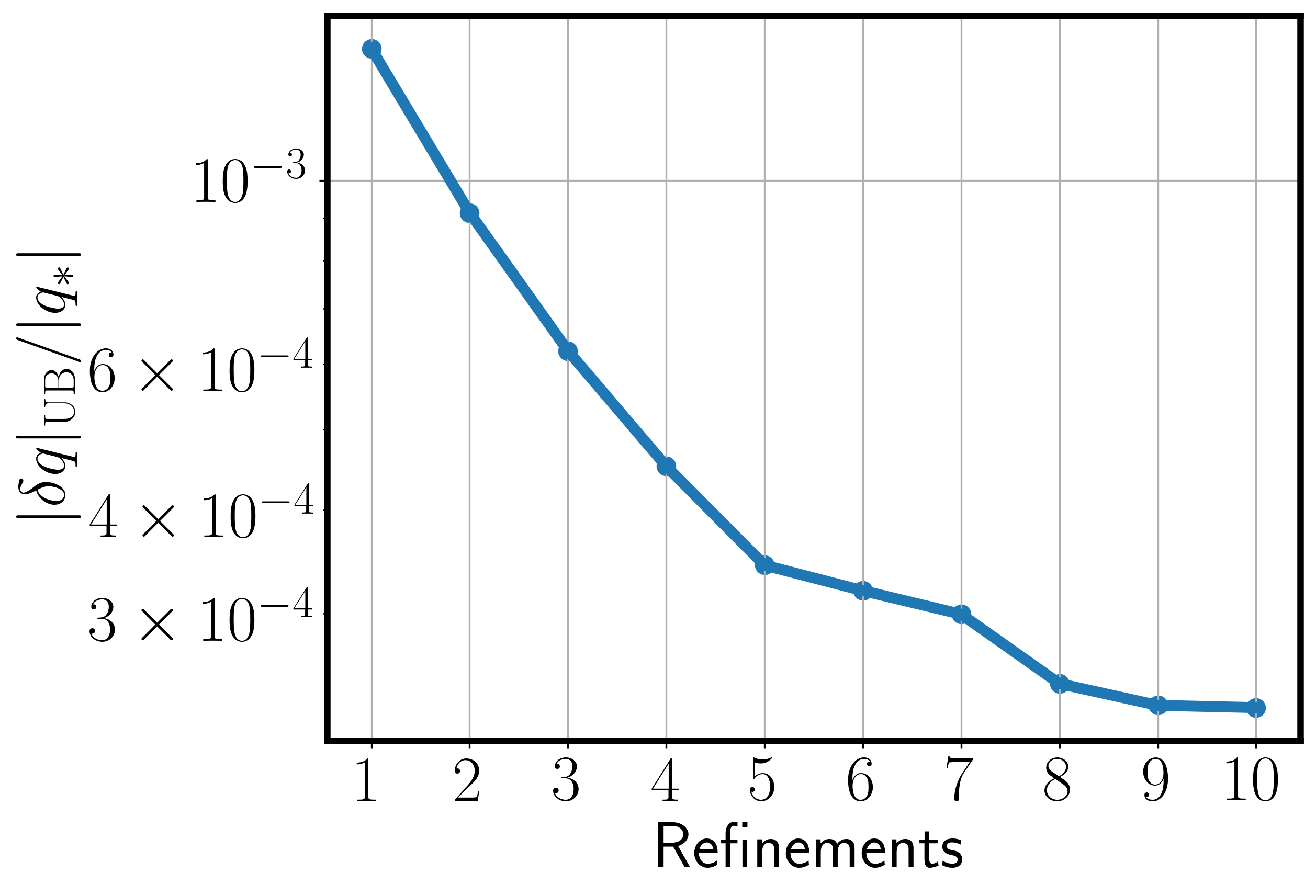} \hfill
    \includegraphics[width=0.50\textwidth]{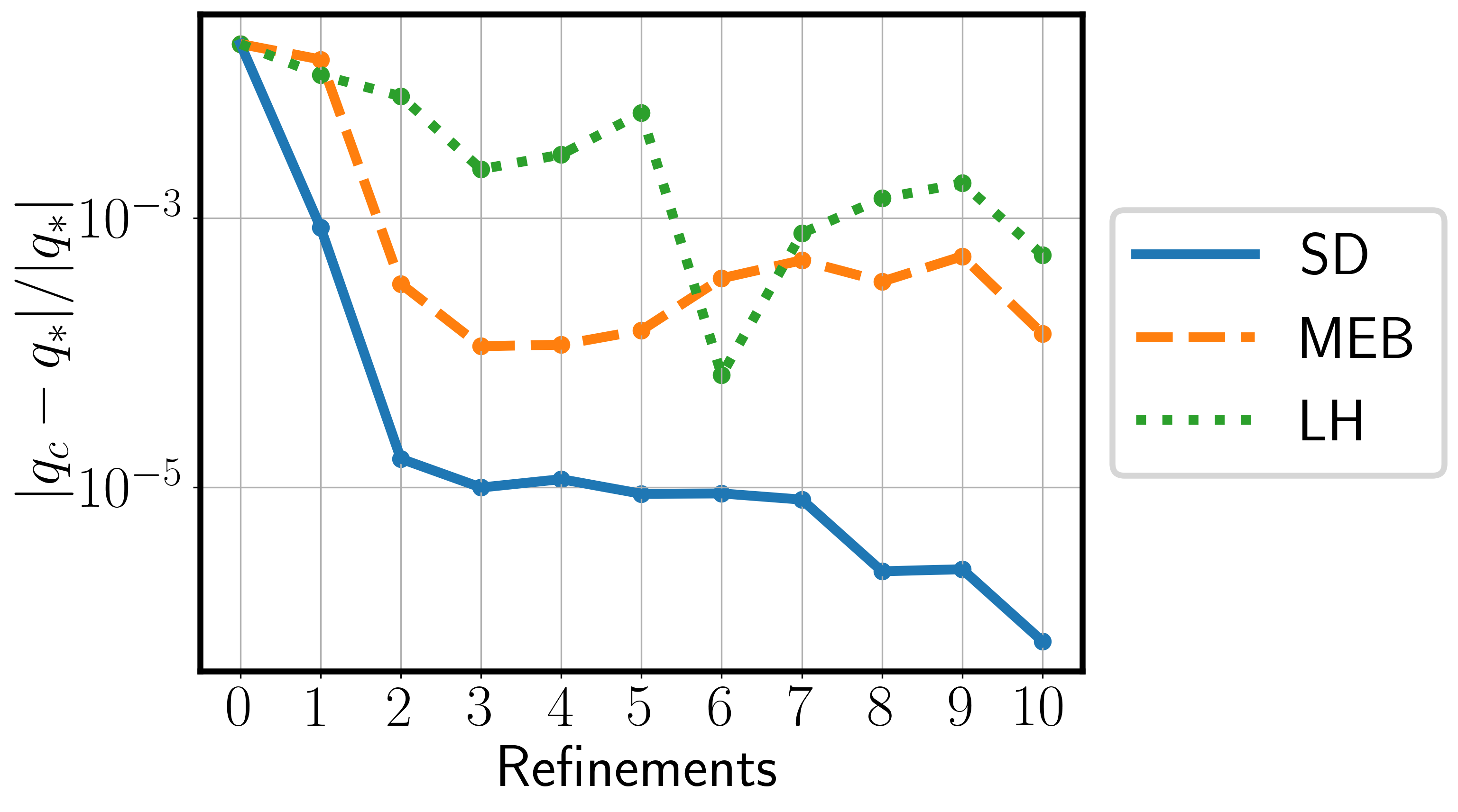}
    \vspace*{-1ex}
    \caption{Sensitivity-based QoI error bounds $| \delta q_{\rm UB}|$ for $r$-th refinement and QoI errors $| q_c - q_*|$ after $r$ refinements for the three approaches.
    The sensitivity-based QoI error bound $| \delta q_{\rm UB}|$ follows roughly the same trend as the actual QoI error $| q_c - q_*|$ when the SD approach is used,
    indicating that our sensitivity-based acquisition function is useful for model refinement. Our QoI-based SD approach significantly outperforms the other approaches.}
    \label{fig:Refinement:ODE:error-comparison}
\end{figure}

\cref{fig:Refinement:ODE:selection} shows the refinement points selected by the three approaches (stopping after $r = 10$ refinements). 
\begin{figure}[!htb]
    \centering
    \includegraphics[width=0.7\textwidth]{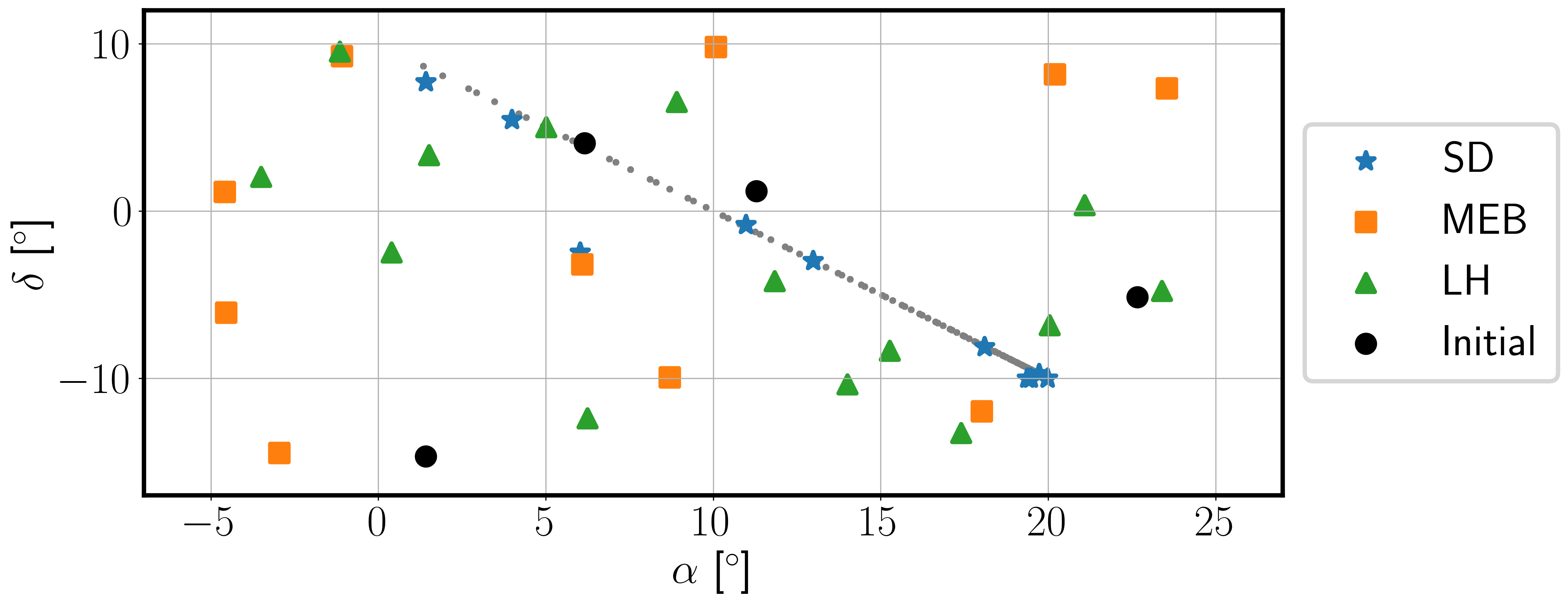}
    \vspace*{-1ex}
    \caption{The small gray dots indicate $(\alpha, \delta)$ pairs encountered at timesteps along the trajectory computed with the true model. 
    Refinement points selected by SD, MEB, LH are shown as blue stars, orange squares, and green triangles respectively.
    The SD approach selects refinement points near the trajectory computed with the true model, whereas MEB distributes samples more evenly because it tries
    to reduce model error everywhere. The LH approach does not construct nested models, so the initial samples are not included in the final LH model.}
    \label{fig:Refinement:ODE:selection}
\end{figure}
The gray dots in \cref{fig:Refinement:ODE:selection} indicate the $(\alpha, \delta)$ pairs encountered along the trajectory computed with the true model 
(true trajectory, in short) at the timesteps 
given by the Legendre-Gauss-Radau discretization scheme (except for a few initial timesteps where the angle of attack quickly adjusts from its initial condition of $0^\circ$ to $\sim 20^\circ$).
Since $\delta = -10^\circ$ is the flap angle for most of the trajectory (cf.\ \eqref{eq:Refinement:ODE:initial-conditions} and \eqref{eq:Refinement:ODE:input-control}) and 
since the moment coefficient tends to be small along the trajectory to avoid tumbling, i.e., $C_M^*\big(\balpha(t), \bdelta(t) \big) \approx 10^\circ - \balpha(t) - \bdelta(t) \approx 0$, 
much of the true trajectory is near $(\alpha, \delta) = (20^\circ, -10^\circ)$ and near the line  $\alpha + \delta = 10^\circ$.
The SD approach selects points near the true trajectory, resulting in much better reduction of the QoI error. 
Meanwhile, the MEB approach simply looks for the point with the largest surrogate error bound, prioritizing points near the boundary of the $(\alpha, \delta)$ ranges under consideration. 
These points tend to be far away from the initial samples and the points along the trajectory, often failing to reduce the QoI error or even increasing it. 
This may be observed directly in \cref{fig:Refinement:ODE:trajectories}, which shows the trajectories obtained after the first refinement for the three approaches.
\begin{figure}[H]
    \centering
    \includegraphics[width=0.43\textwidth]{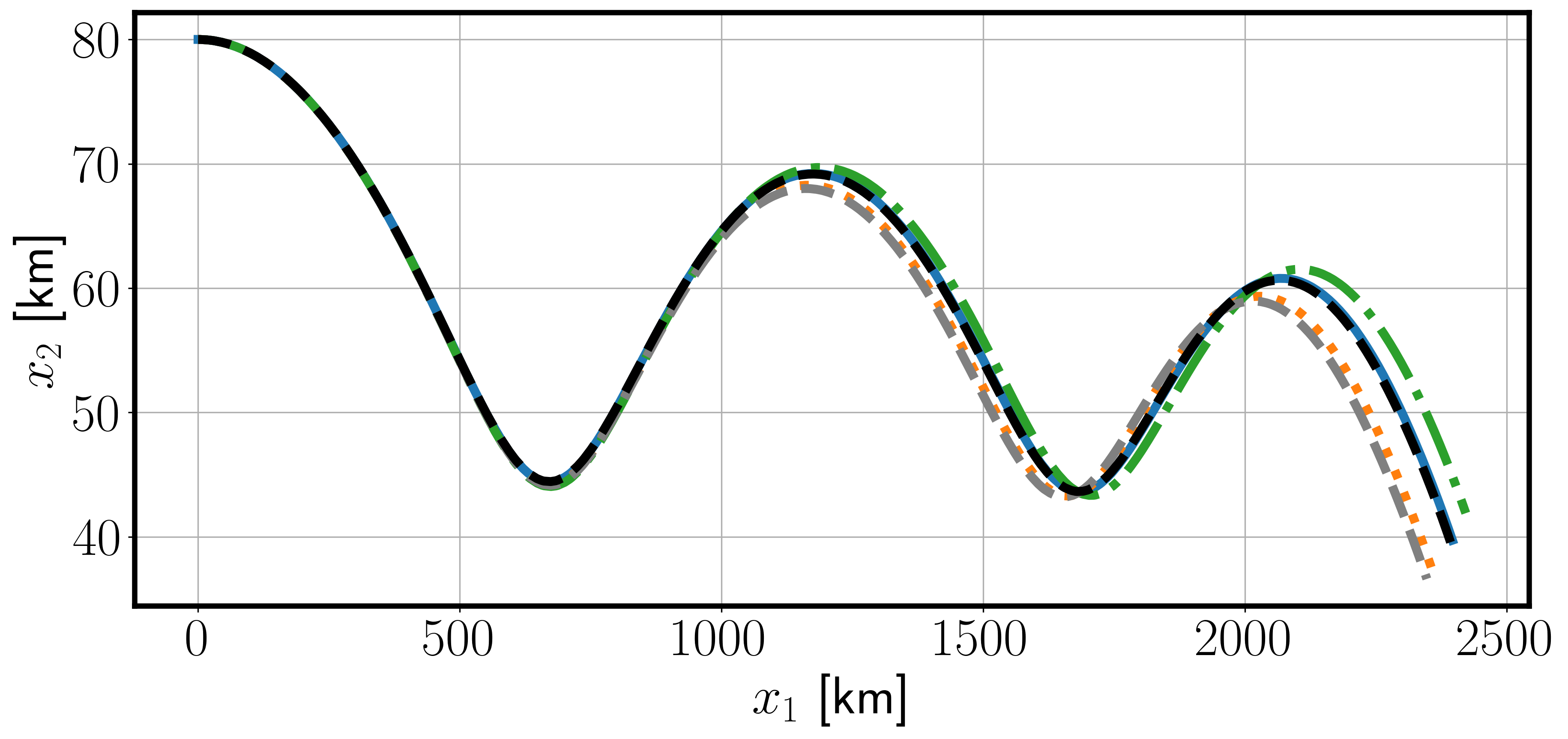} \hfill
    \includegraphics[width=0.56\textwidth]{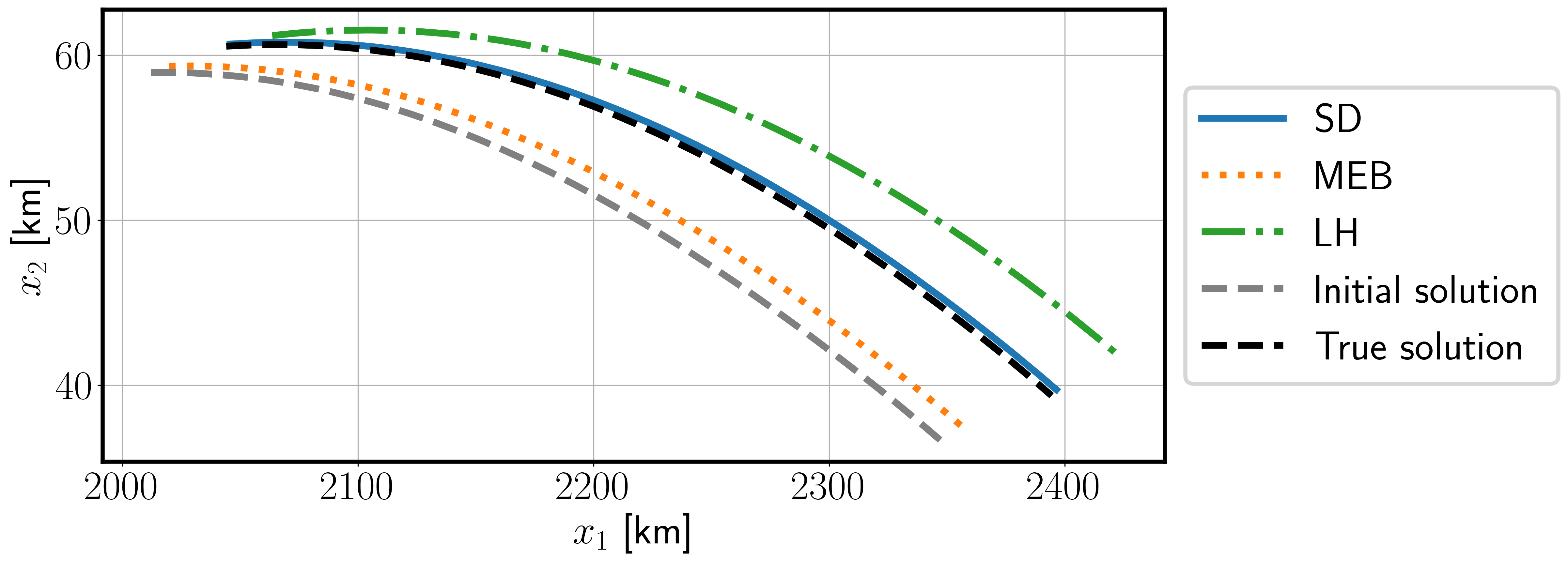}
    \vspace*{-3ex}
    \caption{Trajectories obtained after first refinement using each approach. The right plot focuses on the tail of the trajectory. 
                  After the first refinement, the SD refined trajectory is already close to the trajectory with the true model, whereas
                   the MEB and LH  refined trajectories still differ noticeably from the trajectory with the true model.
    }
    \label{fig:Refinement:ODE:trajectories}
\end{figure}

%%%%%%%%%%%%%%%%%%%%%%%%%%%%%%%%%%%%%%%%%%%%%%%%%%%%%%%%%      
%%%%%%%%%%%%%%%%%%%%%%%%%%%%%%%%%%%%%%%%%%%%%%%%
\subsection{Trajectory Optimization} \label{sec:numerics_OCP}
Next, we consider model refinement for the maximum-downrange problem
\begin{subequations} \label{eq:Refinement:OCP:hypersonic-OCP}
\begin{align}
	\min \quad& -\bx_1(T) \\
	\mbox{s.t.} \quad& \textrm{dynamics \eqref{eq:Refinement:ODE:hypersonic_IVP} with model $\bg$}
						 \textrm{ and initial conditions \eqref{eq:Refinement:ODE:initial-conditions}}, \\
						& 20,000 \leq \bx_2(T) \leq 21,000, \\
						& \overline{q}\big(\bx_2(t), \bv(t)\big) \leq 40,000,  & t \in (0,T),\\
						& \dot{Q}_{\rm stag}\big(\bx_2(t), \bv(t)\big) \leq 6 \times 10^6, & t \in (0,T), \\
						& \bx_2(t) \geq 20,000, & t \in (0,T),  \\
%\end{align}
%\begin{align}
						& -30 \pi / 180 \leq \bgamma(t) \leq 30 \pi / 180,  & t \in (0,T),\\
					        &    -5 \pi / 180 \leq \balpha(t) \leq 20 \pi / 180, & t \in (0,T),\\
						& -10 \pi / 180 \leq \bdelta(t) \leq 15 \pi / 180, & t \in (0,T),
\end{align}
\end{subequations}
with states \eqref{eq:Refinement:ODE:hypersonic-states} and control $\bu = \ddot\bdelta$.
The dynamic pressure $\overline{q}$ is given as in \eqref{eq:Refinement:ODE:DP}, and the stagnation heat rate is
\[
	\dot{Q}_{\rm stag}(x_2, v) = C \rho(x_2)^{0.5} v^3 \textrm{ [W/m$^2$]}, 	\qquad C = 37.356.
\]
The final time $T$ is variable and determined as part of the solution of  \eqref{eq:Refinement:OCP:hypersonic-OCP}.
Using a standard technique (see, e.g.,  \cite[Sec.~1.2.1]{MGerdts_2012a}),
the problem \eqref{eq:Refinement:OCP:hypersonic-OCP} may be converted into a fixed-time optimal control problem 
by converting the dynamics to normalized time, 
\[
	\frac{d}{d\tau} \bx(T \tau) = T \, \bff \Big(T \tau, \bx(T \tau), \bu(T \tau), \bp, \, \bg \big(T \tau, \bx(T \tau), \bu(T \tau), \bp \big) \Big), \qquad \aall \tau \in (0, 1),
\]
 making $T$ an optimization variable which will be entered as a component in $\bp$.  In the problem \eqref{eq:Refinement:OCP:hypersonic-OCP} the duration is the only parameter, i.e., $\bp = T$.

Because of the presence of path constraints and inequality constraints on some state variables, the problem \eqref{eq:Refinement:OCP:hypersonic-OCP} is not in the form \eqref{eq:Sensitivity:OCP:OCP}.
Thus, we cannot perform our sensitivity computations on it directly. Instead, we use the following reference trajectory tracking problem to compute sensitivities.
Given a model $\bg$, we solve \eqref{eq:Refinement:OCP:hypersonic-OCP} to compute $\bx^{\rm ref} = \bx^{\rm ref}(\bg)$, $\bu^{\rm ref} = \bu^{\rm ref}(\bg)$, $\bp^{\rm ref} := T^{\rm ref} = T^{\rm ref}(\bg)$.
With positive semidefinite matrix $\BQ$ and positive definite matrices $\BR_u$ and $\BR_p$ we consider the reference tracking problem
\begin{equation} \label{eq:Refinement:OCP:hypersonic-RT}
\begin{aligned}
	\min \quad
	& \int_{0}^{1} \big( \bx(T \tau) - \bx^{\rm ref}(T^{\rm ref} \tau) \big)^T \BQ \big( \bx(T \tau) - \bx^{\rm ref}(T^{\rm ref} \tau) \big) \\ 
	&\quad + \big( \bu(T \tau) - \bu^{\rm ref}(T^{\rm ref} \tau) \big)^T \BR_u \big( \bu(T \tau) - \bu^{\rm ref}(T^{\rm ref} \tau) \big) \, d\tau  \\ 
	&+ ( \bp - \bp^{\rm ref} )^T \BR_p ( \bp - \bp^{\rm ref} ) \\
	\mbox{s.t.} \quad& \textrm{dynamics \eqref{eq:Refinement:ODE:hypersonic_IVP} with model $\bg$}
						 \textrm{ and initial conditions \eqref{eq:Refinement:ODE:initial-conditions}}.
\end{aligned}
\end{equation}
Because $\bx^{\rm ref} = \bx^{\rm ref}(\bg)$, $\bu^{\rm ref} = \bu^{\rm ref}(\bg)$, $\bp^{\rm ref} := T^{\rm ref} = T^{\rm ref}(\bg)$ solve the dynamics \eqref{eq:Refinement:ODE:hypersonic_IVP} with model $\bg$
and $\bx^{\rm ref} = \bx^{\rm ref}(\bg)$ satisfies the initial conditions \eqref{eq:Refinement:ODE:initial-conditions},
the solution of the reference tracking problem \eqref{eq:Refinement:OCP:hypersonic-RT} is the reference trajectory, $\bx(\bg) = \bx^{\rm ref}$, $\bu(\bg) = \bu^{\rm ref}$, $\bp(\bg) = T(\bg) = T^{\rm ref}$.
We compute the sensitivity of the solution of the reference tracking problem \eqref{eq:Refinement:OCP:hypersonic-RT} with respect to $\bg$.
Note that in our sensitivity computation, $\bx^{\rm ref} $, $\bu^{\rm ref}$, $\bp^{\rm ref} := T^{\rm ref}$ are considered fixed (not a function of $\bg$), and
we compute the sensitivity of the solution of \eqref{eq:Refinement:OCP:hypersonic-RT} with respect to changes in the model $\bg$ in the dynamics \eqref{eq:Refinement:ODE:hypersonic_IVP}.
This sensitivity information describes the ability to track the reference trajectory when the aerodynamic coefficient models differ from the ones used to compute the reference trajectory.
Note, however, that this does not tell us how sensitive the solution of \eqref{eq:Refinement:OCP:hypersonic-OCP} is to the coefficients.

For model refinement, we use the downrange QoI
\begin{equation} \label{eq:Refinement:OCP:hypersonic-RT-QoI}
      q(\bg) := q\big(\bx(\cdot \, ; \bg), \bu(\cdot \, ; \bg), \bp(\bg), \bg \big) = \bx_1(T),
\end{equation}
where $\big(\bx(\cdot \, ; \bg), \bu(\cdot \, ; \bg), \bp(\bg)\big)$ solves  \eqref{eq:Refinement:OCP:hypersonic-RT}.

In the following computations we use \eqref{eq:Refinement:OCP:hypersonic-RT} with
\[
	\BQ = \mathrm{diag}(10^{-3}, 10^1, 0, 0, 10^1, 0, 0, 0), \qquad \BR_u = 10^8, \qquad \BR_p = 10^{-3},
\]
with the ordering of the states given by \eqref{eq:Refinement:ODE:initial-conditions}.

The sensitivity-driven (SD) model refinement approach proceeds as follows:
\begin{enumerate}
	\item Construct the initial $\bg_c = (C_L, C_D, C_M)$  from $N = 4$ Latin hypercube samples in ranges \eqref{eq:Refinement:ODE:ranges}.
                 Use $\gamma_i = \| (\bg_c)_i \|_{\cH_\bk(\Omega)}$  as the computable estimate for $\| (\bg_*)_i \|_{\cH_\bk(\Omega)}$. 
                 The constant $\gamma_i$ is kept fixed from this point onwards, even as the model is refined.
                 	
	\item Solve \eqref{eq:Refinement:OCP:hypersonic-OCP} with the model $\bg_c$ to obtain a reference trajectory.
	
	\item Use the reference trajectory from step 2 in the reference tracking problem \eqref{eq:Refinement:OCP:hypersonic-RT} with model $\bg_c$
	         and computable model bound functions \eqref{eq:Refinement:intro:error-indicator-drivative-RKHS}
	         to
	         compute the bound of \cref{thm:Refinement:OCP:OCP-acquisition-function} for each candidate $y_+ = (\alpha_+, \delta_+)$,
	         and select the one that minimizes this bound.
	
	\item Use $(\alpha_+, \delta_+)$ from step 3 to compute a refined model $\bg_+$.
		
	\item Set $\bg_c \leftarrow \bg_+$ and repeat steps 2-4 as desired.
\end{enumerate}
The MEB approach proceeds analogously, but uses the acquisition function \eqref{eq:Refinement:ODE:MEB-acquisition-function} in step 3
to compute the new $y_+ = (\alpha_+, \delta_+)$. The acquisition function \eqref{eq:Refinement:ODE:MEB-acquisition-function} does not
consider the impact of model change on the QoI \eqref{eq:Refinement:OCP:hypersonic-RT-QoI}.
The LH approach does not use an acquisition function at all and randomly selects points with which the new model is constructed.

\cref{fig:Refinement:OCP:error-comparison} shows the performance of the refinement approach.
\begin{figure}[!htb]
\centering
\includegraphics[width=0.40\textwidth]{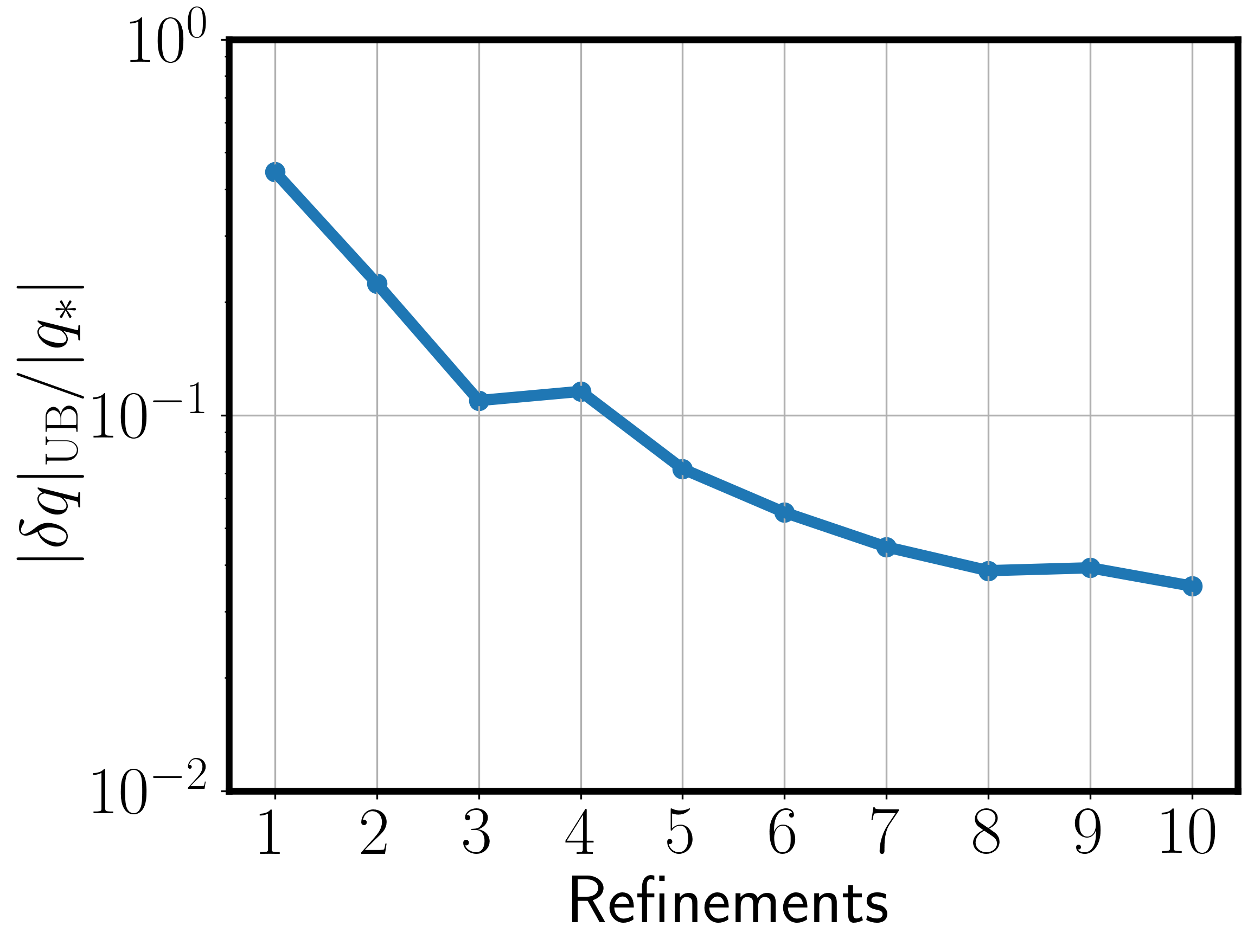} \hfill
\includegraphics[width=0.53\textwidth]{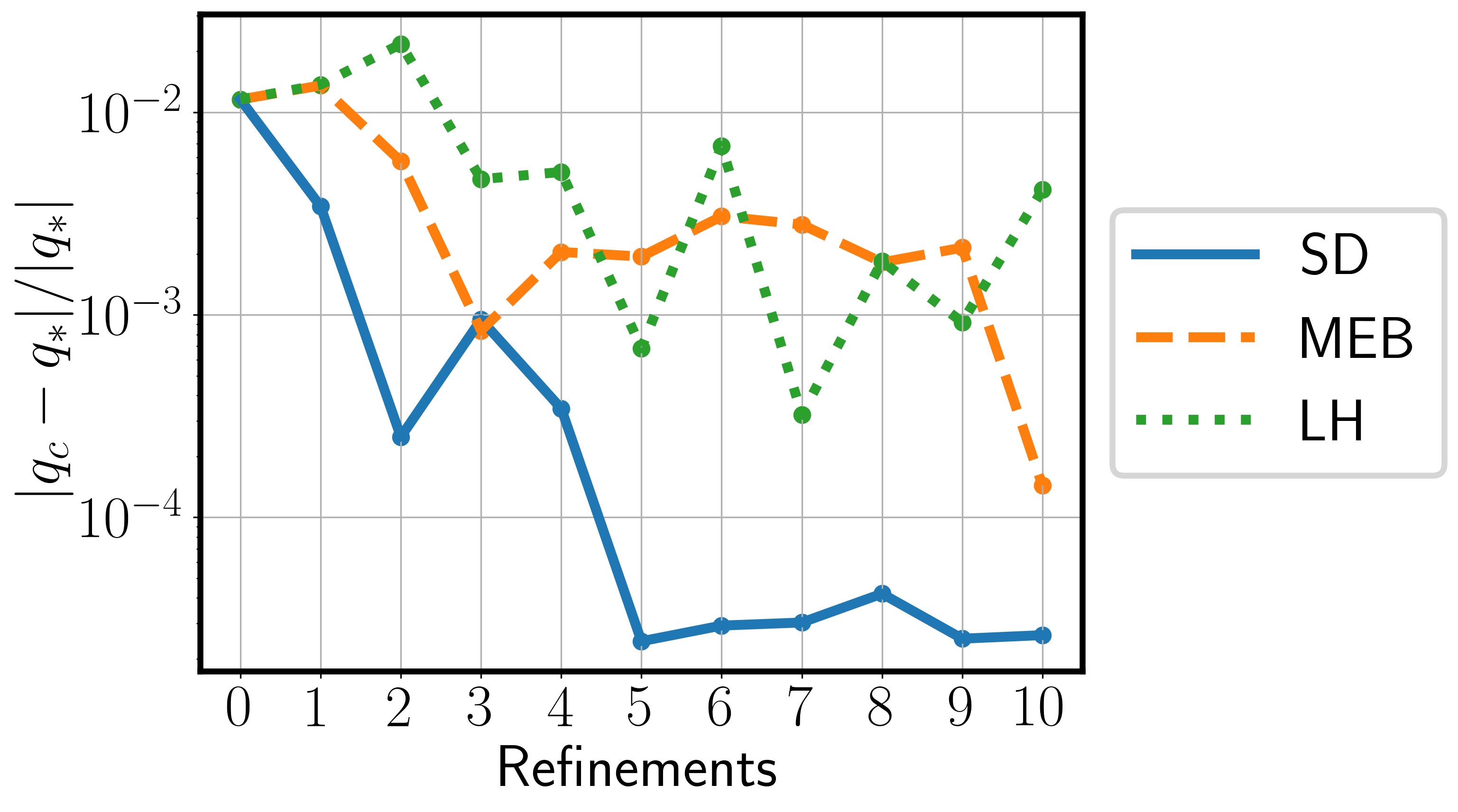}
\vspace*{-1ex}
\caption{Sensitivity-based QoI error bounds for $r$-th refinement and relative QoI errors after $r$ refinements for the three approaches.
        The SD approach significantly outperforms the other approaches.  
        The error $| q_c - q_*|/|q_*|$ for the SD approach does not decrease monotonically, and there is less agreement between $|\delta q_{\rm UB}|/|q_*|$   
        and $| q_c - q_*|/|q_*|$ because sensitivities of the solution of  \eqref{eq:Refinement:OCP:hypersonic-RT} with respect to changes 
        in the model $\bg$ in the dynamics \eqref{eq:Refinement:ODE:hypersonic_IVP} are used, but as the current model $\bg_c$ is updated
        to the new model $\bg_+$, the reference trajectory changes as well. 
}
\label{fig:Refinement:OCP:error-comparison}
\end{figure}
In \cref{fig:Refinement:OCP:error-comparison}, $q_c$ refers to the downrange QoI \eqref{eq:Refinement:OCP:hypersonic-RT-QoI}
computed with the reference trajectory $\bx^{\rm ref} = \bx^{\rm ref}(\bg_c)$, 
$\bu^{\rm ref} = \bu^{\rm ref}(\bg_c)$, $\bp^{\rm ref} := T^{\rm ref} = T^{\rm ref}(\bg_c)$, and the model $\bg_c$.
The quantity  $|\delta q_{\rm UB}|$ is the lowest bound computed in step 3 above, which gives an estimate for the improvement
achieved by replacing $\bg_c$ with $\bg_+$. 
To assess the effect of model refinement, we solve \eqref{eq:Refinement:OCP:hypersonic-RT} with the 
current reference trajectory $\bx^{\rm ref} = \bx^{\rm ref}(\bg_c)$, $\bu^{\rm ref} = \bu^{\rm ref}(\bg_c)$, $\bp^{\rm ref} := T^{\rm ref} = T^{\rm ref}(\bg_c)$,
and with model $\bg = \bg_*$. The QoI \eqref{eq:Refinement:OCP:hypersonic-RT-QoI} computed with this trajectory is referred to as $q_*$ in \cref{fig:Refinement:OCP:error-comparison}.
Note that since $q_*$ is computed with the solution of  \eqref{eq:Refinement:OCP:hypersonic-RT} with fixed model $\bg = \bg_*$ in the dynamics,
but the reference trajectory $\bx^{\rm ref} = \bx^{\rm ref}(\bg_c)$, $\bu^{\rm ref} = \bu^{\rm ref}(\bg_c)$, $\bp^{\rm ref} := T^{\rm ref} = T^{\rm ref}(\bg_c)$
changes as the model $\bg_c$ changes, the value $q_*$ changes with refinement level.
Again, the right plot in \cref{fig:Refinement:OCP:error-comparison} shows that the SD approach significantly outperforms the other approaches.
However, the error $| q_c - q_*|/|q_*|$ for the SD approach is not decreasing monotonically, and there is less agreement between $|\delta q_{\rm UB}|/|q_*|$   and
$| q_c - q_*|/|q_*|$.  A reason is that we compute sensitivities of the solution of  \eqref{eq:Refinement:OCP:hypersonic-RT} with respect to changes 
in the model $\bg$ in the dynamics \eqref{eq:Refinement:ODE:hypersonic_IVP}, but as we update the current model $\bg_c$ with the new model
$\bg_+$, the reference trajectory changes as well.
For the other two approaches, there is no expectation that the error $| q_c - q_*|/|q_*|$ decreases monotonically.

\cref{fig:Refinement:OCP:selection} shows the refinement points selected by each refinement approach similarly to \cref{fig:Refinement:ODE:selection}, 
and the gray dots indicate the $(\alpha, \delta)$ pairs corresponding to the optimal trajectory of \eqref{eq:Refinement:ODE:hypersonic-states} 
computed with the true model $\bg_*$. 
\begin{figure}[!htb]
\centering
\includegraphics[width=0.7\textwidth]{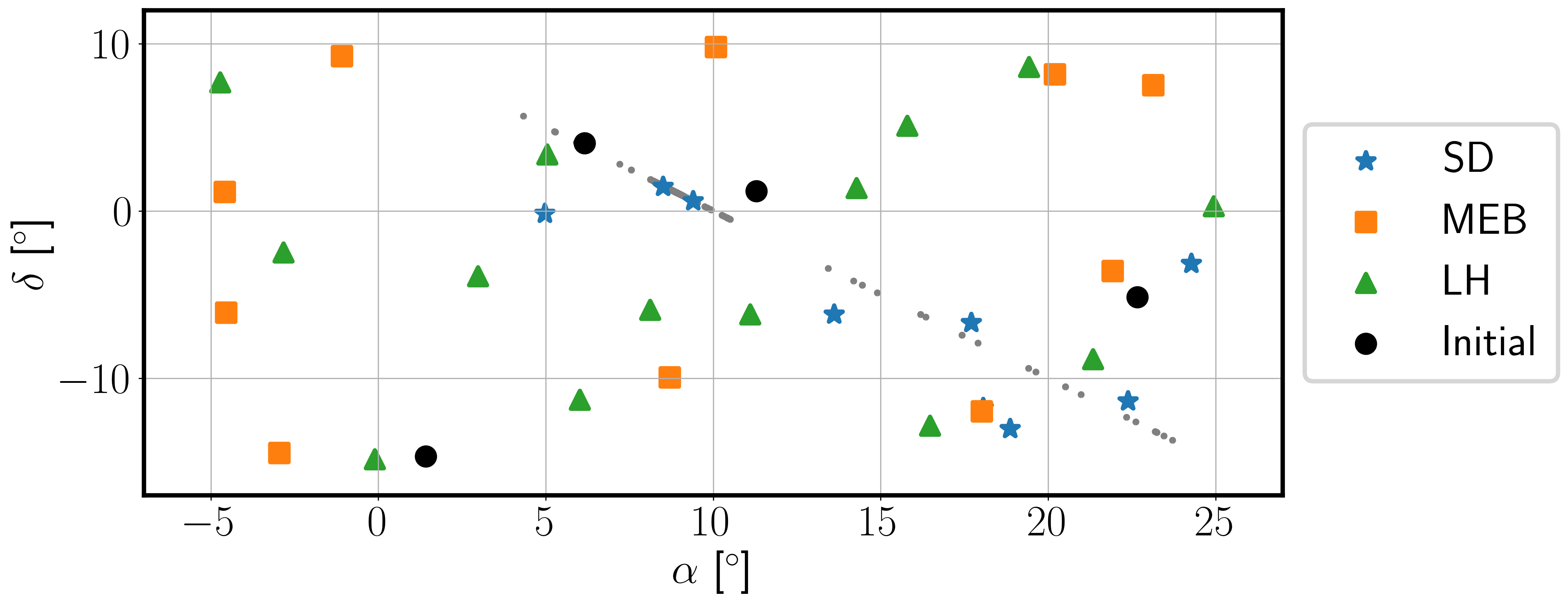}
\vspace*{-1ex}
\caption{The small gray dots indicate $(\alpha, \delta)$ pairs encountered at timesteps along the trajectory computed with the true model. 
Refinement points selected by SD, MEB, LH are shown as blue stars, orange squares, and green triangles respectively.
The SD selects refinement points near the trajectory computed with the true model, whereas MEB distributes samples more evenly because it tries
to reduce model error everywhere. The LH approach does not construct nested models, so the initial samples are not included in the final LH model.}
\label{fig:Refinement:OCP:selection}
\end{figure}
Because the reference trajectories approach this trajectory as the model $\bg_c$ better approximates $\bg_*$, we expect good  refinement points 
to be near the gray points, and indeed the SD approach selects such points.
In contrast, as in the simulation case, the MEB approach simply looks for the point with the largest surrogate error bound, 
prioritizing points near the boundary of the $(\alpha, \delta)$ ranges under consideration. 
These points tend to be far away from the initial samples and the points along the trajectory.
The LH approach selects samples randomly distributed in the $(\alpha, \delta)$ ranges under consideration.

Finally, \cref{fig:Refinement:OCP:trajectories} shows the reference trajectories obtained after the first refinement using each approach. 
\begin{figure}[!htb]
\centering
\includegraphics[width=0.43\textwidth]{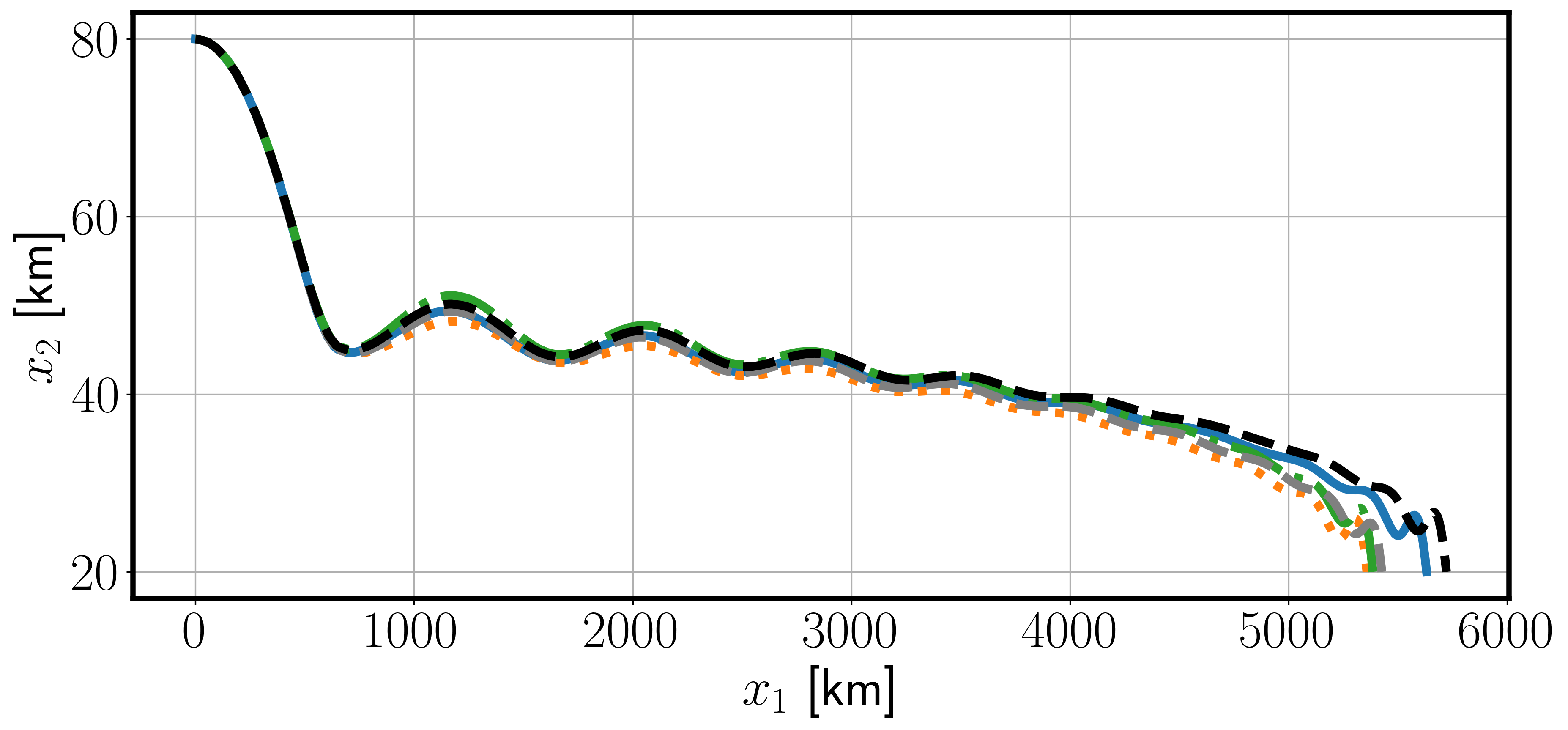}
\includegraphics[width=0.56\textwidth]{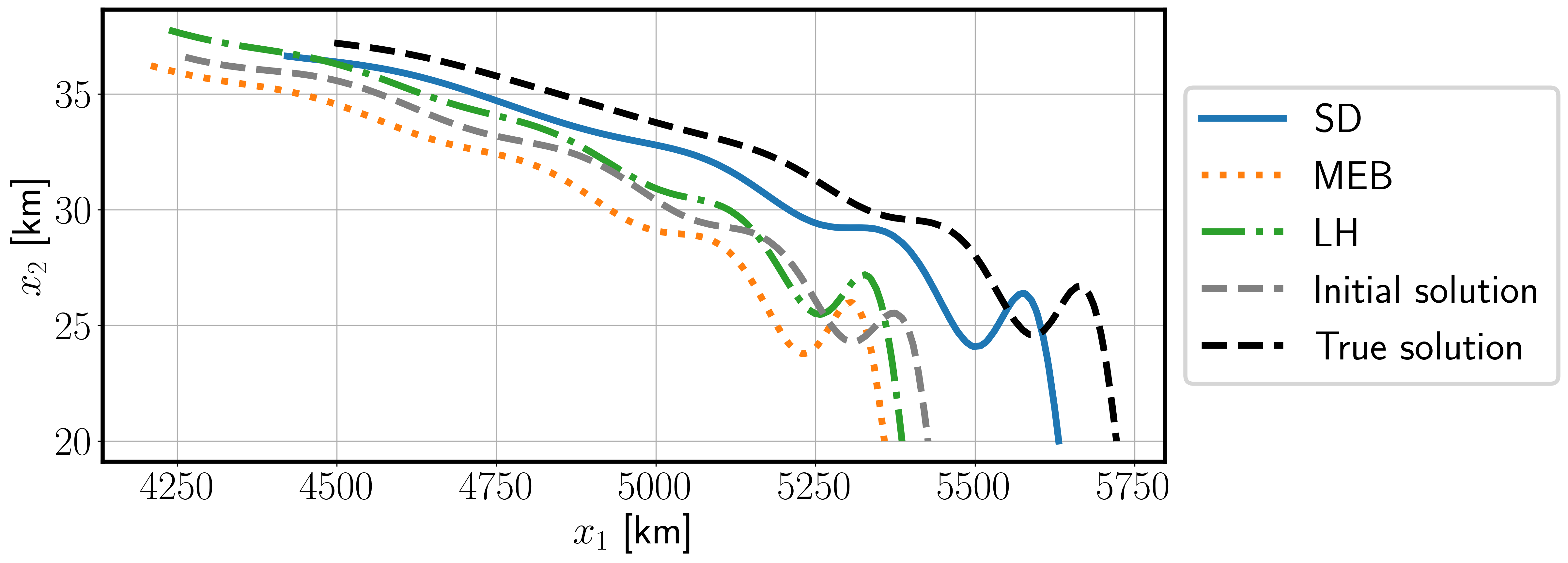}
\vspace*{-4ex}
\caption{Trajectories obtained after first refinement using each approach. The right plot focuses on the tail of the trajectory. 
              After the first refinement, the SD refined trajectory is closer to the trajectory with the true model, whereas
               the MEB and LH  refined trajectories still differ substantially from the trajectory with the true model.
}
\label{fig:Refinement:OCP:trajectories}
\end{figure}
Once again, the SD approach gives a trajectory that is closer to the true solution. As mentioned before, we do not compute sensitivities of the original OCP \eqref{eq:Refinement:OCP:hypersonic-OCP}, so it is not obvious why refining the model based on the reference tracking problem should give a more accurate solution to the original problem. One possible reason is that the SD approach aims to improve the ability of a controller to track the reference trajectory computed with the current model when the actual dynamics are subject to the true model, and the true trajectory can be perfectly tracked with $\bg = \bg_*$, i.e., the optimal objective value in \eqref{eq:Refinement:OCP:hypersonic-RT} is zero when taking the solution of \eqref{eq:Refinement:OCP:hypersonic-OCP} with $\bg = \bg_*$ as the reference trajectory. This suggests that in order to improve the ability of a controller to track the reference trajectory, the reference trajectory must be close to the true trajectory.

%%%%%%%%%%%%%%%%%%%%%%%%%%%%%%%%%%%%%%%%%%%%%%%%%%%%%%%%%     
%%%%%%%%%%%%%%%%%%%%%%%%%%%%%%%%%%%%%%%%%%%%%%%%%%%%%%%%%      
\section{Conclusions and Future Work} \label{sec:conclusion}

We have developed an approach for surrogate model refinement for the simulation of dynamical systems and the solution of optimization problems governed by dynamical systems 
in which surrogates replace expensive-to-compute state- and control-dependent component functions in the dynamics or in the objective function. 
A new acquisition function was developed to guide the selection of additional data points for evaluation of the true component function.
This acquisition function is based on two ingredients.
First, we use the sensitivity of the solution to the simulation or optimization problem with respect to variations in the component function to derive 
the sensitivity of a solution-dependent QoI with respect to the component function.
This sensitivity is used as an estimate between the QoI evaluated at a surrogate and the QoI evaluated at the true component function.
The second ingredient is an efficient-to-evaluate bound for the error between a surrogate and the true component function were the surrogate to be refined at a given point.
We have shown that surrogates computed via kernel interpolation possess such bounds.
Our acquisition function evaluated at a sample point
is computed as the maximum absolute value of the sensitivity of the QoI in the direction of a perturbation,
where the maximum is taken over all perturbations that satisfy the error bound between a surrogate computed with the additional sample and the true component function.
This maximization problem is a linear program in function space, which has an analytical solution.
A new sample point is computed by evaluating the acquisition function at candidate samples and selecting the sample that gives the smallest acquisition function value.

The proposed adaptive surrogate model refinement performed well on the problem of simulating the trajectory of a hypersonic vehicle and
the problem of optimizing the trajectory of a hypersonic vehicle. In both cases, a few additional samples (less than five in our examples) were sufficient
to substantially improve the quality of the solution computed with the refined surrogate model over the quality of the solution computed with the initial
surrogate.

While our approach performs well in the examples provided in this work, there are no theoretical results yet that quantify the improvement when a sample is added, 
let alone theoretical results that prove the convergence of our approach.  
The difficulty in establishing such results is that our approach does not aim to generate surrogates that approximate the true component function everywhere.
Instead, the surrogates approximate the true component function where it is important for the solution-dependent QoI, but the error between
the surrogate and the true component function can be large in other regions.
Theoretical analysis of the performance of our approach is part of future work, as is the application of our approach to other examples.

Because our acquisition functions rely on a first-order sensitivity-based approximation of the QoI error, 
we expect that the initial surrogate for $\mathbf{g}$ must be sufficiently accurate along the states encountered 
by the trajectory for the acquisition functions to be useful. In the absence of a priori knowledge of this region, 
constructing such a surrogate may require sampling $\mathbf{g}_*$ over a large portion of the state space, especially when the trajectory covers a large portion of the state space.
In this scenario, in the worst case, the number of required samples may suffer from the curse of dimensionality. 
However, if prior knowledge is available that restricts the trajectory to a small subset of the state space, then the 
effective dimensionality of the problem is reduced, and the sampling burden can be significantly mitigated.
Overall, our approach aims to combat the curse of dimensionality by approximating the true component
function locally where needed, but not globally.

One can also explore the extension of our approach to partial differential equations
(PDEs) or PDE-constrained optimization. To extend the sensitivity analysis with respect to component functions 
to the PDE setting, where component functions may represent nonlinear constitutive equations, the state space
and the function space for the component functions must be carefully chosen. 
The selection of these spaces also impacts surrogate model design. If the component function depends on the pointwise evaluations
of the PDE solution in space and time, then, assuming sufficient regularity of the PDE solution, the present kernel
interpolation framework could be used. However, e.g., for time-dependent PDEs the component function 
may also depend on the Hilbert- or Banach space $\mathcal{X}$-valued PDE solution $\bx(t) \in \mathcal{X}$ 
at time $t$, which means one needs surrogate functions defined on $\mathcal{X}$ and corresponding error bounds.
Extensions in these directions are interesting avenues for future work.

%%%%%%%%%%%%%%%%%%%%%%%%%%%%%%%%%%%%%%%%%%%%%%%%%%%%%%%%%%%%
%\bibliographystyle{siamplain}
%\bibliography{references}

%%%%%%%%%%%%%%%%%%%%%%%%%%%%%%%%%%%%%%%%%%%%%%%%
\newpage
%!TEX root = SISC_M181346.tex

\setcounter{section}{0}
\renewcommand\thesection{SM\arabic{section}}

\setcounter{figure}{0}

\setcounter{page}{1}
\renewcommand{\thepage}{SM.\arabic{page}}

\begin{center}
{\bf SUPPLEMENTARY MATERIALS: SENSITIVITY-DRIVEN ADAPTIVE SURROGATE MODELING FOR SIMULATION AND OPTIMIZATION OF DYNAMICAL SYSTEMS}

\medskip
JONATHAN R.\ CANGELOSI AND MATTHIAS HEINKENSCHLOSS 

\medskip
\end{center}

In this supplement, we discuss an approach to adapt the model based on improving 
the error in the solution $\bx$ of the simulation problem, 
or on improving the error in the solution $\bx, \bu, \bp$ of the optimal control problem.

%%%%%%%%%%%%%%%%%%%%%%%%%%%%%%%%%%%%%%%%%%%
\section{Model Refinement Based on Solution Improvement} \label{sec:model_refinement_sol}
\Cref{sec:model_refinement} of the main paper describes our approach to adapt the model based on improving 
a QoI. In principle, one can also use the error in the solution $\bx$ of the simulation 
\eqref{eq:problem_formulation:IVP}, or the error in the solution $\bx, \bu, \bp$ of the 
optimal control problem \eqref{eq:Sensitivity:OCP:OCP}, respectively, to refine the model.
We will sketch the solution-based refinement approach here.
The main difference between the QoI-based refinement approach of \cref{sec:model_refinement}
and the solution-based refinement approach presented here is that 
the acquisition functions  \eqref{eq:Refinement:ODE:OCP-QoI-aquisition-function}
and \eqref{eq:Refinement:OCP:OCP-acquisition-function} in the QoI-based refinement approach
have an easily computable analytical expression, whereas the acquisition functions in the 
solution-based refinement approach are objective function values of some optimization 
problems that need to be solved numerically, and that, in the discretized case, are known to be NP-hard.
In practice, tailored solution approaches to these optimization problems are sometimes
possible, but even then evaluation of the acquisition functions via the numerical solution of
an optimization problem is computationally more expensive than the evaluation of
the acquisition functions \eqref{eq:Refinement:ODE:OCP-QoI-aquisition-function}
and \eqref{eq:Refinement:OCP:OCP-acquisition-function} in the QoI-based case.

%%%%%%%%%%%%%%%%%%%%%%%%%%%%%%%%%%%%%%%%%%%%%%%%%%%%%%%%%%%%%
\subsection{Simulation}  \label{sec:Refinement:ODE:ODE-sol}

This section discusses model refinement for simulation problems of the form \eqref{eq:problem_formulation:IVP}.
We continue to use the notation of \cref{sec:model_refinement}.

%%%%%%%%%%%%%%%%%%%%%%%%%%%%%%%%%%%%%%%%%%%%%%%%%%%%%%%%%%%%%
%\subsubsection{Improving the ODE Solution} \label{subsec:Refinement:ODE:ODE}
 Let $\bg_*$ be the true model and $\bx_* := \bx(\cdot \, ; \bg_*)$ the corresponding solution of the ODE \eqref{eq:problem_formulation:IVP} with $\bg = \bg_*$.
 Furthermore, let $\bg_+[y_+]  = \bg(\cdot \, ; Y_+, G_+)$ be the new model constructed using $y_+ \in \Omega$, and
 let $\bx(\cdot \, ; \bg_+[y_+])$ be the solution of \eqref{eq:problem_formulation:IVP} with $\bg = \bg_+[y_+]$.
 Given a weight matrix function  $\BQ \in \big(L^2(I)\big)^{n_x \times n_x}$ that is symmetric positive semidefinite for almost all $t \in I$, 
 we ideally want to select a new point  $y_+ \in \Omega$ such that the solution error 
\begin{align} \label{eq:Refinement:ODE:error-measure-ODE}
	& \| \bx(\cdot \, ; \bg_+[y_+]) - \bx_*(\cdot) \|_\BQ^2 \nonumber \\
	& := \int_{t_0}^{t_f} \big(\bx(t ; \bg_+[y_+]) - \bx_*(t)\big)^T \BQ(t) \big(\bx(t ; \bg_+[y_+]) - \bx_*(t)\big) \, dt
\end{align}
is small.
We approximate
\begin{align} \label{eq:Refinement:ODE:approx-error-measure-ODE}
	& \| \bx(\cdot \, ; \bg_+[y_+]) - \bx(\cdot \, ; \bg_*) \|_\BQ^2  \nonumber \\
	& \approx \int_{t_0}^{t_f}
          [\bx_\bg(\bg_+[y_+])  (\bg_+[y_+] - \bg_*)](t)^T \BQ(t) \, [\bx_\bg(\bg_+[y_+]) (\bg_+[y_+] - \bg_*)](t) \, dt \nonumber \\
        & \approx \int_{t_0}^{t_f}
          [\bx_\bg(\bg_c)  (\bg_+[y_+] - \bg_*)](t)^T \BQ(t) \, [\bx_\bg(\bg_c)  (\bg_+[y_+] - \bg_*)](t) \, dt.
\end{align}

We cannot compute $\bg_+[y_+] - \bg_*$, but we have the post-refinement error bound  \eqref{eq:Refinement:intro:error-indicator}.
Thus, using  the ODE sensitivity result of \cref{thm:Sensitivity:ODE:sensitivity-result-ODE} and \eqref{eq:Refinement:intro:error-indicator},
the optimal objective function value of
\begin{subequations} \label{eq:Refinement:ODE:LQOCP-naive}
\begin{align}
	\sup_{\delta \bx, \delta \bg} \quad &\int_{t_0}^{t_f} \delta \bx(t)^T \BQ(t) \delta \bx(t) \, dt \\
	\mbox{s.t.} \quad & \delta \bx'(t) = \BA_c[t] \delta \bx(t) + \BB_c[t] \delta \bg\big( \by_c(t) \big),  & \aall t \in I, \\
	& \delta \bx(t_0) = 0, \\
	&  \big| \delta \bg_i\big( \by_c(t) \big) \big|  \leq c_i \bepsilon_i\big( \by_c(t); Y_+, G_c \big), 
	& i = 1, \dots, n_g, \quad \aall t \in I,
	    \label{eq:Refinement:ODE:LQOCP-naive-box}
\end{align}
\end{subequations}
gives an upper bound for \eqref{eq:Refinement:ODE:approx-error-measure-ODE}.
The term $\delta \bg\big( \by_c(t) \big)$ represents all possible errors in the surrogate along the current solution $\by_c(t)$ based on the post-refinement error bound \eqref{eq:Refinement:intro:error-indicator}.
Because $\delta \bg\big( \by_c(t) \big)$ is evaluated along the current solution $\by_c(t)$, 
 \eqref{eq:Refinement:ODE:LQOCP-naive} is difficult to implement. 
Therefore, we relax \eqref{eq:Refinement:ODE:LQOCP-naive}  by replacing the trajectory-dependent $\delta \bg\big( \by_c(\cdot) \big)$ by a 
time-dependent function $\bdelta \in \big(L^\infty(I)\big)^{n_g}$, and we drop the unknown constants $c_i$.
Thus, instead of  \eqref{eq:Refinement:ODE:LQOCP-naive}, we solve
\begin{subequations} \label{eq:Refinement:ODE:LQOCP}
\begin{align}
	\sup_{\delta \bx, \delta \bg} \quad &\int_{t_0}^{t_f} \delta \bx(t)^T \BQ(t) \delta \bx(t) \, dt \\
	\mbox{s.t.} \quad & \delta \bx'(t) = \BA_c[t] \delta \bx(t) + \BB_c[t] \bdelta(t),  \qquad \aall t \in I, \\
	& \delta \bx(t_0) = 0, \\
	& - \bepsilon\big( \by_c(t); Y_+, G_c \big) \leq \bdelta(t)\leq \bepsilon\big( \by_c(t); Y_+, G_c \big), & \qquad \aall t \in I. \label{eq:Refinement:ODE:box-constraints}
\end{align}
\end{subequations}
\sloppy In \eqref{eq:Refinement:ODE:LQOCP}, $\bepsilon\big( \by_c(t); Y_+, G_c \big) \in \real^{n_g}$ is the
vector-valued function with components $\bepsilon_i\big( \by_c(t); Y_+, G_c \big)$, $i = 1, \ldots, n_g$, and
the inequality constraints are understood componentwise.

\begin{theorem}   \label{thm:Refinement:ODE:ODE-solution-error-bound}
     If the assumptions of \cref{thm:Sensitivity:ODE:sensitivity-result-ODE} and  \eqref{eq:Refinement:intro:error-indicator} hold,
     and $V( \by_c; Y_+, G_c \big)$ denotes the optimal objective function value of \eqref{eq:Refinement:ODE:LQOCP},
     then the approximate error measure \eqref{eq:Refinement:ODE:approx-error-measure-ODE} is bounded above by
     \begin{equation}  \label{eq:Refinement:ODE:ODE-solution-error-bound}
         \int_{t_0}^{t_f}
          [\bx_\bg(\bg_c)  (\bg_+[y_+] - \bg_*)](t)^T \BQ(t) [\bx_\bg(\bg_c)  (\bg_+[y_+] - \bg_*)](t) \, dt 
         \le c^2 \, V( \by_c; Y_+, G_c \big),
     \end{equation}
     where  $c = \max\{ c_1, \ldots, c_{n_g}\}$ and $c_1, \ldots, c_{n_g}$ are the constants in  \eqref{eq:Refinement:intro:error-indicator}.
\end{theorem}
\begin{proof}
        \sloppy
	Clearly $\bdelta(\cdot) = c^{-1} (\bg[y_+] - \bg_*)\big( \by_c(\cdot) \big) \in \LL^{n_g}$ and its corresponding state 
	$\delta \bx = c^{-1}  \bx_\bg(\bg_c) (\bg[y_+] - \bg_*)$ are feasible for \eqref{eq:Refinement:ODE:LQOCP} due to \eqref{eq:Refinement:intro:error-indicator}. The bound immediately follows.
\end{proof}
\cref{thm:Refinement:ODE:ODE-solution-error-bound} is an adaptation of \cite[Thm.~3.3]{JRCangelosi_MHeinkenschloss_2025c} to the model refinement setting.

Because the constants $c_i$ are dropped from \eqref{eq:Refinement:intro:error-indicator} in the formulation of \eqref{eq:Refinement:ODE:box-constraints}, it is beneficial to have 
constants $c_1 \approx  \ldots \approx c_{n_g} \approx 1$,
 which in the kernel interpolation case is achieved when $\gamma_i  \approx \| (\bg_*)_i \|_{\cH_\bk(\Omega)}$,
 $i = 1, \ldots, n_g$; see \eqref{eq:Refinement:intro:error-indicator-RKHS}.

The problem \eqref{eq:Refinement:ODE:LQOCP} is parametrized by the refinement point $y_+ \in \Omega$, which affects the post-refinement error bound $\bepsilon\big( \by_c(t); Y_+, G_c \big)$ appearing in the box constraints for the control $\bdelta \in \big(L^\infty(I) \big)^{n_g}$.
The optimal objective function value of \eqref{thm:Refinement:ODE:ODE-solution-error-bound} is used as an acquisition function
to determine the next refinement point $y_+ \in \Omega$ to choose. 
%The constant $c$ in  \eqref{eq:Refinement:intro:error-indicator},
%which appears in the upper bound \eqref{eq:Refinement:ODE:ODE-solution-error-bound}, is irrelevant because it simply scales the
%acquisition function, but does not change its minimum $y_+ \in \Omega$.

The new refinement point $y_+ \in \Omega$ is selected to minimize the sensitivity-based upper bound 
\eqref{eq:Refinement:ODE:ODE-solution-error-bound} over a chosen set of candidates $Y_{\rm cand} \subset \Omega$, 
yielding the min-max problem
\begin{equation} \label{eq:Refinement:ODE:minmax}
	\min_{y_+ \in Y_{\rm cand}} \;  V( \by_c; Y_+, G_c \big),
\end{equation}
where, again, $V( \by_c; Y_+, G_c \big)$ denotes the optimal objective function value of \eqref{eq:Refinement:ODE:LQOCP}.

The model refinement method proceeds as follows. First, solve \eqref{eq:Refinement:ODE:minmax} to determine  $y_+ \in Y_{\rm cand}$, 
then compute $\bg_*(y_+)$ and obtain the refined surrogate $\bg(y; Y_+, G_+)$, then re-solve \eqref{eq:problem_formulation:IVP} 
with it to obtain a more accurate solution.

Note that \eqref{eq:Refinement:ODE:LQOCP} is a convex {\em max}imization problem.
A solution may not exist in the infinite-dimensional setting, and
such problems are generally NP-hard in discrete settings. See, e.g., \cite{HPBenson_1995a},
\cite{SBurer_ANLetchford_2009a}, \cite{RHorst_PMPardalos_NVThoai_2000a}.
However, due to the symmetry of the box constraints, the linearity of the dynamics, 
and the symmetry of the quadratic objective function,  one may still obtain a usable approximate upper 
bound in practice using tailored interior point methods with zero as the initial point. 
More details are given in \cite{JRCangelosi_MHeinkenschloss_2025c}.

%%%%%%%%%%%%%%%%%%%%%%%%%%%%%%%%%%%%%%%%%%%%%%%%%%%%%%%%%
\subsection{Optimization}  \label{sec:Refinement:OCP:OCP-sol}
An extension of the approach in \cref{sec:Refinement:ODE:ODE-sol} to refine models based on
improvements of the solution $\bx, \bu, \bp$ of the optimal control problem \eqref{eq:Sensitivity:OCP:OCP}
and of the correspnding co-state $\blambda$ is easily possible, at least in theory.
The corresponding acquisition function requires the solution of a linear-quadratic optimal control problem
in variables $\delta  \bx, \delta \bu, \delta \bp, \delta \blambda$ and $\bdelta$, $\bdelta^x$, $\bdelta^u$, $\bdelta^p$; however, this problem would be more difficult to solve than the already NP-hard problem \eqref{eq:Refinement:ODE:LQOCP}, so we do not discuss it further.

%%%%%%%%%%%%%%%%%%%%%%%%%%%%%%%%%%%%%%%%%%%%%%%%%%%%%%%%%
%\bibliographystyle{siamplain}
%\bibliography{references}

%%%%%%%%%%%%%%%%%%%%%%%%%%%%%%%%%%%%%%%%%%%%%%%%

\end{document}